\documentclass[11pt]{amsart}
\usepackage{amsmath,amsthm}
\usepackage{latexsym,amssymb}
\usepackage{amssymb}
\usepackage{amscd}
\usepackage{stmaryrd}
\usepackage{tikz}
\usepackage{tikz-cd}
\usepackage{url}
\usetikzlibrary{matrix}
\usetikzlibrary{cd}
\usepackage{siunitx}
\usepackage{mathrsfs}
\usepackage{amsrefs}
\DefineSimpleKey{bib}{myurl}
\allowdisplaybreaks
\newcommand\myurl[1]{\url{#1}}

\BibSpec{webpage}{%
  +{}{\PrintAuthors} {author}
  +{,}{ \textit} {title}
  +{}{ \parenthesize} {date}
  +{,}{ \myurl} {myurl}
  +{,}{ } {note}
  +{.}{ } {transition}
}
\theoremstyle{plain}
\title{Patching over Berkovich Curves and  Quadratic Forms}
\author{Vler\"e Mehmeti} \thanks{The author was supported by the ERC Starting Grant ``TOSSIBERG": 637027.}
\thanks{2010 Mathematics Subject Classification 14G22, 11E08}
\begin{document}
\maketitle
\newtheorem{thm}{Theorem}[section]
\newtheorem{lm}[thm]{Lemma}
\newtheorem{prop}[thm]{Proposition}
\newtheorem{cor}[thm]{Corollary}
\newtheorem*{cor*}{Corollary}
\newtheorem*{thm*}{Theorem}
\theoremstyle{definition}
\newtheorem{defn}[thm]{Definition}
\newtheorem*{defn*}{Definition}
\newtheorem{rem}[thm]{Remark}
\newtheorem{ex}[thm]{Example}
\newtheorem{set}{Setting}
\newtheorem{nota}{Notation}

\begin{abstract}
We extend field patching to the setting of Berkovich analytic geometry and use it to prove a local-global principle over function fields of analytic curves with respect to completions. In the context of quadratic forms, we combine it with sufficient conditions for local isotropy over a Berkovich curve to obtain applications on the $u$-invariant. The patching method we adapt was introduced by Harbater and Hartmann in~\cite{HH}, and further developed by these two authors and Krashen in \cite{HHK}. The results presented in this paper generalize those of  \cite{HHK} on the local-global principle and quadratic forms. 
\end{abstract}

\renewcommand{\abstractname}{R\'esum\'e}

\begin{abstract}\textbf{Recollement sur les courbes de Berkovich et formes quadratiques.} 
Nous \'etendons la technique de recollement sur les 
corps au cadre de la g\'eom\'etrie analytique de 
Berkovich pour d\'emontrer un principe local-global 
sur les corps de fonctions de courbes analytiques par rapport \`a certains de leurs compl\'et\'es.
Dans le contexte des formes quadratiques, nous le 
combinons avec des conditions suffisantes 
d'isotropie locale sur une courbe de Berkovich pour 
obtenir des applications au $u$-invariant. La m\'ethode de recollement que nous adaptons a \'et\'e 
introduite par Harbater et Hartmann dans \cite{HH}, 
 puis developp\'ee par ces deux auteurs et 
Krashen dans \cite{HHK}. Dans ce texte, nous pr\'esentons des r\'esultats sur le principe local-global et les formes quadratiques qui g\'en\'eralisent ceux de \cite{HHK}.
\end{abstract}

\section*{Introduction} 

Patching techniques were introduced as one of the main approaches to inverse Galois theory. Originally of purely geometric nature, this method provided a way to obtain a global Galois cover from local ones, see for example \cite{H}. Another example is \cite{poi1}, where Poineau used patching on analytic curves in the Berkovich sense, and consequently generalized results shown by Harbater in \cite{har2} and \cite{har}. In \cite{HH}, Harbater and Hartmann extended the technique to structures over fields, while constructing a setup of heavily algebraic flavor. Patching over fields has recently seen many applications to local-global principles and quadratic forms, see for example \cite{HHK} and \cite{ctps}. In particular, in \cite{HHK}, Harbater, Hartmann, and Krashen (from now on referred to as HHK) obtained results on the \mbox{$u$-invariant}, generalizing those of Parimala and Suresh \cite{parsur}, which were proven through different methods. Another source for results on the $u$-invariant is Leep's article \cite{leep}.

In this paper, we use field patching in the setting of Berkovich analytic geometry. A convenience of this point of view is the clarity it provides into the overall strategy. By patching over analytic curves, we prove a local-global principle and provide applications to quadratic forms and the $u$-invariant.  The results we obtain generalize those of \cite{HHK}. Because of the  geometric nature of this approach, we believe it to be a nice framework for potential generalizations in different directions, and in particular to higher dimensions. 

Before presenting the main results of this paper, let us introduce some terminology.
\begin{defn*} [HHK] Let $K$ be a field. Let $X$ be a $K$-variety, and $G$ a linear algebraic group over $K.$ We say that $G$ acts \emph{strongly transitively} on $X$ if $G$ acts on $X,$ and for any field extension $L/K,$ either $X(L)=\emptyset$ or $G(L)$ acts transitively on $X(L).$
\end{defn*}

Our main results, the local-global principles we show, are: 
\begin{thm*}
\begin{sloppypar}
Let $k$ be a complete non-trivially valued ultrametric field. Let $C$ be a normal irreducible projective $k$-algebraic curve. Denote by ~$F$ the function field of $C.$ Let $X$ be an~$F$-variety, and $G$ a connected rational linear algebraic group over $F$ acting strongly transitively on $X.$ 

Let $V(F)$ be the set of all non-trivial rank 1 valuations on ~$F$ which either extend the valuation of $k$ or are trivial when restricted to $k.$

Denote by $C^{{\mathrm{an}}}$ the Berkovich analytification of $C,$ so that $F=\mathscr{M}(C^{{\mathrm{an}}}),$ where $\mathscr{M}$ denotes the sheaf of meromorphic functions on $C^{\mathrm{an}}.$ 
Then, the following local-global principles hold:
\end{sloppypar}
\begin{itemize}
\item(Theorem \ref{ohlala}) $X(F)\neq \emptyset \iff X(\mathscr{M}_{x}) \neq \emptyset  \ \textrm{for all}\ x \in C^{\mathrm{an}}.$
\item(Corollary \ref{uhlala}) If $F$ is a perfect field or $X$ is a smooth variety, then: 
\begin{center}
$X(F)\neq \emptyset \iff X(F_v) \neq \emptyset$  for all $v \in V(F),$ 
\end{center}
where $F_v$ denotes the completion of $F$ with respect to ~$v.$
\end{itemize}
\end{thm*}
The statement above remains true for affinoid curves if $\sqrt{|k^\times|} \neq \mathbb{R}_{>0}$, where $\sqrt{|k^{\times}|}$ denotes the divisible closure of the value group $|k^{\times}|$ of $k$. Being a local-global principle with respect to completions, the second equivalence evokes some resemblance to more classical versions of local-global principles. The statement can be made to include trivially valued base fields, even though in this case we obtain no new information (since one of the overfields will be equal to $F$). 

We recall that for any finitely generated field extension $F/k$ of transcendence degree 1, there exists a unique normal projective $k$-algebraic curve with function field~$F.$ Thus, the result of the theorem above is applicable to any such field $F.$ 

While HHK work over models of an algebraic curve, we work directly over analytic curves. 
Remark that we put no restrictions on the complete valued base field $k.$ Apart from the framework, this is one of the fundamental differences with Theorem 3.7 of \cite{HHK}, where the base field needs to be complete with respect to a discrete valuation. Another difference lies in the nature of the overfields, which here are completions or fields of  meromorphic functions. Section 4 shows that the latter contain the ones appearing in HHK's  article, and thus that \cite[Theorem 3.7]{HHK} is a direct consequence of the local-global principle stated in Theorem \ref{ohlala}. Moreover, we show the converse is true as well provided we choose a ``fine" enough model. The proof of the theorem above is based on the patching method, but used in a different setting from the one of \cite{HHK}. 

As a consequence, in the context of quadratic forms we obtain the following theorem, which is a generalization of \cite[Theorem 4.2]{HHK}.
 
\begin{thm*}
\begin{sloppypar}
Let $k$ be a complete non-trivially valued ultrametric field. 
Let $C$ be a normal irreducible projective $k$-algebraic curve. Denote by $F$ the function field of $C.$ Suppose $char(F) \neq 2.$ Let $q$ be a quadratic form over $F$ of dimension different from $2$. 

Let $V(F)$ be the set of all non-trivial rank 1 valuations on ~$F$ which either extend the valuation of $k$ or are trivial when restricted to $k.$

Let $C^{{\mathrm{an}}}$ be the Berkovich analytification of $C,$ so that $F=\mathscr{M}(C^{{\mathrm{an}}}),$ where $\mathscr{M}$ is the sheaf of meromorphic functions on $C^{{\mathrm{an}}}.$
\begin{enumerate}
\item{(Theorem \ref{17}) The quadratic form $q$ is isotropic over $F$ if and only if it is isotropic over $\mathscr{M}_x$ for all $x \in C^{{\mathrm{an}}}.$}
\item{(Corollary \ref{quad}) The quadratic form $q$ is isotropic over $F$ if and only if it is isotropic over $F_v$ for all $v \in V(F),$ where $F_v$ is the completion of $F$ with respect to $v.$}
\end{enumerate} 
\end{sloppypar}
\end{thm*}

As mentioned in the introduction of \cite{HHK}, it is expected that for a ``nice enough" field $K$ the $u$-invariant remains the same after taking finite field extensions, and that it becomes $2^d u(K)$ after taking a finitely generated field extension of transcendence degree $d.$ Since we work only in dimension one, this explains the motivation behind the following: 

\begin{defn*}
Let $K$ be a field. 
\begin{enumerate}
\begin{sloppypar}
\item {[Kaplansky] The $u$-\textit{invariant} of $K,$ denoted by $u(K),$ is the maximal dimension of anisotropic quadratic forms over $K.$ We say that $u(K)=\infty$ if there exist anisotropic quadratic forms over $K$ of arbitrarily large dimension.}
\item {[HHK] The strong $u$-invariant of $K,$ denoted by $u_s(K),$ is the smallest real number~$m$ such that:
\begin{itemize}
\item $u(E) \leqslant m$ for all finite field extensions $E/K;$ 
\item $\frac{1}{2}u(E) \leqslant m$ for all finitely generated field extensions $E/K$ of transcendence degree 1.  
\end{itemize}
We say that $u_s(K)=\infty$ if there exist such field extensions $E$ of arbitrarily large~\mbox{$u$-invariant.}}
\end{sloppypar}
\end{enumerate}
\end{defn*}

The theorem above leads to applications on the $u$-invariant. Let $k$ be a complete non-archimedean valued field with residue field $\widetilde{k}$, such that $\mathrm{char}(\widetilde{k}) \neq 2.$ Suppose that either ${|k^{\times}|}$ is a free $\mathbb{Z}$-module with $\mathrm{rank}_{\mathbb{Z}}{|k^{\times}|}=n,$ or more generally that $\dim_{\mathbb{Q}} \sqrt{|k^{\times}|}=n,$ where $n$ is a non-negative integer.
This is yet another difference with the corresponding results of HHK in \cite{HHK}, where the requirement on the base field is that it be complete discretely valued, \textit{i.e.} that its value group be a free $\mathbb{Z}$-module of rank $1.$ We obtain an upper bound on the $u$-invariant of a finitely generated field extension of $k$ with transcendence degree at most $1$, which depends only on $u_s(\widetilde{k})$ and $n.$ More precisely, in terms of the strong $u$-invariant:

\begin{cor*}[Corollary \ref{31}] Let $k$ be a complete valued non-archimedean field.
Suppose $char(\widetilde{k}) \neq 2.$ Let $n \in \mathbb{N}.$
\begin{enumerate}
\item If $\dim_{\mathbb{Q}}\sqrt{|k^{\times}|}=n,$ then $u_s(k) \leqslant 2^{n+1}u_s(\widetilde{k}).$
\item If $|k^{\times}|$ is a free $\mathbb{Z}$-module with $\mathrm{rank}_{\mathbb{Z}}|k^{\times}|=n$, then $u_s(k) \leqslant 2^{n}u_s(\widetilde{k}).$
\end{enumerate}
\end{cor*}

It is unknown to the author whether there is equality in the corollary above. This is true in the particular case of $n=1$ by using \cite[Lemma 4.9]{HHK}, whose proof is independent of patching. This way we recover \cite[Theorem 4.10]{HHK}, which is the main result of \cite{HHK} on quadratic forms. It also provides one more proof that $u(\mathbb{Q}_p(T))=8,$ where $p$ is a prime number different from $2$, originally proven in \cite{parsur}.

\begin{cor*}[Corollary \ref{32}]
Let $k$ be a complete discretely valued field such that $char(\widetilde{k}) \neq~2.$ Then, $u_s(k)=2u_s(\widetilde{k})$.
\end{cor*}
The first section of this paper is devoted to proving that patching can be applied to an analytic curve. To do this, we 
follow along the lines of the proof of
 \mbox{\cite[Theorem 2.5]{HHK}},  making adjustments to render it suitable to our more general setup.
Recall that an analytic curve is a graph (see \cite[Th\'eor\`eme 3.5.1]{Duc}). We work over any complete valued base field ~$k$ such that $\sqrt{|k^\times|} \neq \mathbb{R}_{>0}.$ This condition is equivalent to asking the existence of type 3 points on $k$-analytic curves, which are characterized by simple topological and algebraic properties. More precisely, a point of type 3 has arity 2 in the graph associated to the curve, and its local ring with respect to the sheaf of analytic functions is a field if the curve is reduced. Type 3 points are crucial to the constructions we make.
Let $C$ be an integral $k$-analytic curve. Let $U,V$ be connected affinoid domains in $C$, such that $W=U \cap V$ is a single type 3 point. We show that given two ``good" algebraic structures over $\mathscr{M}(U),$ $\mathscr{M}(V),$ and a suitable group action on them, they can be patched to give the same type of algebraic structure over $\mathscr{M}(U \cup V).$ The key step in proving this is a ``matrix decomposition" result for certain linear algebraic groups.

In the second part, our aim is to show that any open cover of a projective $k$-analytic curve can be refined into a finite cover that satisfies conditions similar to those of the first section, \textit{i.e.} one over which we can apply patching. The refinement $\mathcal{U}$ we construct is a finite cover of the curve, such that for any $U \in \mathcal{U},$ $U$ is a connected affinoid domain with only type 3 points in its boundary. Furthermore, for any distinct $U,V \in \mathcal{U},$ the intersection~${U \cap V}$ is a finite set of type 3 points. A cover with these properties will be called \textit{nice} (\textit{cf.} Definition \ref{nice}). The existence of a refinement that is a nice cover will first be shown for the projective line~$\mathbb{P}_k^{1,\mathrm{an}},$ and will then be generalized to a broader class of $k$-analytic curves. We work over a complete ultrametric field $k$ such that~${\sqrt{|k^\times|}\neq \mathbb{R}_{>0}}.$

The third section contains the main results of the paper. We show two
local-global principles (Theorem \ref{ohlala}, Corollary \ref{uhlala}) over fields of meromorphic functions of normal projective $k$-analytic curves, and an application to quadratic forms (Theorem~\ref{17}, Corollary \ref{quad}). In the simplest cases, the proofs use patching on nice covers and induction on the number of elements of said covers. We first prove these results over a complete ultrametric base field $k$ such that $\sqrt{|k^\times|} \neq \mathbb{R}_{>0}.$ This is then generalized for projective curves to any complete ultrametric field using a descent argument that is based on results of model theory. We also prove similar results for affinoid curves.   

In the fourth section we interpret the overfields of HHK's \cite{HHK} in the Berkovich setting, and show that \cite[Theorem 3.7]{HHK} is a consequence of Theorem \ref{ohlala}. We show that the converse is true as well provided one works over a ``fine enough" model. 

The purpose of the fifth part is to find  conditions under which there is local isotropy of a quadratic form $q$ over analytic curves. The setup will be somewhat more general, which is partly why it is the most technical section of the paper. The idea  is to find a nice enough representative of the isometry class of $q$ to work with and then use Henselianity conditions. The hypotheses on the base field become stronger here. Namely, we require our complete valued non-archimedean base field $k$ to be such that  the dimension of the ~\mbox{$\mathbb{Q}$-vector space}~$\sqrt{|k^{\times}|}$ be finite (a special case beeing when $|k^{\times}|$ is a free module of finite rank over $\mathbb{Z}$), and the residue characteristic unequal to $2.$ The restriction on the value group is not very strong: when working over a complete ultrametric field $k$ satisfying this property, for every $k$-analytic space $X$ and every point $x \in X,$ the completed residue field $\mathcal{H}(x)$ of $x$ satisfies it as well.  

In the last part, we put together the local-global principle for quadratic forms and the local isotropy conditions of the previous section to give a condition for global isotropy of a quadratic form over an analytic curve. From there we deduce applications to the (strong)~$u$-invariant of a complete valued field $k$ with residue characteristic different from~2, and such that the dimension of the $\mathbb{Q}$-vector space $\sqrt{|k^{\times}|}$ is finite. 
 
\vspace{0.3cm}
\noindent \textbf{Conventions.} Throughout this paper, we use the Berkovich approach to non-archimedean analytic geometry. A Berkovich analytic curve will be a separated analytic space of pure dimension $1.$

 A valued field is a field endowed with a non-archimedean absolute value. For any valued field $l,$ we denote by $\widetilde{l}$ its residue field. 

We call \textit{boundary}, and denote it by $\partial(\cdot),$ the topological boundary. We call \textit{Berkovich relative boundary}  (resp. \textit{Berkovich boundary}), and denote it by $\partial_B(\cdot/\cdot)$ (resp. $\partial_B(\cdot)$), the relative boundary (resp. boundary) introduced in \cite[Definition ~2.5.7]{Ber90} and \cite[Definition ~1.5.4]{ber93}. 

A Berkovich analytic space which is reduced and irreducible is called \textit{integral} in this text. Thus, an \textit{integral affinoid space} is an affinoid space whose corresponding affinoid algebra is a domain.

Throughout the entire paper, we work over a complete valued base field $k.$

\subsection*{Acknowledgements} I am most grateful to J\'er\^ome Poineau for the numerous fruitful discussions and helpful remarks. Many thanks to Philippe Satg\'e for the chaotic yet beneficial mathematical exchanges we had. My gratitude also goes to Antoine Ducros for pointing out an argument that made it possible to get rid of the hypothesis $\sqrt{|k^\times|} \neq \mathbb{R}_{>0}.$ Finally, many thanks to the referee for a very careful reading, leading to a great improvement of this paper.

\section{Patching over Berkovich Curves}
The purpose of this section is to prove a matrix decomposition result under conditions which generalize those of HHK's article \cite[Section 3, Theorem 3.2]{HHK}. As a consequence, we obtain a generalization of vector space patching on analytic curves. Let us start by fixing a somewhat more extensive framework, in which our proof works. 

\begin{set}
Let $R_i, i=1,2,$ be an integral domain endowed with a non-archimedean sub-multiplicative norm $|\cdot|_{R_i},$ with respect to which it is complete. Set $F_i=\text{Frac} \ R_i, i=1,2.$ Let $F$ be an infinite field embedded in both $F_1$ and $F_2.$
Let $F_0$ be a complete ultrametric field with non-trivial valuation, such that there exist bounded morphisms~${R_i \hookrightarrow F_0},$  $i=1,2.$ Suppose the image of~$F_1$ is dense in~$F_0.$ Let~$A_i $ be an~$R_i$-module, such that~$A_i \subseteq F_i.$ Suppose~$A_i$ is finitely generated as an~$R_i$-module, \textit{i.e.} that there exists a surjective~$R_i$-linear morphism~${\varphi_i:R_i^{n_i} \twoheadrightarrow A_i}$ for some positive integer~$n_i, i=1,2.$ Let us endow~$A_i$ with the quotient semi-norm induced from~$\varphi_i.$ Assume that~$A_i$ is complete and the morphism~$A_i \hookrightarrow F_0$ is bounded  for~$i=1,2.$ Remark that this implies that the semi-norm on $A_i$ is a norm. Suppose the induced map~$\pi: A_1 \oplus A_2 \rightarrow F_0$ is surjective. Finally, suppose the norm of~$F_0$ is equivalent to the quotient norm induced by the surjective morphism~$\pi : A_1 \oplus A_2 \twoheadrightarrow F_0,$ where~$A_1 \oplus A_2$ is endowed with the usual max semi-norm~$|\cdot|_{\max}.$ 
\end{set}

As in \cite{Ber90}, a morphism $f : A \rightarrow B$ of semi-normed rings is said to be \textit{admissible} if the quotient semi-norm on $A/\ker(f)$ is equivalent to the restriction to~$f(A)$ of the semi-norm on $B.$ Thus, in the setting above, we suppose that the morphism~$\pi$ is admissible.

Before giving the motivating example for Setting~1, we need to recall Berkovich meromorphic functions. 

\begin{defn}
Let $X$ be a good $k$-analytic space (in the sense of Berkovich). Let $\mathcal{S}$ be the presheaf of functions on $X,$ which associates to any analytic domain $U$ the set of analytic functions on $U$ whose restriction to any affinoid domain in it is not a zero-divisor.  
Let $\mathscr{M}_{-}$ be the presheaf on $X$ that associates to any analytic domain $U$ the ring $\mathcal{S}(U)^{-1}\mathcal{O}(U).$ The sheafification $\mathscr{M}$ of the presheaf $\mathscr{M}_{-}$ is said to be the sheaf of meromorphic functions on~$X.$
\end{defn}

We notice that for any $x \in X,$ $\mathscr{M}_x$ is the total ring of fractions of $\mathcal{O}_{X,x}.$ In particular, if $\mathcal{O}_{X,x}$ is a domain, then $\mathscr{M}_x=\text{Frac} \ \mathcal{O}_{X,x}.$ We make note of the following, well known, fact:
\begin{lm} \label{1.2}
Let $X$ be an integral $k$-affinoid space. Then, $\mathscr{M}(X)=\mathrm{Frac} \ \mathcal{O}(X).$  
\end{lm}
By replacing the fraction field of $\mathcal{O}(X)$ with its total ring of fractions, the statement remains true when removing the condition of integrality on $X.$
\begin{proof}
\begin{sloppypar}
Since $\mathcal{O}(X)$ is an integral domain,  $\text{Frac} \ \mathcal{O}(X) \subseteq \mathscr{M}(X)$ by the definition of~$\mathscr{M}.$ Let~$f \in \mathscr{M}(X).$ The sheaf  $f \mathcal{O} \cap \mathcal{O} \subseteq \mathscr{M}$ is  non-zero and coherent, so by Kiehl's Theorem, it has a non-zero global section $x.$
Then, there exists $y \in \mathcal{O}(X) \backslash \{0\},$ for which~${f=\frac{x}{y} \in \text{Frac} \ \mathcal{O}(X)}.$
\end{sloppypar}
\end{proof}
Another result that will be needed throughout this paper is the following:
\begin{lm} \label{1.3}
Let $C$ be a normal irreducible $k$-analytic  curve. Let $U, V$ be affinoid domains of $C$, such that $U \cap V=\{\eta\},$ where $\eta$ is a point of type 3. Then, the images of $\mathscr{M}(U)$ and $\mathscr{M}(V)$ in~$\mathscr{M}(\{\eta\})$ are dense.   
\end{lm}

\begin{proof}
Let us start by remarking that the set of poles of a meromorphic function is a divisor and as such consists of only rigid points. This implies that a meromorphic function cannot have a pole on any non-rigid point (including $\eta$), which is why it makes sense to evaluate it at $\eta$. 

That $\{\eta\}$ is an affinoid domain of $U$ (resp. $V$) can be checked directly from the definition of an affinoid domain. By the Gerritzen-Grauert theorem (see \cite{tem}), we obtain that it is a rational domain in $U$ (resp. $V$). Then, by the easy implication of 
Corollary~2.2.10 in~\cite{Ber90}, the meromorphic functions on $U$ (resp. $V$) with no poles in $\{\eta\}$ are dense in $\mathcal{O}(\{\eta\}).$ Seeing as $\eta$ is a type 3 point, $\mathcal{O}(\{\eta\})=\mathscr{M}(\{\eta\})=\mathcal{H}(\eta)$ - the completed residue field of~$\eta.$ Finally, this implies that the image with respect to the restriction morphism of $\mathscr{M}(U)$ (resp. $\mathscr{M}(V)$) in $\mathscr{M}(\{\eta\})$  is dense. 
\end{proof}
The example of Setting 1 we will be working with is the following: 
\begin{prop}\label{1}
Let $C$ be a normal irreducible $k$-analytic curve. Set $F_C=\mathscr{M}(C).$ Let~$D$ be an effective divisor of degree $n$ on $C.$ Take two connected affinoid domains $U, V$ in $C,$ such that $W:=U \cap V=:\{\eta\}$, where $\eta$ is a type 3 point.  Set $R_{U}=\mathcal{O}(U), F_U=\mathrm{Frac} \ R_U,$ ${R_V=\mathcal{O}(V)}, {F_V=\mathrm{Frac} \ R_V},$ and $F_W=\mathcal{O}(W).$ Set $A_U=\mathcal{O}(D)(U),$ $A_V=\mathcal{O}(D)(V).$

For large enough $n$ such that $H^1(C, \mathcal{O}(D))=0$, the conditions of Setting 1 are satisfied with $R_1=R_U, R_2= R_V,$ $A_1=A_U, A_2=A_V,$ $F=F_C,$ and $F_0=F_W.$  
\end{prop}

\begin{proof}
\begin{sloppypar}
As $U,V,W$ are connected affinoid domains of a normal analytic curve, they are integral, so $R_U, R_V, \mathcal{O}(W)$ are integral domains that are all complete with respect to non-Archimedean norms. As $W$ is a single type 3 point, $\mathcal{O}(W)$ is a field and so $\mathcal{O}(W)=F_W.$ Moreover, since $\mathcal{O}(W)=\mathcal{H}(\eta),$ the normed ring $F_W$ is a complete ultrametric non-trivially valued field. As $U,V$ and $W$ are integral, by Lemma~\ref{1.2},~$\mathscr{M}(U)=F_U$,~$\mathscr{M}(V)=F_V,$ and $\mathscr{M}(W)=F_W.$ This shows the existence of embeddings of $F_C$ into $F_U, F_V,$ and $F_W.$ The restriction morphisms $R_U, R_V \rightarrow F_W$ are bounded by construction. From Lemma \ref{1.3}, $F_U, F_V$ have dense images in $F_W.$
\end{sloppypar}
\begin{sloppypar}
Notice that for $Z \in \{U,V,W\},$ $\mathcal{O}(Z) \hookrightarrow \mathcal{O}(D)(Z) \hookrightarrow \mathscr{M}(Z).$ In particular, this means that $\mathcal{O}(D)(W)=\mathcal{O}(W)=\mathscr{M}(W).$ Since $\mathcal{O}(D)$ is a coherent sheaf, $A_U$  (resp.~$A_V$) is a finite  $R_U$-module (resp. $R_V$-module). The completness of~$A_U$ (resp. $A_V$) follows from the fact that ideals of affinoid algebras are closed. The morphism~${\mathcal{O}(D)(U)=A_U \hookrightarrow F_W=\mathcal{O}(D)(W)}$ is the restriction morphism of the sheaf $\mathcal{O}(D),$ so it is bounded. The same is true for $A_V \hookrightarrow F_W.$ 
\end{sloppypar}
If $U \cup V$ is not the entire $C,$ it is an affinoid domain thereof (see \cite[Th\'eor\`eme 6.1.3]{Duc}). By Tate's Acyclicity Theorem \cite[Chapter 2, Proposition 2.2.5]{Ber90}, 
$$0 \rightarrow H^0(U \cup V, \mathcal{O}(D)) \rightarrow H^0(U, \mathcal{O}(D)) \oplus H^0(V, \mathcal{O}(D)) \rightarrow H^0(U \cap V, \mathcal{O}(D)) \rightarrow 0$$
is an exact admissible sequence, from which we obtain the surjective admissible morphism~$A_U \oplus A_V \twoheadrightarrow \mathcal{O}(D)(W)=F_W.$ 

Suppose $U \cup V = C.$ Since $C$ is then compact and integral, by \cite[Th\'eor\`eme 6.1.3]{Duc}, it is either an affinoid domain (a case we dealt with in the paragraph above) or a projective curve.
If $C$ is projective, by \cite[Section~7.5, Proposition~5.5]{liu} for large enough ~$n,$  $H^1(U \cup V, \mathcal{O}(D))=0.$ The Mayer-Vietoris exact sequence now produces a bounded surjective morphism~${A_U \oplus A_V \twoheadrightarrow \mathcal{O}(D)(W)=F_W}.$ 
Admissibility follows from Banach's Open Mapping Theorem if $k$ is not trivially valued (for a proof see \cite{bou}), and by a change of basis followed by the Open Mapping Theorem if it is (see \cite[Chapter 2, Proposition ~2.1.2(ii)]{Ber90}).
\end{proof}

We make note of the fact that Proposition \ref{1} assumes the existence of a point of type~3, which is equivalent to $\sqrt{|k^{\times}|}\neq \mathbb{R}_{>0}.$

\begin{rem}
Other examples of  Setting 1 can be obtained by taking instead of $\mathcal{O}(D)$ any coherent sheaf $\mathcal{F}$ of $\mathcal{O}$-algebras that is a subsheaf of $\mathscr{M},$ for which $H^1(C, \mathcal{F})=0.$ 
\end{rem}

\begin{defn}
Let $K$ be a field. A \emph{rational} variety over $K$ is a $K$-variety that has a Zariski open isomorphic to an open of some $\mathbb{A}_K^n.$ 
\end{defn}

Using the same notation as in Setting 1, the main goal of this section is to prove the following matrix decomposition result:

\begin{thm}\label{2}
\begin{sloppypar}
Let $G$ be a connected linear algebraic group over $F$ that is a rational variety over $F$. For any~${g \in G(F_0)},$ there exist $g_1 \in G(F_1),$ $g_2 \in G(A_2),$ such that $g=g_1 \cdot g_2.$
\end{sloppypar}
\end{thm}

This was proven in a slightly different setting by HHK in \cite{HHK}. We follow along the lines of their proof, making adjustements to render it suitable for the hypotheses we want to work with.

Let $K$ be an infinite field. Since a connected rational linear algebraic group $G$ over some infinite field $K$ has a non-empty open subset $U'$ isomorphic to an open subset $U$ of an affine space $\mathbb{A}_K^n$, by translation (since $K$ is infinite) we may assume that the identity element of $G$ is contained in $U',$ that $0 \in U,$ and that the identity is sent to $0.$ Let us denote the isomorphism $U' \rightarrow U$ by $\phi.$

Let $m$ be the multiplication in $G,$ and set $\widetilde{U'}=m^{-1}(U') \cap (U' \times U'),$ which is an open subset of $G \times G.$ It is isomorphic to an open subset $\widetilde{U}$ of $\mathbb{A}_K^{2n},$ and $m_{|\widetilde{U'}}$ gives rise to a map $\widetilde{U} \rightarrow U,$ \textit{i.e.} to a rational function $f: \mathbb{A}_K^{2n}\dashrightarrow \mathbb{A}_K^n.$ Remark that for any $(x,0), (0,x) \in \widetilde{U},$  this function sends them both to $x.$

\begin{center}
\begin{tikzpicture}
  \matrix (m) [matrix of math nodes,row sep=3em,column sep=4em,minimum width=2em]
  {
     \widetilde{U'} & U' \\
     \widetilde{U} & U \\};
  \path[-stealth]
    (m-1-1) edge node [left] {$(\phi \times \phi)_{|\widetilde{U'}}$} 
    (m-2-1) edge  node [above] {$m_{|\widetilde{U'}}$} 
    (m-1-2)
    (m-2-1.east|-m-2-2) edge node [below] {$f$}
      (m-2-2)
    (m-1-2) edge node [right] {$\phi$} 
    (m-2-2) 
    (m-2-1);
\end{tikzpicture}
\end{center}

The theorem we want to prove can be interpreted in terms of the map $f$. Lemma \ref{biggy} below, formulated to fit a more general setup, shows that said theorem is true on some neighborhood of the origin of an affine space. It is the analogue of \cite[Theorem 2.5]{HHK}.

\begin{sloppypar}We proceed first with an auxiliary result. Since the morphisms $A_i \hookrightarrow F_0, i=1,2,$ are bounded, there exists $C>0$, such that for any $x_i \in A_i, |x_i|_{F_0} \leqslant C \cdot |x_i|_{A_i}.$ By changing to an equivalent norm on $A_i$ if necessary, we may assume that $C=1.$ Let us fix the quotient norm $|\cdot|_{F_0}$ on $F_0,$ induced from the surjective morphism ~${\pi : A_1\oplus A_2 \twoheadrightarrow F_0}.$
\end{sloppypar}
\begin{lm} \label{1.7}
\begin{enumerate}
\item{For any $x_i \in A_i, i=1,2,$ $|x_i|_{F_0} \leqslant |x_i|_{A_i}.$}  
\item{
There exists a constant $d \in (0,1),$ such that for any $c \in F_0,$ there exist~${a \in 
A_1, b \in A_2},$ for which $\pi(a+b)=c$ and $d \cdot \max (|a|_{A_1},|b|_{A_2}) \leqslant |c|_{F_0}.$}
\end{enumerate}
\end{lm}
\begin{proof}
\begin{enumerate}
\begin{sloppypar}
\item{See the paragraph above the statement.}
\item{Suppose $c \neq 0$. Let $D$ be any real number, such that $D>1$. For any $c \in F_0,$ there exist ${a \in A_1, b \in A_2}$ (depending on $D$), such that $\pi(a+b)=c$ and ${\max(|a|_{A_1}, |b|_{A_2}) \leqslant D \cdot |c|_{F_0}}.$
Otherwise, for any $x \in A_1, y \in A_2,$ for which $\pi(x+y)=c,$ one would have~$|x+y|_{\max}=\max(|x|_{A_1}, |y|_{A_2}) > D \cdot |c|_{F_0}.$ Then, $$|c|_{F_0}=\inf_{\substack{x \in A_1, y \in A_2\\ \pi(x+y)=c}}|x+y|_{\max} \geqslant D \cdot |c|_{F_0},$$
which is impossible if $c \neq 0.$ Thus, there exist $a$ and $b$ as above, and for~${d=D^{-1} \in (0,1)},$ one obtains $d \cdot \max (|a|_{A_1},|b|_{A_2}) \leqslant |c|_{F_0}.$
}
\end{sloppypar}
If $c=0,$ the statement is true regardless of the choice of $d.$
\end{enumerate}
\end{proof}

From now on, instead of writing $\pi(x+y)=c$ for $x\in A_1, y\in A_2, c \in F_0,$ we will just put $x+y=c$ without risk of ambiguity.

In what follows, for any positive integer $n,$ let us endow $F_0^n$ with the max norm induced from the norm on $F_0,$ and let us also denote it by $|\cdot|_{F_0}.$ 

\begin{lm}\label{biggy}
Let $f: \mathbb{A}_{F_0}^n \times \mathbb{A}^n_{F_0} \dashrightarrow \mathbb{A}_{F_0}^n$ be a rational map defined on a Zariski open $\widetilde{S}$ such that $(0,0) \in \widetilde{S},$ and $f(x,0)=f(0,x)=x$ whenever $(x,0),(0,x) \in \widetilde{S}.$ Then, there exists $\varepsilon >0,$ such that for any $a \in \mathbb{A}^n(F_0)$ with $|a|_{F_0} \leqslant \varepsilon,$ there exist $u \in A_1^n$ and $v \in A_2^n,$ for which $(u,v) \in \widetilde{S}(F_0)$ and $f(u,v)=a.$
\end{lm}

\begin{proof}
The rational function $f$ can be written as $(f_1, \dots, f_n),$ where the $f_i$ are elements of $F_0[T_1, \dots, T_n, S_1, \dots, S_n]_{(T_1,\dots, T_n, S_1, \dots, S_n)}.$ Furthermore, since $f_i(0,0)=0,$ they belong to the maximal ideal of this ring. Lemmas 2.1 and 2.3 of \cite{HHK} remain true in our setting without any significant changes to their proofs (this is where the condition $f(x,0)=f(0,x)=x$ is crucial). They tell us that: 
\begin{enumerate}
\item{we can see these rational functions as elements of $F_0[[T_1, \dots, T_n, S_1, \dots, S_n]];$}
\item{there exists $M \geqslant 1,$ such that 
$$f_i=S_i+T_i+ \sum_{|(l,m)|\geqslant 2}c_{l,m}^i \underline{T}^l \underline{S}^m \in F_0[[T_1, \dots, T_n, S_1, \dots, S_n]],$$ with $|c_{l,m}^i|_{F_0} \leqslant M^{|(l,m)|},$ for  $i=1,2,\dots, n$ and $(l,m) \in \mathbb{N}^{2n},$ where $|(l,m)|$ is the sum of the coordinates of $(l,m).$ (Remark that since $f_i(x,0)=f_i(0,x)=x$ for any $x$ for which $(0,x),(x,0) \in \widetilde{S}$, we can even assume that $l,m$ are both non-zero.)}
\end{enumerate}
Since $\widetilde{S}$ is open, $\widetilde{S}(F_0)$ is a Zariski open in $\mathbb{A}^{2n}(F_0),$ and so it is open in $F_0^{2n}$ in the topology induced by the max norm (which is finer than the Zariski one). Seeing as $0 \in \widetilde{S}(F_0),$ there exists $\delta>0,$ such that for any $(x,y) \in F_0^{2n}$ with $|(x,y)|_{F_0}< \delta,$ one has $(x,y) \in \widetilde{S}(F_0)$ and~ $f(x,y)$ is defined. 

Let us fix the constant $d$ given by Lemma \ref{1.7}. Let $0 <\varepsilon' \leqslant \min \{1/2M, d^2/M^4, \delta/2\}.$ Set $\varepsilon=d\varepsilon'.$ Since ${\varepsilon< \varepsilon' < \min(1/M, \delta/2),}$ by Lemma 2.1 in \cite{HHK}, for any $(x,y) \in \widetilde{S}(F_0)$ with $|(x,y)|_{F_0} \leqslant \varepsilon',$ $f(x,y)$ is well defined, and the series by which $f_i(x,y)$ is given is convergent in $F_0,$ ${i=1,2,\dots, n}.$

Let $a =(a_1, a_2, \dots, a_n) \in \mathbb{A}^n(F_0)$ be such that $|a|_{F_0} \leqslant \varepsilon.$ Let $u_0=0 \in A_1^n,$ and $v_0=0 \in A_2^n.$ Using induction, one constructs sequences $(u_s)_s$ in $A_1^n,$ and $(v_s)_s$ in $A_2^n,$ such that the following conditions are satisfied: 
\begin{enumerate}
\item{$|u_s|_{A_1}, |v_s|_{A_2} \leqslant \varepsilon'$ for all $s \geqslant 0;$}
\item{$|u_s-u_{s-1}|_{A_1}, |v_s-v_{s-1}|_{A_2} \leqslant \varepsilon'^{\frac{s+1}{2}} $ for all $s \geqslant 1$;}
\item{$|f(u_s,v_s)-a|_{F_0} \leqslant d\varepsilon'^{\frac{s+2}{2}} $ for all $s \geqslant 0.$}
\end{enumerate} 

The first terms $u_0$ and $v_0$ satisfy conditions 1 and 3. We notice that the first condition implies $|(u_s,v_s)|_{F_0} \leqslant \varepsilon',$ so $f(u_s,v_s)$ is well defined, and $f_i(u_s,v_s)$ is convergent for $s \in \mathbb{N}$ and $i=1,2,\dots, n.$ Suppose that for $j\geqslant 0,$ we have constructed $u_j$ and $v_j$ satisfying all three conditions above. Then, $d_j:=a-f(u_j,v_j)\in F_0^n$ is well defined, and ${|d_j|_{F_0} \leqslant d\varepsilon'^{\frac{j+2}{2}}}.$ From Lemma \ref{1.7}, there  exist $u_{j}' \in A_1^n$ and $v_j' \in A_2^n,$ such that $d_j=u_j'+v_j',$ and  ${d \cdot \max(|u_j'|_{A_1}, |v_j'|_{A_2}) \leqslant |d_j|_{F_0} \leqslant d \varepsilon'^{\frac{j+2}{2}}}.$ 

Set $u_{j+1}=u_j+u_j'$ and $v_{j+1}=v_j+v_j'.$ Then, $|u_{j+1}|_{A_1} \leqslant \max \left({\varepsilon'}, \varepsilon'^{\frac{j+2}{2}} \right)=\varepsilon',$ and the same is true for $v_{j+1}.$  
Also, $|u_{j+1}-u_j|_{A_1}=|u_j'|_{A_1} \leqslant \varepsilon'^{\frac{j+2}{2}},$ and similarly, ${|v_{j+1}-v_j|_{A_2} \leqslant \varepsilon'^{\frac{j+2}{2}}}.$ 

\begin{sloppypar}
For $r \in \mathbb{N},$ $i \in \{1,2,\dots, r\}$ and $\alpha  \in F_0^r,$ let $\alpha_i$ be the $i$-th coordinate of $\alpha$. For $p=(p_1, p_2, \dots, p_r) \in \mathbb{N}^{r},$ set  $\alpha^p:=\prod_{i=1}^r \alpha_i^{p_i}.$ Then, for the third condition, 
\begin{align*}
|f_i(u_{j+1}, v_{j+1})-a_i|_{F_0}
= & \left|u_{j+1,i}+v_{j+1,i} -a_i +\sum_{|(l,m)|\geqslant 2} c_{l,m}^{i}u_{j+1}^l v_{j+1}^m\right|_{F_0}\\
= & \left| u_{j,i}+v_{j,i} + u'_{j,i}
+v'_{j,i}-a_i+ \sum_{|(l,m)|\geqslant 2} 
c_{l,m}^{i}u_{j+1}^l v_{j+1}^m  \right|
_{F_0} \\
= & \left|f_i(u_j,v_j) -a_i +u'_{j,i} +v'_{j,i}+ \sum_{|(l,m)|\geqslant 2} 
c_{l,m}^{i}(u_{j+1}^l v_{j+1}^m-u_{j}^l v_{j}^m)  \right|_{F_0}\\
= & \left| -d_{j,i}+ u'_{j,i} +v'_{j,i}+\sum_{|(l,m)|\geqslant 2} c_{l,m}^{i}(u_{j+1}^l v_{j+1}^m-u_j^lv_j^m)  \right|_{F_0}\\
= & \left| \sum_{|(l,m)|\geqslant 2} c_{l,m}^{i}(u_{j+1}^l v_{j+1}^m-u_j^lv_j^m)  \right|_{F_0} \\
\leqslant & \max_{|(l,m)|\geqslant 2} |c_{l,m}^{i}|_{F_0} \cdot |u_{j+1}^l v_{j+1}^m-u_j^lv_j^m|_{F_0}.\\
\end{align*}
On the other hand,
\begin{align*}
u_{j+1}^lv_{j+1}^m-u_j^lv_j^m= & (u_j+u_j')^l(v_j+v_j')^m-u_j^lv_j^m\\
= & \sum_{\substack{0 \leqslant \beta \leqslant l\\ 0 \leqslant \gamma \leqslant m}} A_\beta B_{\gamma} u_j^\beta u_j'^{l-\beta} v_j^\gamma v_j'^{m-\gamma}-u_j^lv_j^m\\
=&\sum_{0\leqslant\alpha \leqslant (l,m)}\sum_{\substack{\beta+\gamma=\alpha \\ 0 \leqslant \beta \leqslant l\\ 0 \leqslant \gamma \leqslant m}}A_\beta B_{\gamma} u_j^\beta u_j'^{l-\beta} v_j^\gamma v_j'^{m-\gamma}-u_j^lv_j^m,\\
\end{align*}  where $A_{\beta}, B_{\gamma}$ are integers (implying they are of norm at most one on $F_0$). 
Thus, ${u_{j+1}^lv_{j+1}^m-u_j^mv_j^l}$ is a finite sum of monomials of degree $|(l,m)|$ in the variables $u_{j,i}, u_{j,i}', v_{j,i}, v_{j,i}', i=1,2,\dots,n,$ where the degree in $u_{j,i}, v_{j,i}$ is 
strictly smaller than $|(l,m)|.$ Finally, since the norm is multiplicative and non-archimedean: 
\begin{align*}
|u_{j+1}^l v_{j+1}^m-u_j^lv_j^m|_{F_0} \leqslant & \max_{\substack{0 \leqslant \beta+\gamma < (l,m)\\ 0 \leqslant \beta \leqslant l, 0  \leqslant \gamma \leqslant m}} |u_j^\beta|_{F_0} |v_j^\gamma|_{F_0} |u_{j}'^{l-\beta}|_{F_0} |v_j'^{m-\gamma}|_{F_0} \\
 \leqslant & \max_{\substack{0 \leqslant \beta+\gamma < (l,m)\\ 0 \leqslant \beta \leqslant l, 0  \leqslant \gamma \leqslant m}} \varepsilon'^{|(\beta,\gamma)|} (\varepsilon'^{\frac{j+2}{2}})^{|(l,m)|-|(\beta,\gamma)|},\\
\end{align*}
so $|u_{j+1}^l v_{j+1}^m-u_j^lv_j^m|_{F_0} \leqslant \max_{0 \leqslant \theta <|(l,m)|} \varepsilon'^{\theta} \cdot (\varepsilon'^{\frac{j+2}{2}})^{|(l,m)|-\theta}.$
This, combined with ${|c_{l,m}^i|_{F_0} \leqslant M^{|(l,m)|}},$ implies that:
\begin{align*}
|f_i(u_{j+1},v_{j+1})-a_i|_{F_0} \leqslant & 
\max_{\substack{|(l,m)|\geqslant 2\\ 0 \leqslant \theta < |(l,m)|}} M^{|(l,m)|}\varepsilon'^{\theta} \cdot (\varepsilon'^{\frac{j+2}{2}})^{|(l,m)|-\theta}\\
 =& \max_{\substack{|(l,m)|\geqslant 2\\ 0 \leqslant \theta < |(l,m)|}} (M\varepsilon')^{\theta} \cdot (M\varepsilon'^{\frac{j+2}{2}})^{|(l,m)|-\theta}.\\
\end{align*}
Since $\varepsilon' \geqslant \varepsilon'^{\frac{j+2}{2}}$ and $M\varepsilon' < 1,$ one obtains:
$$|f_i(u_{j+1},v_{j+1})-a_i|_{F_0} \leqslant \max_{|(l,m)|\geqslant 2} (M\varepsilon')^{|(l,m)|-1} \cdot (M\varepsilon'^{\frac{j+2}{2}}) \leqslant M\varepsilon' \cdot M\varepsilon'^{\frac{j+2}{2}}.$$
At the same time, $M^2 \cdot \varepsilon'^{1+\frac{j+2}{2}}=(\frac{M^2}{d} \varepsilon'^{1/2}) d\varepsilon'^{\frac{j+3}{2}} \leqslant d \varepsilon'^{\frac{j+3}{2}},$  which concludes the induction argument.
\end{sloppypar}

The second property of the sequences $(u_s)_s, (v_s)_s$ tells us that they are Cauchy (hence convergent) in the complete spaces $A_1^n, A_2^n,$ respectively. Let $u \in A_1^n$ and $v \in A_2^n$ be the corresponding limits. The first property implies that $|(u,v)|_{F_0} \leqslant \varepsilon'<\delta,$ so $(u,v) \in \widetilde{S}(F_0),$ and $f(u,v)$ is well defined. Lastly, the third property implies that $f(u,v)=a.$
\end{proof}

From this point on, Theorem \ref{2} can be proven the same way as in~{\cite[Theorem 3.2]{HHK}}. Since $ A_i \hookrightarrow F_i, i=1,2,$ we immediately obtain:
\begin{sloppypar}
\begin{cor}\label{3}
Let $G$ be a connected linear algebraic group over $F$ that is a rational variety over $F.$ For any~${g \in G(F_0)},$ there exist~$g_i \in G(F_i),$ $i=1,2,$ such that $g=g_1 \cdot g_2$ in ~$G(F_0).$
\end{cor}
\end{sloppypar}
\section{Retracting Covers}

In the first section we mentioned that the most important example of Setting 1  was the one given by Proposition \ref{1}. This should serve as motivation for the following:

\begin{defn}\label{nice}
A finite cover $\mathcal{U}$ of a $k$-analytic curve will be called \textit{nice} if:
\begin{enumerate}
\item the elements of $\mathcal{U}$ are connected affinoid domains with only type 3 points in their topological boundaries;
\item for any different $U, V \in \mathcal{U},$ $U \cap V=\partial{U} \cap \partial{V};$ 
\item for any two different elements of $\mathcal{U},$ neither is contained in the other. 
\end{enumerate}
\end{defn}

We recall once again that we will use the term \textit{boundary} for the topological boundary. 

The purpose of this section is to prove that, under certain conditions, for any open cover of a $k$-analytic curve, there exists a \textit{nice refinement}, \textit{i.e.} a refinement that is a nice cover of the curve. The main goal is to be able to apply Corollary \ref{3} to any open cover. 

\begin{defn}
Let $P \in k[T]$ be any irreducible polynomial. We will denote by $\eta_{P,0}$ the only (type 1) point of $\mathbb{A}_k^{1,\mathrm{an}}$ for which $|P|=0.$ For $s \in \mathbb{R}_{>0},$ we will denote by $\eta_{P,s}$ the point of $\mathbb{A}_k^{1,\mathrm{an}}$ that is the Shilov boundary of the affinoid domain $\{|P| \leqslant s\} \subseteq \mathbb{A}_k^{1,\mathrm{an}}.$
\end{defn}

\begin{prop}\label{39}
For any point ${\eta \in \mathbb{A}_k^{1,\mathrm{an}}}$ of type 2 or 3, there exist an irreducible polynomial $P \in k[T]$ and $r \in \mathbb{R}_{> 0},$ such that $\eta=\eta_{P,r}.$ Then, $|P|_{\eta}=r$ and:
\begin{enumerate}
\item $r \in \sqrt{|k^{\times}|}$ if and only if $\eta$ is a type 2 point;
\item $r \not \in \sqrt{|k^{\times}|}$ if and only if $\eta$ is a type 3 point, in which case $\eta$ is the only element of $\mathbb{A}_k^{1,\mathrm{an}}$ for which $|P|=r.$
\end{enumerate}
\end{prop}

\begin{proof} 
We recall that the projective line $\mathbb{P}_k^{1,\mathrm{an}}$ is uniquely path-connected and can be obtained by adding a rational point $\infty$ to $\mathbb{A}_k^{1,\mathrm{an}}.$ For any two points $a,b \in \mathbb{P}_k^{1,\mathrm{an}},$ we denote by $[a,b]$ the unique path connecting them.

Let $A$ be a connected component of $\mathbb{P}_k^{1,\mathrm{an}} \backslash \{\eta\}$ that doesn't contain $\infty.$ In particular, $A \subseteq \mathbb{A}_k^{1,\mathrm{an}}.$ Let $\eta_0$ be any rigid point of $A$. There exists a unique irreducible polynomial $P \in k[T],$ such that $\eta_0=\eta_{P,0}.$ Then, $\eta \in [\eta_{P,0}, \infty].$

Let $\varphi$ be the finite morphism $\mathbb{P}_k^{1,\mathrm{an}} \rightarrow \mathbb{P}_k^{1,\mathrm{an}}$ determined by the map ${k[T] \rightarrow k[T]},$ ${T \mapsto P(T).}$ Seeing as $\varphi(\eta_{P,0})=\eta_{T,0}$ and $\varphi(\infty)=\infty,$ $[\eta_{P,0}, \infty]$ is mapped by $\varphi$ to $[\eta_{T,0}, \infty].$ Set $\eta'=\varphi(\eta).$ The path connecting $\eta_{T,0}$ to $\infty$ in $\mathbb{P}_k^{1,\mathrm{an}}$ is  $\{\eta_{T,s}: s \in \mathbb{R}_{\geqslant 0}\} \cup \{\infty\}.$  For any $s \geqslant 0, |T|_{\eta_{T,s}}=s,$ and if $\eta_{T,s}$ is a type 3 point, then it is the only one in ~$\mathbb{P}_{{k}}^{1,\mathrm{an}}$ for which $|T|=s.$ Furthermore, $\eta_{T,s}$ is a type 2 (resp. type 3) point if and only if $s \in \sqrt{|k^\times|}$ (resp. $s \not \in \sqrt{|k^\times|}$). 

Thus, there exists $r>0,$ such that $\eta'=\eta_{T,r}.$ Since $\varphi(\eta)=\eta_{T,r},$ by construction, $\eta=\eta_{P,r}$ and $|P|_{\eta_{P,r}}=r.$ Seeing as a finite morphism preserves the type of the point (\textit{i.e.} $\eta_{T,r}$ is a type 2 (resp. 3) point if and only if $\eta_{P,r}$ is so), we obtain (1) and the first part of ~(2).

To prove the second part of (2), we need to show that if $r \not \in \sqrt{|k^\times|},$ $\eta_{P,r}$ is the only point in $\mathbb{A}_k^{1,\mathrm{an}}$ for which $|P|=r.$ Since $P$ is irreducible, by \cite[3.4.24.3]{Duc}, $|P|$ is strictly increasing in $[\eta_{P,0}, \infty),$ and locally constant elsewhere. Hence, $\eta_{P,r}$ is the only point in $[\eta_{P,0}, \infty)$ for which $|P|=r,$ and since it is a type 3 point (\textit{i.e.} $\mathbb{A}_k^{1,\mathrm{an}}$ has exactly two connected components), it is the only such point in $\mathbb{A}_k^{1,\mathrm{an}}.$   
\end{proof}

Let us recall that we denote by $\partial_B(\cdot/\cdot)$ the Berkovich relative boundary, and by $\partial_B(\cdot)$ the boundary relative to the base field $k,$ \textit{i.e.} $\partial_B(\cdot/\mathcal{M}(k))$ (see \cite[Definition 2.5.7]{Ber90} and \mbox{\cite[Definition ~1.5.4]{ber93})}.   

\begin{lm} \label{rita}
Let $V$ be a $k$-affinoid curve. The following sets are equal:
\begin{enumerate}
\item{the Berkovich boundary $\partial_B(V)$ of $V;$}
\item{the Shilov boundary $\Gamma(V)$ of $V.$}
\end{enumerate}
\end{lm}
\begin{sloppypar}
\begin{proof}
If $V$ is strictly affinoid, this is \cite[Lemma 2.3]{rita}. The proof can be extended to the general case by replacing classical reduction with Temkin's graded reduction (see Propositions 3.3 and 3.4 of \cite{tem1}). 
\end{proof}
\end{sloppypar} 
\begin{prop}\label{nc}
Let $C$ be a $k$-analytic curve such that $\partial_B(C)=\emptyset$. Let $V$ be an affinoid domain of ~$C.$ The three following sets coincide:
\begin{enumerate}
\item{the topological boundary $\partial{V}$} of $V$ in $C;$
\item{the Berkovich relative boundary $\partial_B(V/C)$ of $V$ in $C$;}
\item{the Shilov boundary $\Gamma(V)$ of $V.$}
\end{enumerate}
\end{prop}

\begin{proof}
By \cite[Corollary 2.5.13 (ii)]{Ber90}, $\partial_B(V/C)=\partial{V}.$ By \cite[Proposition 1.5.5 (ii)]{ber93}, since $C$ is boundaryless, $\partial_B(V/C)=\partial_B(V).$ Finally, in view of Lemma \ref{rita}, ${\partial{V}=\partial_B(V/C)=\Gamma(V)}.$
\end{proof}
Let us recall that analytification of an algebraic variety is boundaryless. In particular, projective $k$-analytic curves are boundaryless.

Until the end of this section, suppose that  $\sqrt{|k^{\times}|} \neq \mathbb{R}_{>0},$ so that there exist type 3 points in $\mathbb{P}_k^{1,\mathrm{an}}.$ 
\begin{thm}\label{4}
Let $C$ be a  $k$-analytic curve. The family of connected affinoid domains with only type 3 points in their topological boundaries forms a basis of neighborhoods of the Berkovich topology on $C.$
\end{thm}

\begin{proof}
Let $x \in C.$ Seeing as any curve is a good Berkovich space (\textit{i.e.} all points have a neighborhood that is an affinoid domain), we may assume that $C$ is an affinoid domain. Let $U$ be an open neighborhood of $x$ in $C.$ 
There exists an open neighborhood 
of $x$ in $U$ given by $\{|f_i|<r_i, |g_j|
>s_j : i=1,2,\dots, n, j=1,2,\dots, m\},$ where 
$f_i, g_j$ are analytic functions on $C$ and 
${r_i, s_j \in \mathbb{R}_{>0}}.$

Let $r_i', s_j' \in \mathbb{R}_{>0} \backslash \sqrt{|k^\times|},$ such that $r_i'<r_i$ and $s_j'>s_j$, and $|f_i(x)|<r_i', |g_j(x)|>s_j',$ for all $i$ and $j.$ Set $V=\{|f_i| \leqslant r_i', |g_j| \geqslant s_j'\}.$ It is an affinoid domain of $C$ and a neighborhood of $x$ contained in $U.$ 

As $\{|f_i| < r_i', |g_j'|>s_j\}$ is open, it is contained in $\text{Int}(V),$ so $\partial{V} \subseteq \bigcup_{i=1}^n \{|f_i|=r_i'\} \cup \bigcup_{j=1}^m \{|g_j|=s_j'\}.$  Let $y \in \bigcup_{i=1}^n \{|f_i|=r_i'\} \cup \bigcup_{j=1}^m \{|g_j|=s_j'\}.$ Since there exists an analytic function $f$ on $C$ such that $|f(y)| \not \in \sqrt{|k^\times|},$ the point $y$ is of type 3, implying that the boundary of $V$ contains only type 3 points. 
\end{proof}

\subsection{The Case of $\mathbb{P}_k^{1,\mathrm{an}}$}

Recall that $\mathbb{P}_k^{1,\mathrm{an}}$ is uniquely path-connected. For any ${x,y \in \mathbb{P}_k^{1,\mathrm{an}}},$ let us denote by $[x,y]$ the unique injective path connecting them. The next few properties of the projective line will be essential to the remainder of this section.
 
\begin{lm}\label{61}
Let $A \subseteq \mathbb{P}_{k}^{1,\mathrm{an}}.$ Then, $A$ is connected if and only if for any $x, y \in A,$~${[x,y] \subseteq A}.$ Furthermore, the intersection of any two connected subsets of $\mathbb{P}_k^{1,\mathrm{an}}$ is connected. 
\end{lm}

\begin{lm}\label{5}
Let $U,V$ be two non-disjoint connected affinoid domains of $\mathbb{P}_k^{1,\mathrm{an}},$ such that they have disjoint interiors. Then, $U \cap V$ is a single point.
\end{lm}

\begin{proof}
Since $U \cap V=\partial U \cap \partial V,$ it is a finite set of points. At the same time, by Lemma ~\ref{61}, $U\cap V$ is connected, so it must be a single point. 
\end{proof}

\begin{lm}\label{62}
Let $U$ be an affinoid domain of $\mathbb{P}_k^{1,\mathrm{an}}$ with only type 3 points in its boundary. If $\mathrm{Int} (U) \neq \emptyset,$ then $(\mathrm{Int} \ U)^c$ is an affinoid domain of $\mathbb{P}_k^{1,\mathrm{an}}$ with only type 3 points in its boundary.
\end{lm}

\begin{proof}
Let us show that $\text{Int} \ U$ has only finitely many connected components.

Since $U$ is an affinoid domain, it has a finite number of connected components $U_i$, and by \cite[Corollary 2.2.7]{Ber90}, they are all affinoid domains. Furthermore, $U_i$ has only type 3 points in its boundary for all $i.$ 

Let $x,y \in \text{Int} \ U_i.$ Since $U_i$ is connected, $[x, y] \subseteq U_i.$ Let $z \in \partial{U_i}.$ We aim to show that $z \not \in [x,y],$ implying $[x,y] \subseteq \text{Int}\ {U_i},$ and thus the connectedness of $\text{Int} \ U_i.$ 

By \cite[Th\'eor\`eme 4.5.4]{Duc}, there exists a neighborhood $V$ of $z$ in $U_i$ such that it is a closed virtual annulus, and its Berkovich boundary is $\partial_B(V)=\{z,u\}$ for some $u \in U_i.$ By shrinking this annulus if necessary, we may assume that $x,y \not \in V.$ Since $V$ is an affinoid domain in $U_i,$ by \cite[Corollary 2.5.13 (ii)]{Ber90}, the topological boundary $\partial_U{V}$ of $V$ in $U$ is a subset of $\partial_B(V)=\{z,u\}.$ Since $V$ is a neighborhood of $z,$ $\partial_U{V}=\{u\}.$

Suppose  $z \in [x,y].$ Then  we could decompose $[x,y]=[x,z] \cup [z,y].$ Since $x,y \not \in V$, and $z \in V,$ the sets $[x,z] \cap \partial_U{V},$ $[z,y] \cap \partial_U{V}$ are non-empty, thus implying $u$ is contained in both $[x,z]$ and $[z,y],$ which contradicts the injectivity of $[x,y].$

We have shown that $\text{Int} \ U$ has finitely many connected components, so by \cite[Proposition 4.2.14]{Duc}, $(\text{Int} \ U)^c$ is a closed proper analytic domain of $\mathbb{P}_k^{1,\mathrm{an}}.$ By \cite[Th\'eor\`eme 6.1.3]{Duc}, it is an affinoid domain. 
\end{proof}

We can now show a special case of the result we prove in this section. 

\begin{lm}\label{63}
Let $C,D$ be connected affinoid domains of $\mathbb{P}_k^{1,\mathrm{an}}$ with only type 3 points in their boundaries. There exists a nice refinement $\{ C_1, \dots, C_n, D\}$ of the cover $\{C,D\}$ of $C \cup D$, such that $C_i \cap C_j = \emptyset$ for any $i \neq j.$
\end{lm}

\begin{proof}
If $C=D,$ it is straightforward. Otherwise, suppose $C \not \subseteq D.$ By Lemma ~\ref{62}, ~$C \backslash \text{Int} \ D$ is an affinoid domain. Let $C_1', C_2', \dots, C_m'$ be its connected components. They are mutually disjoint connected affinoid domains with only type 3 points in their boundaries. By construction, for any $i$ the intersection $C_i' \cap D$ is either empty or a single type 3 point, $C_i' \cap C_j' =\emptyset$ for all $i \neq j,$ and $\{C_1', C_2', \dots, C_m', D\}$ is a refinement of $\{C,D\}.$ 
For any $i,$ if $C_i'$ is a single point, \textit{i.e.} $C_i' \subseteq D,$ we remove it from $\{C_1', C_2', \dots, C_m'\},$ and if not, we keep it there. Let $C_1, C_2, \dots, C_n$ be the remaining connected components of $C \backslash \text{Int} \ D.$ Then, $\{C_1, C_2, \dots, C_n, D\}$ is a nice refinement of the cover $\{C,D\}$ of $C \cup D.$
\end{proof}

The main result of this section in the case of the projective line is the following generalization:

\begin{prop}\label{6}
For any $n \in \mathbb{N},$ let $\{U_i\}_{i=1}^n$ be a set of affinoid domains of $\mathbb{P}_k^{1,\mathrm{an}}$ with only type 3 points in their boundaries. Set $V_n=\bigcup_{i=1}^n U_i.$ Then, there exists a nice cover of $V_n$ that refines~$\{U_i\}_{i=1}^n,$ satisfying the following properties:
\begin{enumerate}
\item{the intersection of any two of its elements is either empty or a single type 3 point;}
\item{if two domains of the refinement intersect, there is no third one that intersects them both.}
\end{enumerate}
\end{prop}

\begin{proof}
We will use induction on the number of affinoids domains $n.$ For $n=1,$ the statement is trivial. Suppose the proposition is true for any positive integer smaller or equal to some $n-1.$ Let $\{U_i\}_{i=1}^n$ be affinoid domains of $\mathbb{P}_k^{1,\mathrm{an}}$ with only type 3 points in their boundaries. If they are all of empty interior, \textit{i.e.} unions of points, then the statement is trivially true. Otherwise, let $i_0 \in \{1, 2, \dots, n\}$ be any index for which $U_{i_0}$ has  non-empty interior. To simplify the notation, suppose $i_0=n.$ By removing the $U_i$'s contained in $U_n$ if necessary, we may assume that for all $i,$ $U_{i} \not \subseteq U_n.$  

From Lemmas \ref{62} and ~\ref{63}, ${\mathcal{U}=\{U_n\} \cup \{U_i \cap (\text{Int}\ U_n)^c\}_{i=1}^{n-1}}$ is a refinement of $\{U_i\}_{i=1}^n$ containing affinoid domains with only type 3 points in their boundaries.  Let $\{W_l\}_{l=1}^s$ be a nice refinement of ${\{U_i\cap (\text{Int}\ U_n)^c\}_{i=1}^{n-1}}.$ Then, for any $l,$ $U_n \cap W_l \subseteq \partial{U_n}.$ By removing those $W_l$ for which $W_l \subseteq U_n$ if necessary, we obtain that $\{U_n\} \cup \{W_l\}_{l=1}^s$ is a nice refinement of $\{U_i\}_{i=1}^n.$ The first condition of the statement is a direct consequence of Lemma ~\ref{5}.

We have proven that for any positive integer $n,$ there exists a nice refinement of $\{U_i\}_{i=1}^n,$ which satisfies the first property of the statement. Property 2  is immediate from the following:
\begin{lm} \label{jepiiiii}
Let $W_1, W_2, W_3$ be three connected affinoid domains of $\mathbb{P}_k^{1,\mathrm{an}}$ with non-empty interiors and only type 3 points in their boundaries. Suppose their interiors are mutually disjoint. Then, at least one of $W_1 \cap W_2, W_2 \cap W_3, W_3 \cap W_1$ is empty.
\end{lm}
\begin{proof}
Suppose that $W_1 \cap W_2, W_2 \cap W_3,$ and $W_3 \cap W_1$ are all non-empty. If $W_1 \cap W_2 \cap W_3 \neq \emptyset,$ then by Lemma \ref{61} it is a single type 3 point $\{z\}$. Since $\mathbb{P}_k^{1,an} \backslash \{z\}$ has exactly two connected components, and the interiors of $W_1, W_2, W_3$ are non-empty and mutually disjoint, this is impossible. Hence, $W_1 \cap W_2 \cap W_3= \emptyset,$ and so $W_1 \cap W_2, W_2 \cap W_3 $ and $W_3 \cap W_1$ are all non-empty and different. Since $W_1 \cap W_2 \neq \emptyset$, $W_1 \cup W_2$ is a connected affinoid domain with only type 3 points in its boundary. Furthermore, $\text{Int}(W_3) \cap \text{Int}(W_1 \cup W_2) \subseteq (W_3 \cap W_1) \cup (W_3 \cap W_2),$ 
and since this is a finite set of type 3 points, $\text{Int}(W_3) \cap \text{Int}(W_1 \cup W_2)=\emptyset.$

Thus, the interior of $W_1 \cup W_2$ is disjoint to the interior of $W_3.$ By Lemma \ref{5}, $(W_1 \cup W_2) \cap W_3$ is a single type 3 point. But, $W_1 \cap W_3$ and $W_2 \cap W_3$ were both assumed to be non-empty and shown to be different, implying $(W_1 \cap W_3) \cup (W_2 \cap W_3)=(W_1 \cup W_2) \cap W_3$ contains at least two different points, contradiction.

Thus, at least one of $W_1 \cap W_2, W_2 \cap W_3,$ $W_3 \cap W_1$ must be empty. 
\end{proof}
This completes the proof of the proposition.
\end{proof}

In view of Theorem \ref{4}, we obtain:

\begin{thm}\label{7}
Any open cover of a compact subset of $\mathbb{P}_k^{1,\mathrm{an}}$ has a nice refinement.
\end{thm}

The following will be needed later:

\begin{lm}\label{30}
Let $A$ be a connected affinoid domain of $\mathbb{P}_k^{1,\mathrm{an}}.$ Let $S$ be a finite subset of $\mathrm{Int}(A)$ containing only type 3 points. There exists a nice cover $\mathcal{A}$ of $A,$ such that the set of points of intersection of different elements of $\mathcal{A}$ is $S.$
\end{lm}
\begin{sloppypar}
\begin{proof}
Seeing as $S$ consists of type 3 points, they are all contained in a copy of $\mathbb{A}_k^{1,\mathrm{an}}$ in ~$\mathbb{P}_k^{1,\mathrm{an}}.$ Thus, for any element $\eta \in S,$ there exists an irreducible polynomial $P$ over $k$ and a real number $r \not \in \sqrt{|k^\times|}$ such that $\eta=\eta_{P,r}.$ 

Let us prove the statement using induction on the cardinality of $S.$ If $S$ is empty, then the statement is trivially true. 
Suppose we know the statement is true if the cardinality of $S$ is equal to some $n-1.$ 

Let us assume $S$ contains $n$ points. Fix some element $\eta_{P,r} \in S.$ Let $\mathcal{U}$ be a nice cover of $A$ that satisfies the properties of the statement for $S':=S \backslash \{\eta_{P,r}\}.$ There exists a unique $U \in \mathcal{U},$ such that $\eta_{P,r} \in U,$ in which case $\eta_{P,r} \in \text{Int}(U).$ Then, ${\{U \cap \{|P| \leqslant r\}, U \cap \{|P| \geqslant r\}\} \cup \{V \in \mathcal{U} : V \neq U\}}$ is a nice cover that fulfills our requirements.
\end{proof}
\end{sloppypar}

\subsection{Nice Covers of a Berkovich Curve}

\begin{lm}\label{8}
Let $C$ be an irreducible projective generically smooth $k$-analytic curve.  There exists a type 3 point $\eta$ in $C$ such that $C \backslash \{\eta\}$ has exactly two connected components $E_1, E_2$. Furthermore, $E_1 \cup \{\eta\},  E_2 \cup \{\eta\}$ are affinoid domains of $C.$   
\end{lm}

\begin{proof}
By \cite[Th\'eor\`eme 3.7.2]{Duc}, there exists an algebraic projective curve $C^{\mathrm{alg}}/k$ such that $(C^{\mathrm{alg}})^{\mathrm{an}}=C.$ By \cite[Theorem 3.4.1]{Ber90}, there is a bijection between the closed points of $C^{\mathrm{alg}}$ and the rigid points of $C,$ meaning the latter are Zariski dense in $C.$ As $C$ is generically smooth, by \cite[Th\'eor\`eme 3.4]{dex}, the smooth locus of $C$ is a non-empty Zariski open of $C$. Consequently, there exists $\eta_0$ - a rigid smooth point in $C$.

By \cite[Th\'eor\`eme 4.5.4]{Duc}, there exists a  neighborhood $D'$ of $\eta_{0}$ in $C$ which is a  virtual disc. By density of type 3 points in $C$ (Theorem \ref{4}), there exists a type 3 point $\eta \in D'.$ By \cite[3.6.34]{Duc}, $D'$ is uniquely path-connected with a single boundary point $x$. 
By \cite[1.4.21]{Duc}, $D:=\overline{D'}$ - the closure of $D$ in $C$, is uniquely path-connected. Remark that $\partial{D}=\{x\}$, and $D=D' \cup \{x\}.$ 

As it is of type 3, by \cite[4.2.11.2]{Duc}, there exist at most two \textit{branches} coming out of ~$\eta$, and there are exactly two if and only if $\eta \in \text{Int}_B(D).$ As $\eta \in \text{Int}(D)=\text{Int}_B(D)$ (Proposition 3.1.3(i) of \cite{Ber90}), there are two branches coming out of $\eta.$ As $D$ is uniquely path-connected, by \cite[1.3.12]{Duc}, this means that $D \backslash \{\eta\}$ has exactly two connected components. Let us denote them by $A$ and $B,$ and assume, without loss of generality, that $x \in B.$ Remark that $A \subseteq D'.$

Set $E:=(C \backslash D) \cup \{x\} = C \backslash D'$. Let us show that $E$ is connected. Let $a, b \in E.$ Since $C$ is connected, by \cite[Theorem 3.2.1]{Ber90}, there exists an injective path  $[a,b]$ in $C$ connecting $a$ and $b.$ Suppose $[a, b] \cap D'\neq \emptyset.$ Let $d \in [a, b] \cap D'.$ Then, $[a,b]$ induces injective paths $[a,d]$ and $[d,b]$ in $C$ connecting $a$ and $d$, resp. $d$ and $b.$
As $a,b \not \in  D'$ and $d \in D',$ we obtain that $[a,d] \cap \partial{D}, [d,b] \cap \partial{D} \neq \emptyset,$ so $x \in [a,d]$ and $x \in [d,b].$ This contradicts the injectivity of $[a,b]$ unless $x=d,$ which is impossible seeing as $x \not \in D'.$ Thus, $[a, b] \cap D'=\emptyset,$ \textit{i.e.} $[a, b] \subseteq E,$ implying $E$ is connected. 

As $B,E$ are connected, and $B \cap E=\{x\},$ $G:=B \cup E$ is a connected subset of $C.$ Remark that $A \cap G=(A \cap B) \cup (A \cap E) \subseteq D' \cap E=\emptyset.$ Also, $A \cup G \cup \{\eta\}=A \cup B \cup E \cup \{\eta\}=D \cup E=C.$

It only remains to show that $A':=A \cup \{\eta\}$ and $G':=G \cup \{\eta\}$ are affinoid domains in $C.$ By \cite[Proposition 4.2.14]{Duc}, they are both closed analytic domains in $C.$ As $C$ is projective, it is boundaryless, so $\partial_B(A')=\partial{A'}=\{\eta\}$, and the same is true for $G'$ (Proposition \ref{nc}). Let $I$ be an irreducible component of $A'$ (resp. $G'$). By \cite[3.2.3]{Duc}, if $\partial_B(I)=\emptyset,$ then $I=C$, implying $A'$ (resp. $G'$) is $C,$ which is false. Hence, $\partial_B(I) \neq \emptyset.$

As $I$ is a Zariski closed subset of $A'$ (resp. $G'$), there exists a closed immersion (hence, a finite morphism) $I \rightarrow A'$ (resp. $I \rightarrow G'$). By \cite[Corollary 2.5.13(i)]{Ber90} and \cite[Proposition 3.1.3(ii)]{Ber90}, $\partial_B(I)$ is a subset of $I \backslash \text{Int}_B(A')$ (resp. $I \backslash \text{Int}_B(G')$). Hence, $\partial_B(I)$ is a non-empty subset of $\partial_B(A')$ (resp. $\partial_B(G')$). We conclude by \cite[Th\'eor\`eme 6.1.3]{Duc}.
\end{proof}

\begin{rem}
In general, $C \backslash \{\eta\}$ has at most two connected components ``around"~$\eta$, and it might happen that it has exactly one (for example in a Tate curve), see also \cite[4.2.11.2]{Duc} and the remarks made after Lemma ~\ref{40}.
\end{rem}

\begin{prop}\label{9}
Let $C$ be a normal connected projective $k$-algebraic curve. Then, there exists a nice cover $\{U_1, U_2\}$ of $C^{\mathrm{an}}$ - the Berkovich analytification of $C,$ such that $U_1\cap U_2$ is a single type 3 point.
\end{prop}

\begin{proof}
Let $C \rightarrow \mathbb{P}_k^{1}$ be a finite morphism. It induces an embedding of function fields $k(\mathbb{P}_k^1) \hookrightarrow k(C).$ Let $K$ be the separable closure of $k(\mathbb{P}_k^1)$ in $k(C).$ There exists a connected normal projective algebraic curve $Y$ over $k,$ such that $k(Y)=K.$ Since the field extension $K/k(\mathbb{P}_k^1)$ is separable, the induced morphism $Y \rightarrow \mathbb{P}_k^1$ is generically \'etale, so $Y$ is a generically smooth curve. In particular, this implies that the $k$-analytic curve $Y^{\mathrm{an}}$ is generically smooth (\cite[Th\'eor\`eme 3.4]{dex}). At the same time, since the finite extension $k(C)/K$ is purely inseparable, the induced finite type morphism $C \rightarrow Y$ is a homeomorphism. Consequently, by  \cite[Proposition ~3.4.6]{Ber90}, its analytification $f : C^{\mathrm{an}} \rightarrow Y^{\mathrm{an}}$ is a finite morphism that is a homeomorphism. 

By Proposition \ref{8}, there exists a nice cover $\{U_1', U_2'\}$ of $Y^{\mathrm{an}},$ such that $U_1' \cap U_2'$ is a single type 3 point. Seeing as $f$ is finite and a homeomorphism, $U_i:=f^{-1}(U_i'), i=1,2,$ is a connected affinoid domain, and $U_1 \cap U_2$ is a single type 3 point. 
\end{proof}

\begin{nota}
For a nice cover $\mathcal{U}$ of a $k$-analytic curve, let us denote by $S_{\mathcal{U}}$ the finite set of type 3 points that are in the intersections of different elements of $\mathcal{U}.$
\end{nota}

Remark that for a nice cover $\mathcal{U}$ of a $k$-analytic curve $C$, if $s \in S_{\mathcal{U}},$ the set $\{s\}$ is an affinoid domain of $C.$ This is because $\{s\}$ is a connected component of the intersection of two affinoid domains.

The following notion will be needed in the section to come.
\begin{defn}
Let $C$ be a $k$-analytic curve. Let $\mathcal{U}$ be a nice cover of $C.$ A function $T_{\mathcal{U}} : \mathcal{U} \rightarrow \{0,1\}$ will be called \textit{a parity function for} $\mathcal{U}$ if for any different $U', U'' \in \mathcal{U}$ that intersect, $T_{\mathcal{U}}(U') \neq T_{\mathcal{U}}(U'').$ 
\end{defn}

\begin{lm}\label{14}
For any $n \in \mathbb{N},$ let $U_1, U_2, \dots, U_n$ be affinoid domains in $\mathbb{P}_k^{1, an}$ such that $\mathcal{U}_n:=\{U_i\}_{i=1}^n$ is a nice cover of $K_n:=\bigcup_{i=1}^n U_i.$  Then, there exists a parity function $T_{\mathcal{U}_n}$ for $\mathcal{U}_n.$
\end{lm}
\begin{proof}
It suffices to prove the result under the assumption that $K_n$ is connected. We will use induction on the cardinality $n$ of $\mathcal{U}_n.$ If $n=1,$ the statement is trivially true. Suppose it to be true for some $n-1.$ 
\begin{lm}\label{10}
Let $Z$ be a topological space. For any positive integer $m,$ let $\{W_i\}_{i=1}^m$ be a set of closed connected subsets of $Z.$ Suppose $\bigcup_{i=1}^m W_i$ is connected. Then, there exists $i_0 \in \{1,2,\dots, m\},$ such that $\bigcup_{i \neq i_0} W_i$ is connected. 
\end{lm}

\begin{proof}
\begin{sloppypar}
Let $l$ be the largest integer such that $l<m$ and there exist $W_{i_1}, W_{i_2}, \dots, W_{i_l},$ with $\bigcup_{j=1}^l W_{i_j}$ connected. As all the $W_i$ are connected, $l>0.$ Set $J=\{1,2,\dots, m\} \backslash \{i_1, i_2, \dots, i_l\}.$ If $l<m-1,$ then for any $p \in J,$ we obtain $W_p \cap \bigcup_{j=1}^l W_{i_j}=\emptyset$.  This implies that $\left( \bigcup_{p \in J} W_p \right) \cap \left( \bigcup_{j=1}^l W_{i_j}  \right)=\emptyset,$ which contradicts the connectedness of $\bigcup_{i=1}^m W_i.$ Thus, ~${l=m-1}.$
\end{sloppypar}
\end{proof}

Seeing as 
$\bigcup_{i=1}^n U_i$ is connected, from Lemma \ref{10}, there exist $n-1$ elements of $\mathcal{U}_n$ whose union remains connected. For simplicity of notation, assume them to be the elements of  $\mathcal{U}_{n-1}:=\{U_1, U_2, \dots, U_{n-1}\}.$ Then, $\mathcal{U}_{n-1}$ is a nice cover of $K_{n-1}:=\bigcup_{i=1}^{n-1}U_i.$ Let $T_{\mathcal{U}_{n-1}}$ be a parity function for $\mathcal{U}_{n-1}.$ By Lemma \ref{5}, $U_n \cap \bigcup_{i=1}^{n-1} U_i$ is a single type 3 point, so by Lemma \ref{jepiiiii}, $U_n$ intersects exactly one of the elements of $\mathcal{U}_{n-1}.$ Suppose it to be $U_{n-1}.$ Define $T_{\mathcal{U}_n}$ as follows: 
\begin{enumerate}
\item for any $U \in \mathcal{U}_{n-1},$ $T_{\mathcal{U}_n}(U):=T_{\mathcal{U}_{n-1}}(U);$
\item $T_{\mathcal{U}_n}(U_n):=1-T_{\mathcal{U}_{n-1}}(U_{n-1}).$
\end{enumerate}
The function $T_{\mathcal{U}_n}$ is a parity function for $\mathcal{U}_n.$
\end{proof}

\begin{prop}\label{11}
Let $Y, Z$ be $k$-analytic curves with $Y$ normal and $Z$ compact. Let ${f: Z \rightarrow Y}$ be a finite surjective morphism. Suppose $\mathcal{V}$ is a nice cover of $Y.$ Then, the connected components of $f^{-1}(V), {V \in \mathcal{V}},$ form a nice cover $\mathcal{U}$ of $Z,$ such that ${f^{-1}(S_{\mathcal{V}})=S_{\mathcal{U}}}.$

Furthermore, if $T_{\mathcal{V}}$ is a parity function for $\mathcal{V},$ then the function $T_{\mathcal{U}}$ that to an element $U \in \mathcal{U}$ associates $T_{\mathcal{V}}(f(U)),$ is a parity function for $\mathcal{U}.$
\end{prop}

\begin{proof}
Since $Z$ is compact and $Y$ is Hausdorff, $f$ is a closed morphism. By \cite[3.5.12]{Duc}, $f$ is also open. 

If $V$ is any connected affinoid domain of $Y,$ for any connected component $V_0'$ of $f^{-1}(V)$, $f(V_0')=V.$ To see this, recall that by \cite[Lemma 1.3.7]{ber93}, $f_{|f^{-1}(V)}: f^{-1}(V) \rightarrow V$ is a finite morphism of affinoid spaces, and by \cite[Th\'eor\`eme 3.4]{dex}, as $Y$ is normal, so is $V.$ Thus, $f_{|f^{-1}(V)}$ is open and closed. Seeing as $V_0'$ is a connected component of $f^{-1}(V),$ it is both open and closed in $f^{-1}(V),$ so its image is both open and closed in $V.$ As $V$ is connected, $f(V_0')=V.$

The connected components of $f^{-1}(V)$ for all $V \in \mathcal{V}$ form a finite cover $\mathcal{U}$ of $Z$ consisting of affinoid domains (see Corollary 2.2.7(i) of \cite{Ber90}). As $f$ is open, for any $V \in \mathcal{V},$ $\partial(f^{-1}(V))=f^{-1}(\partial{V}).$ Since a finite morphism preserves the type of point of an analytic curve, $\partial{f^{-1}(V)}$ is an affinoid domain containing only type 3 points in its boundary. Thus, the elements of $\mathcal{U}$ are connected affinoid domains containing only type 3 points in their boundaries.

 Let $U_1, U_2 \in \mathcal{U}$ be such that $U_1 \cap U_2 \neq \emptyset.$ Set $V_i=f(U_i), i=1,2.$ Then, ${V_1, V_2 \in \mathcal{V}},$ and~${V_1 \neq V_2}.$ To see the second part, if $V_1=V_2,$  then $U_1, U_2$ would be connected components of $f^{-1}(V_1),$ thus disjoint, which contradicts the assumption $U_1 \cap U_2 \neq \emptyset.$ Seeing as  $U_1 \cap U_2 \subseteq f^{-1}(V_1 \cap V_2),$  $U_1 \cap U_2$ is a finite set of type 3 points. Hence, $U_1 \cap U_2 = \partial{U_1} \cap \partial{U_2}.$ The third condition of a nice cover is trivially satisfied. Since $f^{-1}(\partial{V})=\partial{f^{-1}(V)}$ for all $V \in \mathcal{V},$ it follows that $f^{-1}(S_{\mathcal{V}})=S_{\mathcal{U}}.$ Finally, $T_{\mathcal{U}}(U_1)=T_{\mathcal{V}}(V_1) \neq T_{\mathcal{V}}(V_2) = T_{\mathcal{U}}(U_2),$ so $T_{\mathcal{U}}$ is a parity function for $\mathcal{U}.$
\end{proof}

\begin{cor}
Let $C$ be a normal projective $k$-analytic curve or a strict $k$-affinoid curve. Any open cover of $C$ has a nice refinement. 
\end{cor}

\begin{proof} By Theorem \ref{4}, we may assume that the open cover only  contains elements with finite boundary consisting of type 3 points. Since $C$ is compact, there is a finite subcover $\mathcal{U}$ of the starting open cover. Set $S=\bigcup_{U \in \mathcal{U}} \partial{U}.$ 
There exists a finite surjective morphism $C \rightarrow \mathbb{P}_k^{1,\mathrm{an}}.$ Set $S':=f(S).$ By Lemma~ \ref{30}, there exists a nice cover~$\mathcal{D}$ of $\mathbb{P}_k^{1,\mathrm{an}},$ such that $S_{\mathcal{D}}=S'.$ We conclude by applying Proposition \ref{11}.

If $C$ is a strict $k$-affinoid curve, by Noether's Normalization Lemma there exists a finite surjective morphism $C \rightarrow \mathbb{D},$ where $\mathbb{D}$ is the closed unit disc in $\mathbb{P}_{k}^{1,\mathrm{an}}.$ We conclude as above. 
\end{proof}

\section{A Local-Global Principle over Berkovich Curves}

Unless mentioned otherwise, we assume that  
$k$ is a complete ultrametric field such that $\sqrt{|k^{\times}|} \neq \mathbb{R}_{>0}.$ Following \cite{HHK}:

\begin{defn}
Let $F$ be a field. A linear algebraic group $G$ over $F$ acts \emph{strongly transitively} on an $F$-variety $X$ if $G$ acts on $X$ and for any field extension $E/F,$ either $X(E)=\emptyset$ or the action of $G(E)$ on $X(E)$ is transitive. 
\end{defn}
We start by showing some patching results over nice covers.
\begin{prop}\label{12}
Let $D$ be $\mathbb{P}_k^{1,\mathrm{an}}$ or a connected affinoid domain of $\mathbb{P}_k^{1,\mathrm{an}}.$ Let $\mathcal{D}$ be a nice cover of $D,$ and $T_{\mathcal{D}}$  a parity function for $\mathcal{D}.$
Let $G/\mathscr{M}(D)$ be a connected rational linear algebraic group. Then, for any $(g_s)_{s \in S_{\mathcal{D}}} \in \prod_{s \in S_{\mathcal{D}}}G(\mathscr{M}(\{s\}))$, there exists ${(g_U)_{U \in \mathcal{D}} \in \prod_{U \in \mathcal{D}} G(\mathscr{M}(U))},$ 
satisfying: for any $s \in S_{\mathcal{D}},$ if $U_0, U_1$ are the elements of $\mathcal{D}$ contain $s$ and $T_{\mathcal{D}}(U_0)=0,$ then $g_s=g_{U_0} \cdot g_{U_1}^{-1}$ in $G(\mathscr{M}(\{s\})).$ 
\end{prop}

\begin{proof}
\begin{sloppypar} 
We will use induction on the cardinality $n$ of a nice cover. If $n=2,$ then this is Corollary \ref{3} (considering Proposition \ref{1} with $\mathcal{O}(D)=\mathcal{O}$). Suppose the result is true for some $n-1.$ If $\mathcal{D}=:\{U_1, U_2, \dots, U_n\},$ since $\bigcup_{i=1}^n U_i$ is connected, from Lemma ~\ref{10}, there exist $n-1$ elements of $\mathcal{U}$ whose union remains connected. For simplicity of notation, suppose them to be the elements of $\mathcal{D}':=\{U_1, U_2, \dots, U_{n-1}\}.$ By Lemma \ref{5}, $\bigcup_{i=1}^{n-1} U_i \cap U_n$ is single type 3 point, so by Lemma \ref{jepiiiii}, $U_n$ intersects exactly one of the elements of $\mathcal{D}'.$ To simplify the notation, suppose it to be $U_{n-1}.$ Set $\{\eta\}=U_{n-1} \cap U_{n},$ so that $S_{\mathcal{D}}=S_{\mathcal{D'}} \cup \{\eta\}.$
\end{sloppypar}
Let $(g_s)_{s \in S_{\mathcal{D}}}$ be any element of $\prod_{s\in S_{\mathcal{D}}} G(\mathscr{M}(\{s\})).$
By the induction hypothesis, for $(g_s)_{s\in S_{\mathcal{D}'}} \in \prod_{s \in S_{\mathcal{D'}}}G(\mathscr{M}(\{s\})),$ there exists $(g_U)_{U \in \mathcal{D}'} \in \prod_{U \in \mathcal{D'}}G(\mathscr{M}(U))$, satisfying the conditions of the statement. 
\begin{itemize}
\item Suppose $T_{\mathcal{D}}(U_n)=0.$ By Corollary \ref{3}, there exist $a \in G(\mathscr{M}(U_n))$ and $b \in G(\mathscr{M}(\bigcup_{i=1}^{n-1} U_i)),$ such that $g_{\eta} \cdot g_{U_{n-1}}=a \cdot b$ in $G(\mathscr{M}(\{\eta\})).$ For any $i \neq n,$ set $g_{U_i}'=g_{U_i} \cdot b^{-1}$ in $G(\mathscr{M}(U_i)).$ Also, set $g_{U_n}'=a$ in $G(\mathscr{M}(U_n)).$
\item Suppose $T_{\mathcal{D}}(U_n)=1.$ By Corollary \ref{3}, there exist $c \in G(\mathscr{M}(\bigcup_{i=1}^{n-1} U_i))$ and $d \in G(\mathscr{M}(U_n)),$ such that $g_{U_{n-1}}^{-1} \cdot g_{\eta} = c \cdot d$ in $G(\mathscr{M}(\{\eta\})).$ For any $i \neq n,$ set $g_{U_i}'=g_{U_i} \cdot c$ in $G(\mathscr{M}(U_i)).$ Also, set $g_{U_n}'=d^{-1}$ in $G(\mathscr{M}(U_n)).$
\end{itemize}
The family $(g_{U_i}')_{i=1}^n \in \prod_{i=1}^n G(\mathscr{M}(U_i))$ satisfies the conditions of the statement for $(g_s)_{s \in S_{\mathcal{D}}}$.
\end{proof}

\begin{prop}\label{13}
\begin{sloppypar}
Let $Y$ be an integral strict $k$-affinoid curve. Set $K=\mathscr{M}(Y).$ Let $G/K$ be a connected rational linear algebraic group. For any open cover $\mathcal{V}$ of $Y,$ there exists a nice refinement $\mathcal{U}$ of $\mathcal{V}$ with a parity function $T_{\mathcal{U}},$ such that for any given ${(g_y)_{y \in S_{\mathcal{U}}} \in \prod_{y \in S_{\mathcal{U}}} G(\mathscr{M}\{y\})}$,  there exists $(g_U)_{U \in \mathcal{U}} \in \prod_{U \in \mathcal{U}} G(\mathscr{M}(U)),$  satisfying: for any $y \in S_{\mathcal{U}}$ there are exactly two elements $U', U''$ of $\mathcal{U}$ containing $y,$ and if $T_{\mathcal{U}}(U')=0,$ then  $g_y=g_{U'} \cdot g_{U''}^{-1}$ in $G(\mathscr{M}\{y\}).$ 
\end{sloppypar}
\end{prop}

\begin{proof}
By Theorem \ref{4}, we may assume that the cover $\mathcal{V}$ only contains elements with finite boundary consisting of only type 3 points. Since $Y$ is compact, we may also assume that $\mathcal{V}$ is finite.  

Let $f : Y \rightarrow \mathbb{D}$ be a finite surjective morphism we obtain from Noether's Normalization Lemma, where $\mathbb{D}$ is the closed unit disc in $\mathbb{P}_{k}^{1, \mathrm{an}}.$ Set $S=f(\bigcup_{V \in \mathcal{V}}\partial{V}).$ It is a finite set of type ~3 points. By Lemma ~\ref{30}, there exists a nice cover $\mathcal{D}$ of $\mathbb{D}$ such that $S_{\mathcal{D}}=S.$ Let ~$T_{\mathcal{D}}$ be a parity function for $\mathcal{D}$ (it exists by Lemma \ref{14}). From Proposition ~\ref{11}, the connected components of $f^{-1}(Z'),$ $Z' \in \mathcal{D},$ form a nice cover $\mathcal{U}$ of $Y$ such that $f^{-1}(S_{\mathcal{D}})=S_{\mathcal{U}},$ and $T_{\mathcal{D}}$ induces a parity function $T_{\mathcal{U}}$ for $\mathcal{U}.$ 

Let us show that $\mathcal{U}$ refines $\mathcal{V}.$ Suppose, by contradiction, that $Z \in \mathcal{U}$ is such that there does not exist an element of $\mathcal{V}$ containing it. Then, there must exist $a \in \bigcup_{V \in \mathcal{V}} \partial{V} \subseteq S_{\mathcal{U}}$ such that $a \in \text{Int}(Z).$ Since $a \in S_{\mathcal{U}},$ there exists $U \in \mathcal{U}$ such that $a \in \partial{U}.$ But then, $Z  \cap U \neq \partial{Z} \cap \partial{U},$ which contradicts the fact that $\mathcal{U}$ is a nice cover of $Y.$ Consequently, $\mathcal{U}$ must refine $\mathcal{V}.$

Suppose that for $s \in S_{\mathcal{U}}$ there exist different $U_1, U_2, U_3 \in \mathcal{U}$ containing $s.$ Then, $f(s) \in V_1 \cap V_2 \cap V_3,$ where $V_i:=f(U_i) \in \mathcal{D}$, $i=1,2,3,$ (the fact that $V_i \in \mathcal{D}$ was shown in the beginning of the proof of Proposition \ref{11}). By Lemma \ref{jepiiiii}, this is only possible if at least two of the $V_1, V_2, V_3$ coincide. Suppose, without loss of generality, that $V_1=V_2.$ Then, $U_1, U_2$ are connected components of $f^{-1}(V_1),$ so $U_1 \cap U_2 =\emptyset,$ contradiction. Hence, for any $s \in \mathcal{S}_{\mathcal{U}},$ there exist at most two elements of $\mathcal{U}$ containing $s.$ Considering the definition of $S_{\mathcal{U}},$ there must exist exactly two.

Set $G'=\mathcal{R}_{K/\mathscr{M}(\mathbb{D})}(G)$ - the restriction of scalars from $K$ to $\mathscr{M}(\mathbb{D})$ of $G.$ It is still a connected rational linear algebraic group (see \cite[7.6]{neron} or \cite[Section 1]{Mil}). 

\begin{lm} For any point $s$ of type 3 in $\mathbb{D},$ 
${\mathscr{M}(\{s\}) \otimes_{\mathscr{M}(\mathbb{D})} \mathscr{M}(Y)= \prod_{x \in f^{-1}(s)} \mathscr{M}(\{x\})}.$
\end{lm}
\begin{sloppypar}
\begin{proof} Seeing as $s$ is a type 3 point, the set $f^{-1}(s)$ is finite consisting of only type 3 points. Hence, $\mathcal{O}(\{s\})=\mathscr{M}(\{s\}),$ and $\mathcal{O}(\{x\})=\mathscr{M}(\{x\})$ for all $x\in f^{-1}(s).$ 

Set $A=\mathcal{O}(\mathbb{D}), B=\mathcal{O}(Y),$ and $C=\mathcal{O}(\{s\}).$ Let us denote by $S$ the set of non-zero elements of ~$A.$
We know that $C \otimes_A B= \prod_{x \in f^{-1}(s)} \mathcal{O}(\{x\})=\prod_{x \in f^{-1}(s)} \mathscr{M}(\{x\}).$ Then, localizing on both sides, we obtain: $S^{-1}(C \otimes_A B)=C \otimes_{S^{-1}A} S^{-1}B$ and $S^{-1}\left(\prod_{x \in f^{-1}(s)} \mathscr{M}(\{x\})\right)=\prod_{x \in f^{-1}(s)} \mathscr{M}(\{x\}).$ 
Since $B$ is a finite $A$-module, $S^{-1}B$ is a domain that is a finite dimensional $S^{-1}A$-vector space. Then, for any $b \in B \backslash \{0\},$ the map $S^{-1}B \rightarrow S^{-1}B, \alpha \mapsto b\alpha$ is injective, so surjective. Thus, there exists $b' \in S^{-1}B$ such that ${bb'=1},$ implying $S^{-1}B=\text{Frac}\ B.$ Consequently, ${S^{-1}(C 
\otimes_{A} B)=\mathscr{M}(\{s\}) 
\otimes_{\mathscr{M}(\mathbb{D})} \mathscr{M}
(Y)}.$
\end{proof}
\end{sloppypar}

By the definition of the restriction of scalars, for any $s \in S_{\mathcal{D}},$ one obtains $G'(\mathscr{M}(\{s\}))=G(\mathscr{M}(\{s\}) \otimes_{\mathscr{M}(\mathbb{D})} \mathscr{M}(Y)).$ By the lemma above, $G'(\mathscr{M}(\{s\}))=\prod_{x \in f^{-1}(s)} G(\mathscr{M}(\{x\})).$ 

Consequently, $(g_y)_{y \in S_{\mathcal{U}}} \in \prod_{y \in S_{\mathcal{U}}}G(\mathscr{M}(\{y\}))$ determines uniquely an element $(h_s)_{s \in S_{\mathcal{D}}}$ of $\prod_{s \in S_{\mathcal{D}}} G'(\mathscr{M}(\{s\}))$. By Proposition \ref{12}, there exists $(h_Z)_{Z \in \mathcal{D}} \in \prod_{Z \in \mathcal{D}} G'(\mathscr{M}(Z)),$ such that if for two different $Z_0, Z_1 \in \mathcal{D}$ with $T_{\mathcal{D}}(Z_0)=0,$ $s \in Z_0 \cap Z_1,$ then $h_s=h_{Z_0} \cdot h_{Z_1}^{-1}$ in $G'(\mathscr{M}(\{s\})).$

For any $Z \in \mathcal{D},$ let $Z_1, Z_2, \dots, Z_r$ be the connected components of $f^{-1}(Z).$ The map $\mathscr{M}(Z) \otimes_{\mathscr{M}(\mathbb{D})} \mathscr{M}(Y) \rightarrow \prod_{i=1}^r \mathscr{M}(Z_i)$ induces $G'(\mathscr{M}(Z))=G(\mathscr{M}(Z) \otimes_{\mathscr{M}(\mathbb{D})} \mathscr{M}(Y)) \rightarrow \prod_{i=1}^r G(\mathscr{M}(Z_i)),$ which maps $h_Z$ to an element $(g_{Z_1}, g_{Z_2}, \dots, g_{Z_r})$ of $\prod_{i=1}^r G(\mathscr{M}(Z_i)).$ Thus, for any $U \in \mathcal{U},$ we have an element $g_U \in G(\mathscr{M}(U)).$ It remains to show that given different $U_0, U_1 \in \mathcal{U}$ with $T_{\mathcal{U}}(U_0)=0,$ such that $y \in U_0 \cap U_1$ for some $y\in S_{\mathcal{U}},$ we have $g_y=g_{U_0} \cdot g_{U_1}^{-1}$ in $G(\mathscr{M}(\{y\})).$ This is a consequence of the relation between $T_{\mathcal{D}}$ and $T_{\mathcal{U}}$, and of the commutativity of the following diagram for any $Z \in \mathcal{D}$ and any $s \in Z$ of type ~3:

\begin{center}
\begin{tikzpicture}
  \matrix (m) [matrix of math nodes,row sep=3em,column sep=4em,minimum width=2em]
  {
     \mathscr{M}(Z) & \mathscr{M}_{\mathbb{D}}(\{s\}) \\
     \prod_{i=1}^r \mathscr{M}(Z_i) & \prod_{y \in f^{-1}(s)}\mathscr{M}_{Y}(\{y\}) \\};
  \path[-stealth]
    (m-1-1) edge node [left] {} 
    (m-2-1) edge  node [above] {} 
    (m-1-2)
    (m-2-1.east|-m-2-2) edge node [below] {}
      (m-2-2)
    (m-1-2) edge node [right] {} 
    (m-2-2) 
    (m-2-1);
\end{tikzpicture}
\end{center} 
\end{proof}

\begin{prop}\label{15}
Let $Y$ be a normal irreducible strict $k$-affinoid curve. Set $K=\mathscr{M}(Y).$ Let $X/K$ be a variety, and $G/K$ a connected rational linear algebraic group acting strongly transitively on ~$X.$ The following local-global principles hold:\begin{itemize}
\item $ X(K) \neq \emptyset \iff X(\mathscr{M}_{x}) \neq \emptyset \ \text{for all} \ x \in Y;$
\item for any open cover $\mathcal{P}$ of $Y,$ $ X(K) \neq \emptyset \iff X(\mathscr{M}(U)) \neq \emptyset \ \text{for all} \ U \in \mathcal{P}.$

\end{itemize}
\end{prop}

\begin{proof}
Since $Y$ is irreducible and normal, $\mathcal{O}_x$ is a domain for all $x \in Y,$ and $\mathscr{M}_x=\text{Frac} \ \mathcal{O}_x.$

Seeing as $K \hookrightarrow \mathscr{M}_x$ for all $x \in Y,$ the implication $``\Rightarrow"$ is true. 

Suppose $X(\mathscr{M}_x) \neq \emptyset$ for all $x \in Y.$ Then, there exists an open cover $\mathcal{V}$ of $Y$ such that for any $V \in \mathcal{V},$ $X(\mathscr{M}(V)) \neq \emptyset.$ Let $\mathcal{U}$ be a nice refinement of $\mathcal{V}$ given by Proposition \ref{13}, and $T_{\mathcal{U}}$ its associated parity function. Remark that for any $U \in \mathcal{U},$ we have $X(\mathscr{M}(U)) \neq \emptyset.$
\begin{sloppypar}
For $U \in \mathcal{U},$ let $x_U \in X(\mathscr{M}(U)).$ For any $y \in S_{\mathcal{U}},$ there exists exactly one element $U_i \in \mathcal{U},$ with $T_{\mathcal{U}}(U_i)=i,$ $i=0,1,$ containing $y$. From the transitivity of the action of $G,$  there exists ${g_y \in G(\mathscr{M}(\{y\}))},$ such that $x_{U_0}=g_y \cdot x_{U_1}$ in $G(\mathscr{M}(\{y\})).$ This gives us an element $(g_y)_{y \in S_{\mathcal{U}}} \in \prod_{y \in S_{\mathcal{U}}} G(\mathscr{M}(\{y\})).$ By Proposition \ref{13}, there exists ${(g_U)_{U \in \mathcal{U}} \in \prod_{U \in \mathcal{U}} G(\mathscr{M}(U))}$,  satisfying: for any different $U', U'' \in \mathcal{U}$ containing some point $y \in S_{\mathcal{U}},$ such that $T_{\mathcal{U}}(U')=0$ (implying $T_{\mathcal{U}}(U'')=1$), $g_y=g_{U'} \cdot g_{U''}^{-1}$ in $G(\mathscr{M}\{y\}).$ 
\end{sloppypar}
For any $U \in \mathcal{U},$  set $x_{U}'=g_{U}^{-1} \cdot x_U \in X(\mathscr{M}(U)).$  We have construced a meromorphic function over $U$ for any $U \in \mathcal{U}.$ Let us show they are compatible, \textit{i.e.} that they coincide on the intersections of the elements of $\mathcal{U}$. Let $D, E \in \mathcal{U}$ be such that $D \cap E \neq \emptyset.$ Suppose $T_{\mathcal{U}}(D)=0.$ For any $s \in D \cap E,$  $x_E'=g_E^{-1} \cdot x_E= g_D^{-1} (g_D  g_E^{-1}) \cdot x_E=g_D^{-1} g_s \cdot x_E= g_D^{-1} x_D=x_D'$ in $X(\mathscr{M}(\{s\})).$ Consequently, $x_E'=x_D'$ in $X(\mathscr{M}(E \cap D)).$

Compatibility of these meromorphic functions implies they can be glued to give a meromorphic function on the entire $Y.$ Thus, $X(K)=X(\mathscr{M}(Y)) \neq \emptyset.$

The second version of this local-global principle is a direct consequence of the first ~one.

\end{proof}

Let us show the same result (Theorem \ref{42}) for any $k$-affinoid space. Recall that $\Gamma(\cdot)$ denotes the Shilov boundary of an affinoid space.

\begin{lm}\label{0} Let $k$ be a complete ultrametric field.
Let $E$ be a $k$-affinoid space. Let $e$ be any point of $E.$ Then, the following statements are equivalent:
\begin{enumerate}
\item there exists an affinoid neighborhood $N_0$ of $e$ in $E$ such that $e \in \Gamma(N_0);$
\item for any affinoid neighborhood $N$ of $e$ in $E,$ $e \in \Gamma(N);$
\item $e \in \Gamma(E).$
\end{enumerate}
\end{lm}

\begin{proof}
Suppose there exists an affinoid neighborhood $N_0$ of $e$ in $E,$ such that $e \in \Gamma(N_0).$ By \cite[Proposition 2.5.20]{Ber90}, $\Gamma(N_0) \subseteq \partial_B(N_0/E) \cup (\Gamma(E) \cap N_0).$ Since $\partial_B(N/E)$ is the topological boundary of $N_0$ in $E$ (see \cite[Corollary 2.5.13 (ii)]{Ber90}), we obtain that $e \not \in \partial_B(N_0/E),$ implying $e \in \Gamma(E) \cap N_0 \subseteq \Gamma(E).$

On the other hand, if $e \in \Gamma(E),$ for any affinoid neighborhood $N$ of $e$ in $E,$ since $\Gamma(E) \cap N \subseteq \Gamma(N)$ (see \cite[Proposition 2.5.20]{Ber90}), we obtain $e \in \Gamma(N).$
\end{proof}

\begin{lm}\label{40}
Let $Y$ be an integral $k$-affinoid curve. Let $y \in Y$ be any point of type 3, and $Z$ a connected affinoid neighborhood of $y$ in $Y.$ Then,
\begin{enumerate}
\item
the subspace $Y \backslash \{y\}$ has at most two connected components at the neighborhood of ~$y;$ it is connected at the neighborhood of $y$ if and only if $y \in \Gamma(Y);$

\item if $y \in \Gamma(Y),$ then there exist connected affinoid domains $A, B$ of $Y,$ such that $A$ is a neighborhood of $y$ in $Z,$ $\Gamma(Y) \cap A=\{y\},$ $A \cup B=Y,$ and  $A \cap B$ is a single type ~3 point;
\item if $k$ is non-trivially valued and $y \not \in \Gamma(Y),$ there exists a strict affinoid neighborhood of $y$ in $Y.$
\end{enumerate}
\end{lm}

\begin{proof}
Let $p$ denote the characteristic exponent of $k.$ Then, by \cite[Th\'eor\`eme 6.10]{dex}, there exists $n$ such that $Y':=(Y \times k^{1/p^n})_{\text{red}}$ is geometrically reduced. Since $k^{1/p^n}/k$ is a  purely inseparable field extension, the map $f:Y' \rightarrow Y$ is a homeomorphism. As $Y'$ is geometrically reduced, the set of its smooth points is a non-empty Zariski-open subset, \textit{i.e.} the complement of a set of rigid points. Consequently, since $y':=f^{-1}(y)$ is non-rigid, it is smooth in $Y'.$ Remark also that by \cite[Proposition 4.2.14]{Duc}, the image (resp. preimage) of a connected affinoid domain is a connected analytic domain, and thus by \mbox{\cite[Th\'eor\`eme 6.1.3]{Duc}}, a connected affinoid domain. Finally, for any affinoid domain $U$ of $Y',$ we have that $\Gamma(U)=f^{-1}(\Gamma(f(U)))$: by Proposition 2.5.8 (iii) and \mbox{Corollary 2.5.13 (i)} of \cite{Ber90}, this is true for finite morphisms, and taking the reduction of an affinoid space does not change its Shilov boundary.  Set $Z':=f^{-1}(Z).$ It suffices to prove the statement for $Y', y', Z'.$

(1) By \cite[Th\'eor\`eme 4.5.4]{Duc}, $y'$ has an affinoid neighborhood $A'$ in $Y'$ (we may assume, seeing as type 3 points are dense, that $\partial{A'}$ consists of type 3 points) that is a closed virtual annulus, implying $\partial_B(A')$ contains exactly two points. Thus, $A'$ has at most two connected components at the neighborhood of $y',$ and it is connected there if and only if $y' \in \Gamma(A').$

Finally, $Y'$ has at most two connected components at the neighborhood of $y'$, and by Lemma \ref{0}, it is connected there if and only if $y' \in \Gamma(Y').$

\begin{sloppypar}
(2) Suppose furthermore that $y' \in \Gamma(Y'),$ implying $y' \in \Gamma(A').$ Set  $\Gamma(A')=\{y',z'\},$ where $z'$ is a type 3 point. Then, $\partial{A'}=\{z'\}$ and by Proposition 4.2.14 and Th\'eor\`eme 6.1.3 of \cite{Duc}, $B':=(Y' \backslash A')\cup \{z'\}$ is an affinoid domain. We have: $A' \cup B'=Y', A' \cap B'=\{z'\}$ (which implies $B'$ is connected). Finally, by shrinking $A'$ if necessary, we can always assume $z' \not \in \Gamma(Y'),$ and since $\Gamma(Y') \cap A' \subseteq \Gamma(A'),$ this implies $\Gamma(Y') \cap A' =\{y'\}.$
\end{sloppypar}
(3) If $y' \not \in \Gamma(Y'),$ then $y' \not \in \Gamma(A'),$ and for the non-trivially valued field $k^{1/p^n},$ the statement follows from the fact that $A'$ is a closed virtual annulus. 
\end{proof}

By the terminology introduced in \cite[Section 1.7]{Duc} and \cite[Th\'eor\`eme 3.5.1]{Duc}, the first part of Lemma \ref{40} shows that  points of type 3 of certain $k$-analytic curves have at most two branches. Furthermore, in view of Lemma \ref{rita} and \cite[Proposition 1.5.5 (ii)]{ber93}, it has one branch if and only if it is in the Berkovich boundary of the curve. 

The following argument will be used often in what is to come.

\begin{lm} \label{41} Let $C$ be a normal irreducible $k$-analytic curve. Set $F=\mathscr{M}(C).$ Let $X/F$ be a variety, and $G/F$ a connected rational linear algebraic group acting strongly transitively on $X.$
\begin{enumerate}
\item  Suppose $X(\mathscr{M}_x) \neq \emptyset$ for all $x \in C.$ Let $Z$ be any affinoid domain of $C.$ Then, $G_Z:=G \times_{F} \mathscr{M}(Z)$ is a connected rational linear algebraic group over $\mathscr{M}(Z)$ acting strongly transitively on the $\mathscr{M}(Z)$-variety $X_Z:=X \times_F \mathscr{M}(Z).$ Furthermore, $X_{Z}(\mathscr{M}_{Z,x}) \neq \emptyset$ for all $x \in Z,$ where $\mathscr{M}_Z$ is the sheaf of meromorphic functions over $Z.$
\item  Let $U_1, U_2$ be connected affinoid domains of $C$ such that $U_1 \cap U_2=\{s\},$ where $s$ is a type 3 point. If $X(\mathscr{M}(U_i)) \neq \emptyset, i=1,2,$ then $X(\mathscr{M}(U_1 \cup U_2)) \neq \emptyset.$
\end{enumerate}
\end{lm}

\begin{proof}
(1) That $G_Z=G \times_F{\mathscr{M}(Z)}$ is still a connected rational linear algebraic group acting strongly transitively on the variety $X_Z=X \times_F \mathscr{M}(Z)$ is immediate. Also, $\mathscr{M}_x \hookrightarrow \mathscr{M}_{Z,x}$ for any ${x \in Z}$. Thus, $X(\mathscr{M}_x)\neq \emptyset$ implies $X(\mathscr{M}_{Z,x})=X_Z(\mathscr{M}_{Z,x})\neq \emptyset$ for any $x \in Z.$ 

(2) Let $x_i \in X(\mathscr{M}(U_i)), i=1,2.$ By the transitivity of the action of $G$, there exists $g \in G(\mathscr{M}(\{s\})),$ such that $x_1=g \cdot x_2$ in $X(\mathscr{M}(\{s\})).$ By Corollary \ref{3}, there exist $g_i \in G(\mathscr{M}(U_i)),$ such that $g=g_1 \cdot g_2$ in $G(\mathscr{M}(\{s\})).$ Thus $g_1^{-1} \cdot x_1=g_2 \cdot x_2$ in $X(\mathscr{M}(\{s\})).$ Set $x_1'=g_1^{-1} \cdot x_1$ and $x_2'=g_2 \cdot x_2.$ They represent meromorphic functions over $U_1$ and $U_2,$ respectively, whose restrictions to $U_1 \cap U_2$ are compatible. Thus, they can be glued to give a meromorphic function $x$ over $\mathscr{M}(U_1 \cup U_2),$ where $x \in X(\mathscr{M}(U_1 \cup U_2)),$ implying ${X(\mathscr{M}(U_1 \cup U_2)) \neq \emptyset.}$
\end{proof}

\begin{thm}\label{42} Suppose $k$ is non-trivially valued.
Let $Y$ be a normal irreducible $k$-affinoid curve. Set $K=\mathscr{M}(Y).$ Let $X/K$ be a variety, and $G/K$ a connected rational linear algebraic group acting strongly transitively on $X.$ The following local-global principles hold:
\begin{itemize}
\item $ X(K) \neq \emptyset \iff X(\mathscr{M}_{x}) \neq \emptyset \ \text{for all} \ x \in Y;$
\item for any open cover $\mathcal{P}$ of $Y,$ $ X(K) \neq \emptyset \iff X(\mathscr{M}(U)) \neq \emptyset \ \text{for all} \ U \in \mathcal{P}.$
\end{itemize}
\end{thm}

\begin{proof}
Seeing as $K \hookrightarrow \mathscr{M}_x$ for any $x \in Y,$ the direction $``\Rightarrow$" is true.

For the other one, let us use induction on the number $n$ of type 3 points in the Shilov boundary of $Y.$
If $n=0,$ then by \cite[Corollary 2.1.6]{Ber90}, $Y$ is a strict $k$-affinoid curve, in which case the statement has already been proven in Proposition \ref{15}. Assume we know the statement for any positive integer not larger than $n-1, n>0.$

Suppose $\Gamma(Y)$ contains $n$ type 3 points. Let $u \in \Gamma(Y)$. Since $X(\mathscr{M}_u)  \neq \emptyset,$ there exists a connected affinoid neighborhood $U_1'$ of $u$ in $Y,$ such that $X(\mathscr{M}(U_1')) \neq \emptyset.$ By Lemma \ref{40}, there exist two connected affinoid domains $U_1, U_2$ of $Y,$ such that $U_1$ is a neighborhood of ~$u$ in $U_1',$ ${\Gamma(Y) \cap U_1=\{u\}},$ $U_1 \cup U_2=Y,$ and $U_1 \cap U_2=\{s\},$ where $s$ is a type 3 point. Since $U_1 \subseteq U_1',$ we obtain $X(\mathscr{M}(U_1')) \subseteq X(\mathscr{M}(U_1)),$ so $X(\mathscr{M}(U_1)) \neq \emptyset.$ Let $U_s$ be a connected strict affinoid neighborhood of $s$ in $Y$ (see Lemma \ref{40}). Set $Z_i=U_i \cup U_s, i=1,2.$ It is an integral affinoid domain. Let us show $\Gamma(Z_2)$ contains at most $n-1$ type 3 points.

For any $y \in U_s$ of type 3, seeing as $\Gamma(U_s)$ doesn't contain any type 3 points, ${y \not \in \Gamma(U_s)}.$ Taking into account $\Gamma(Z_i) \cap U_s \subseteq \Gamma(U_s),$ we obtain $y \not \in \Gamma(Z_i).$ Similarly, for any ${y \in U_i \backslash \Gamma(U_i)},$ we have $y \not \in \Gamma(Z_i).$ Thus, if $z$ is a type 3 point in the Shilov boundary of $Z_i,$ then $z \in \Gamma(U_i).$ For a subset $S$ of $Y,$ let us denote by $S_{3}$ the set of type 3 points contained in $S.$ We have just shown that $\Gamma(Z_i)_3=\Gamma(U_i)_3 \backslash \{s\}, i=1,2.$ At the same time, $\Gamma(Y)_3$ is a disjoint union of $\Gamma(U_i)_3 \backslash \{s\}, i=1,2$.  By construction, $u \in\Gamma(U_1)_3 \backslash \{s\},$ so the cardinality of $\Gamma(Z_2)_3$ is at most $n-1$.  

By the first part of Lemma \ref{41}, $X_{Z_2}(\mathscr{M}_{Z_2,x}) \neq \emptyset$ for any $x \in Z_2.$ In view of the paragraph above and the induction hypothesis, $X(\mathscr{M}(Z_2))=X_{Z_2}(\mathscr{M}(Z_2)) \neq \emptyset.$ Seeing as $\mathscr{M}(Z_2) \subseteq \mathscr{M}(U_2),$ we obtain $X(\mathscr{M}(U_2)) \neq \emptyset.$ Considering we also have $X(\mathscr{M}(U_1))\neq \emptyset,$ we can conclude by applying the second part of Lemma \ref{41}.

The second version of this local-global principle is a direct consequence of the first ~one.
\end{proof}
We are now able to prove the following:
\begin{thm}\label{16} 
\begin{sloppypar}
Let $k$ be a complete valued non-archimedean field such that ${\sqrt{|k^{\times}|} \neq \mathbb{R}_{>0}}.$
Let $C$ be a normal irreducible projective $k$-analytic curve. Set $F=\mathscr{M}(C).$ Let $X/F$ be a variety, and $G/F$ a connected rational linear algebraic group acting strongly transitively on $X.$  The following local-global principles hold:
\end{sloppypar}
\begin{itemize}
\item $ X(F) \neq \emptyset \iff X(\mathscr{M}_{x}) \neq \emptyset \ \text{for all} \ x \in C;$
\item for any open cover $\mathcal{P}$ of $C,$ $ X(F) \neq \emptyset \iff X(\mathscr{M}(U)) \neq \emptyset \ \text{for all} \ U \in \mathcal{P}.$
\end{itemize}
\end{thm}

\begin{proof}
Since $F \hookrightarrow \mathscr{M}_{x}$ for any $x \in C,$ the direction $``\Rightarrow"$ is true.

Suppose $k$ is non-trivially valued. By Proposition \ref{9}, there exists a nice cover $\{Z_1, Z_2\}$ of $C,$ such that $Z_1 \cap Z_2$ is a single type 3 point. Set $\{\eta\}=Z_1 \cap Z_2.$ By the first part of Lemma \ref{41}, $G_{Z_i}$ is a connected rational linear algebraic group acting strongly transitively on the variety $X_{Z_i},$ and $X_{Z_i}(\mathscr{M}_{Z_i,x}) \neq \emptyset$ for any $x \in Z_i, i=1,2.$ Thus, by Theorem \ref{42}, $X(\mathscr{M}(Z_i))= X_{Z_i}(\mathscr{M}(Z_i)) \neq \emptyset.$ We now conclude by the second part of Lemma \ref{41}.

Suppose $k$ is trivially valued. Being a projective analytic curve over a trivially valued field, the curve $C$ has exactly one type 2 point $x.$ In that case, $\mathscr{M}_x=F,$ so the statement is trivially satisfied.

The second version of this local-global principle is a direct consequence of the first ~one.
\end{proof}

The condition on the value group of $k$ can be removed using model-theoretic arguments. We are very grateful to Antoine Ducros for bringing this to our attention.

\begin{thm}\label{ohlala}
Let $k$ be a complete ultrametric field. Let $C$ be an irreducible normal projective $k$-analytic curve. Set $F=\mathscr{M}(C).$ Let $X/F$ be a variety, and $G/F$ a connected rational linear algebraic group acting strongly transitively on $X.$  The following local-global principles hold:
\begin{itemize}
\item $ X(F) \neq \emptyset \iff X(\mathscr{M}_{x}) \neq \emptyset \ \text{for all} \ x \in C;$
\item for any open cover $\mathcal{P}$ of $C,$ $ X(F) \neq \emptyset \iff X(\mathscr{M}(U)) \neq \emptyset \ \text{for all} \ U \in \mathcal{P}.$
\end{itemize}
\end{thm} 

\begin{proof}
 If $\sqrt{|k^\times|} \neq \mathbb{R}_{>0},$ then the statement was already proven in Theorem~ \ref{16}. Let us show that we can always reduce to this case.

Since $F \hookrightarrow \mathscr{M}_x$ for all $x \in C,$ the direction $``\Rightarrow"$ is clear. Assume $X(\mathscr{M}_x) \neq \emptyset$ for all $x \in C.$ Since $C$ is compact, there exists a finite cover $\mathcal{V}$ of $C$ containing only affinoid domains, such that $\{\text{Int}(V): V \in \mathcal{V}\}$ is also a cover of $C,$ and $X(\mathscr{M}(V)) \neq \emptyset$ for all $V \in \mathcal{V}.$ Let $x_V \in X(\mathscr{M}(V)).$ 

Recall that for any $V,$ $\mathscr{M}(V)$ is the fraction field of an algebra of convergent series over ~$k.$ Hence, $C, X, G,$ the action of $G$ on $X,$ the isomorphism of a Zariski open of $G$ to an open of some $\mathbb{A}_F^n,$ and $x_V, V \in \mathcal{V},$ are all determined by countably many elements of $k.$ Let $S \subseteq k$ denote a countable subset containing all these elements. 

Let $k_0$ be the prime subfield of $k.$ Let $k_1$ be the field extension of $k_0$ generated by $S.$ Remark that $k_1$ is countable. By \cite[Theorem 2.3.7]{mark}, there exists a subfield $k_2$ of $k$ that is a countable extension of $k_1,$ such that $k_2 \subseteq k$ is an elementary embedding in the language of valued fields.

Then, by \cite[Theorem 2.5.36]{mark}, there exists a field extension $K$ of $k,$ such that ${K=k_2^I/D},$ where $I$ is an index set and $D$ is a non-principal ultra-filter on $I.$ Furthermore, by \cite[Exercise 2.5.22]{mark}, it is an elementary extension.

Since $k_2$ is a countable subfield of $k,$ the value group of $k_2$ with respect to the valuation induced by that of $k$ satisfies $\sqrt{|k_2^\times|} \neq \mathbb{R}_{>0}.$ Let $k'$ be the completion of $k_2$ with respect to this valuation. Then, $\sqrt{|k'^\times|} \neq \mathbb{R}_{>0}.$

Since $C$ is defined over $k',$ there exists a compact integral $k'$-analytic curve $C'$, such that $C' \times_{k'} k=C.$ Set $F'=\mathscr{M}(C').$ By construction, there exists an $F'$-variety $X',$ and a connected rational linear algebraic group $G'/F'$ acting on $X',$ such that $X=X' \times_{F'} F,$ $G=G' \times_{F'} F,$ and the action of $G$ induced on $X$ is the one given in the statement.  Let us show that $G'$ acts strongly transitively on $X'.$ Let $L/F'$ be any field extension such that $X'(L) \neq \emptyset.$ Set $L_1=L^I/D.$ This is a field containing $F'$ and $k$ (since $k \subseteq k'^I/D \subseteq L_1$), so it is a field extension of $F.$ Consequently, $G'(L_1)=G(L_1)$ acts transitively on $X'(L_1)=X(L_1),$ and since by \cite[Exercise 2.5.22]{mark}, $L \subseteq L_1$ is an elementary embedding, $G'(L)$ acts transitively on $X'(L).$ 

For any $V \in \mathcal{V},$ let $V'$ denote the image of $V$ with respect to the projection morphism $C \rightarrow C'.$ By construction, $X'(\mathscr{M}(V')) \neq \emptyset.$ Hence, $X'(\mathscr{M}_x) \neq \emptyset$ for all $x \in C',$ implying $X'(F') \neq \emptyset,$ thus in particular $X'(F')=X(F') \subseteq X(F) \neq \emptyset.$

The second part of the statement is a direct consequence of the first one.
\end{proof}

We can apply Theorem \ref{ohlala} to the projective variety $X$  defined by a quadratic form ~$q$ over ~$F.$  In \cite[Theorem 4.2]{HHK}, HHK show that for a regular quadratic form $q$ over $F,$ if $\mathrm{char}(F) \neq 2,$ $SO(q)$ - the special orthogonal group of $q$, acts strongly transitively on $X$ when $\dim{q} \neq 2,$ so in that case we can take $G=SO(q).$ If $\dim{q}=2,$ then $X$ may not be connected and consequently the group $SO(q)$ doesn't necessarily act strongly transitively on $X$ (see \cite[Example 4.4]{HHK} and the proof of \cite[Theorem ~4.2]{HHK}). 

\begin{thm}\label{17} Let $k$ be a complete ultrametric field. 
Let $C$ be a compact irreducible normal $k$-analytic curve.  If $\sqrt{|k^\times|}=\mathbb{R}_{>0}$ (resp. $|k^\times|=\{1\}$) assume $C$ is projective (resp. strict). Set $F=\mathscr{M}(C).$
Suppose $char(F) \neq 2.$ Let $q$ be a quadratic form over $F$ of dimension different from $2$. 
\begin{enumerate}
\item{The quadratic form $q$ is isotropic over $F$ if and only if it is isotropic over $\mathscr{M}_x$ for all $x \in C.$}
\item{Let $\mathcal{U}$ be an open cover of $C.$ Then, $q$ is isotropic over $F$ if and only if it is isotropic over $\mathscr{M}(U)$ for all $U \in \mathcal{U}.$}
\end{enumerate} 
\end{thm}

\begin{proof}
By Witt decomposition, $q=q_t \bot q_r,$ where $q_r$ is regular and $q_t$ is totally isotropic. If $q_t \neq 0,$ then $q$ is isotropic, so we may assume that $q$ is regular. Consequently,  Proposition ~\ref{15}, Theorem \ref{42}, and Theorem \ref{ohlala} are applicable according to the paragraph above the statement.
\end{proof}

Because of the relation of Berkovich points to valuations of the function field of a curve, as a result of Theorem \ref{ohlala} we will obtain a local-global principle with respect to completions.

\begin{defn}
Let $k$ be a complete ultrametric field. Let $F$ be a field extension of $k$. For any valuation $v$ on $F,$ we denote by $R_v$ the valuation ring of $F$ with respect to $v,$ and $m_v$ its maximal ideal. We denote by $F_v$ the completion of $F$ with respect to $v.$ We use the following notations:

\begin{itemize}
\item $V_k(F)$ is the set of all rank 1 valuations $v$ on $F$ that extend the valuation of $k;$
\item  $V_0(F)$ is the set of all non-trivial rank 1 valuations on $F$ that when restricted to $k$ are trivial;
\item for a $k$-subalgebra $R$ of $F,$ $R \neq k,$ $V_R'(F)$ is the set of valuations $v \in V_0(F),$ such that $R \subseteq R_v;$
\item $V(F):=V_k(F) \cup V_0(F);$
\item for a $k$-subalgebra $R$ of $F$, $R \neq k,$ $V_R(F):=V_k(F) \cup V_R'(F).$
\end{itemize}
\end{defn}

Remark that if $k$ is trivially valued, then $V(F)$ and $V_R(F)$ contain the trivial valuation on $F$ for any $k$-subalgebra $R$ of $F$, $R \neq k.$

\begin{rem} \label{j}
Let $C$ be a normal irreducible $k$-analytic curve. Then, for any point $x \in C,$ $\mathcal{O}_x$ is either a field or a discrete valuation ring. If $\mathcal{O}_x$ is a field, then $\mathscr{M}_x=\mathcal{O}_x \hookrightarrow \mathcal{H}(x),$ so we endow $\mathscr{M}_x$ with the valuation induced from $\mathcal{H}(x).$ If $\mathcal{O}_x$ is a discrete valuation ring, then we endow $\mathscr{M}_x$ with the corresponding discrete valuation. 
\end{rem}

\begin{prop}\label{val}
Let $k$ be a non-trivially valued complete ultrametric field. Let $C$ be a normal irreducible $k$-analytic curve. 
\begin{enumerate}
\item Suppose there exists an affine curve $S$ over $k$, such that $S^{\mathrm{an}}=C.$ Let $F$ denote the function field of $S.$ Then, there exists a bijective correspondence $C\longleftrightarrow V_{\mathcal{O}(S)}(F).$
\item If $C$ is projective, set $F=\mathscr{M}(C).$ Then, there exists a bijective correspondence $C \longleftrightarrow V(F).$
\end{enumerate}
In either case, if to $x \in C$ is associated the valuation $v$ of $F,$ then $\widehat{\mathscr{M}_x}=F_v,$ where the completion of $\mathscr{M}_x$ is taken with respect to the valuation introduced in Remark \ref{j}.  
\end{prop}

\begin{proof}(1) Let $x \in C.$ If $x$ is a non-rigid point, then $\mathcal{O}_x$ is a field (\cite[Example 3.2.10]{famduc}), so $|\cdot|_x$  is a norm on $A:=\mathcal{O}(S)$ extending that of $k.$ Consequently, it extends to $F=\text{Frac} \ A$ and defines a valuation $v_x$ on $F$ extending that of ~$k,$ \textit{i.e.} $v_x \in V_k(F).$ If $x$ is a rigid point, $\mathcal{O}_{C,x}$ is a dvr, and $k^\times \subseteq \mathcal{O}_{C,x}^\times,$ so the embedding $A \hookrightarrow \mathcal{O}_{C,x}$ induces a discrete valuation on $A$ whose restriction to $k$ is trivial, \textit{i.e.} a discrete valuation $v_x$ on $F$ whose restriction to $k$ is trivial. Moreover, $A \subseteq R_{v_x}$ by definition, so $v_x \in V_A'(F).$

Let us look at the function $C \longrightarrow V_{A}(F),$ $x \mapsto v_x.$ It is injective by the paragraph above. It is also surjective: if $v \in V_k(F),$ then it determines a norm on $A$ that extends that of $k,$ so it corresponds to a non-rigid point of $C;$ if $v \in V_A'(F),$ then $A \subseteq R_v$, and $P:=A \cap m_v$ is a prime ideal of $A,$ so it corresponds to a rigid point $x$ of $C$ for which $\mathrm{ker} |\cdot|_x=P$ (see \cite[Theorem 3.4.1(i)]{Ber90}). 

If $x \in C$ is non-rigid, then $\widehat{\mathscr{M}_x}=\mathcal{H}(x)$, which is the completion of $F$ with respect to ~$v_x$. If $x$ is a rigid point of $C,$ and $P$ its corresponding prime ideal in ~$A,$ then by \cite[Theorem 3.4.1(ii)]{Ber90},  $\widehat{\mathcal{O}_{C,x}}=\widehat{A_P}=\widehat{A},$ where $\widehat{A}$ denotes the completion of $A$ with respect to the ideal ~$P.$ Consequently, $\widehat{\mathscr{M}_x}=\text{Frac} \ \widehat{A}=F_{v_x}$.

(2) Suppose $C$ is projective. Let $C^{\mathrm{alg}}$ be the normal irreducible projective $k$-algebraic curve such that its Berkovich analytification is $C$, and $\pi: C \rightarrow C^{\mathrm{alg}}$ the canonical analytification morphism.  Let $x \in C$. Let $S'$ be an affine Zariski open of $C^{\mathrm{alg}}$ containing ~$\pi(x).$ Since $C$ is irreducible, the function field of $S'$ is $F.$ By (1), there exists an injective map: $C \longrightarrow V(F),$ $x \mapsto v_x.$

Let us show it is also surjective. Let $v \in V(F)$ such that $v_{|k}$ is the starting valuation on ~$k$. Then, by taking any affine Zariski open subset $S'$ of $C^{\mathrm{alg}}$ (as in the paragraph above), seeing as its function field is $F,$ we obtain that $v$ corresponds to some non-rigid point of ~$S'^{\mathrm{an}} \subseteq C.$
\begin{sloppypar}
Suppose $v \in V(F)$ is such that $v_{|k}$ is trivial. Let us consider an embedding $C^{\mathrm{alg}} \rightarrow \mathbb{P}^n_k = \text{Proj} \ k[x_0, x_1, \dots, x_n].$ Let ${\{U_i:=\text{Spec} \ k[x_j/x_i]_{j \neq i}/I_i\}_{i=1}^{n}}$ be a cover of $C^{\mathrm{alg}}$ by standard open sets.  Let ~$i_0$ be such that $|x_{i_0}|_v \geqslant |x_{i}|_v$ for all ~$i.$ Since $|x_i/x_{i_0}|_v \leq 1,$ $\mathcal{O}(U_{i_0}) \subseteq R_v,$ so by (1),  $v$ corresponds to a rigid point of $U_{i_0}^{\mathrm{an}} \subseteq C.$   
\end{sloppypar}
That $\widehat{\mathscr{M}_x}=F_{v_x}$ for all $x \in C$ follows from part (1) by taking an affine Zariski open containing the point $x.$
\end{proof}

Let us now show a local-global principle with respect to all such completions of the field ~$F.$

We are very grateful to  the referee for bringing to our attention the following lemma:

\begin{lm}\label{toto} Let $K$ be a complete valued field and $K_0$ a dense Henselian (called \emph{quasicomplete}  in \cite[Definition 2.3.1]{ber93}) subfield.  Let $F$ be a subfield of $K_0$ and $X$ an $F$-variety. Then, if $F$ is perfect or $X$ is smooth, $$X(K_0) \neq \emptyset \iff X(K) \neq \emptyset.$$
\end{lm}

\begin{proof}
Since $K_0$ is a subfield of $K,$ the implication ``$\Rightarrow$" is clear. Suppose $X(K) \neq \emptyset.$

Suppose $F$ is perfect. By taking the reduction of $X$ if necessary, we may assume that $X$ is reduced.  Let $a \in X(K).$  Denote by $X'$ the (reduced) Zariski closure of $\{a\}$ in $X.$ Since $F$ is perfect, the smooth locus $X''$ of $X'$ is a dense Zariski open  subset of $X'$ containing ~$a.$ Thus, $X''$ is a smooth $F$-variety such that $X''(K) \neq \emptyset,$ implying it suffices to prove the statement in the case $X$ is smooth. 

Suppose $X$ is smooth. Let $a \in X(K).$
Since $X$ is smooth, there exists a neighborhood $U$ of $a$ in $X,$ such that there exists an \'etale morphism $\varphi: U \rightarrow \mathbb{A}^d_F$ for some $d \in \mathbb{N}.$ Let $\varphi_{K}: U_{K} \rightarrow \mathbb{A}^d_{K}$ be the tensorization by $K$, and let us look at its analytification $\varphi_K^{\mathrm{an}}$. Since $a$ is a rational point, $\varphi_K^{\mathrm{an}}$ induces an isomorphism between a neighborhood $V$ of $x$ in $U_{K}^{\mathrm{an}}$ and an open $V'$ of $\mathbb{A}_{K}^{d,\mathrm{an}}.$ Since $K_0$ is dense in $K,$ there exists $b$ in $V',$ such that $b \in \mathbb{A}^d(K)=K^d$ has coordinates over $K_0.$ Let $c$ be the only pre-image of $b$ in $V.$ Then, $c$ is a $K$-rational point over $b.$
\begin{center}
\begin{tikzpicture}
  \matrix (m) [matrix of math nodes,row sep=3em,column sep=4em,minimum width=2em]
  {
     U_K & U_{K_0} \\
     \mathbb{A}_{K}^d & \mathbb{A}_{K_0}^d \\};
  \path[-stealth]
    (m-1-1) edge node [left] {$\varphi_{K}$} 
    (m-2-1) edge  node [above] {} 
    (m-1-2)
    (m-2-1.east|-m-2-2) edge node [below] {$g$}
      (m-2-2)
    (m-1-2) edge node [right]{$\varphi_{K_0}$} 
    (m-2-2) 
    (m-2-1);
\end{tikzpicture}
\end{center} 
Set $b'=g(b) \in \mathbb{A}^d_{K_0}.$ By commutativity of the diagram, since $b$ has coordinates over $K_0$, $b'$ is a closed point of $\mathbb{A}^d_{K_0}$ which is in the image of $\varphi_{K_0}.$  

Since $\varphi$ is \'etale, $\varphi_{K_0}^{-1}(b')$ is a disjoint union $\bigsqcup_i \mathrm{Spec} \ F_i,$ where $F_i$ are separable finite field extensions of $\kappa(b')=K_0.$ At the same time, $\varphi_{K}^{-1}(b)=\bigsqcup_i F_i \otimes_{K_0} K.$ Set $\widehat{F_i}=F_i \otimes_{K_0} K.$

We know that $\varphi_{K}^{-1}(b)(K) \neq \emptyset.$ Then, there exists $i,$ such that $(\mathrm{Spec}  \ \widehat{F}_i) (K) \neq \emptyset,$ so $\widehat{F}_i=K.$ By Proposition 2.4.1 of \cite{ber93}, this implies that $F_i=K_0,$ and so $\varphi_{K_0}^{-1}(b')(K_0) \neq \emptyset,$ implying $X(K_0) \neq \emptyset.$  
\end{proof}

\begin{cor} \label{e fundit}
Let $k$ be a complete ultrametric field. Let $C$ be a normal irreducible $k$-analytic curve. Set $F=\mathscr{M}(C).$ Let $X$ be an $F$-variety. Then, if $F$ is perfect or $X$ is smooth:
$$X(\mathscr{M}_x) \neq \emptyset \iff X(\widehat{\mathscr{M}_x}) \neq \emptyset$$   
for all $x \in C,$ where the completion $\widehat{\mathscr{M}_x}$ of $\mathscr{M}_x$ is taken with respect to the valuations introduced in Remark ~\ref{j}.
\end{cor}

\begin{proof}
If $\mathcal{O}_x$ is a field, then $\mathscr{M}_x$ is Henselian by \cite[Theorem 2.3.3]{ber93}. If $\mathcal{O}_x$ is  not a field, then it is a discrete valuation ring that is Henselian (see \cite[Theorem 2.1.5]{ber93}), so $\mathscr{M}_x$ is Henselian by \mbox{\cite[Proposition 2.4.3]{ber93}}. We conclude by Lemma \ref{toto}.
\end{proof}

Recall once again that an irreducible compact analytic curve is either projective or affinoid (see Th\'eor\`eme 6.1.3 of \cite{Duc}).

\begin{cor}\label{uhlala}
Let $k$ be a complete ultrametric valued field. Let $C$ be a compact irreducible normal $k$-analytic curve.  Set $F=\mathscr{M}(C).$ Let $X/F$ be a variety, and $G/F$ a connected rational linear algebraic group acting strongly transitively on $X.$  The following local-global principles hold if $F$ is perfect or $X$ is smooth:
\begin{enumerate}
\item if $C$ is affinoid and $\sqrt{|k^\times|} \neq \mathbb{R}_{>0},$
$$X(F) \neq \emptyset \iff X(F_v) \neq \emptyset \ \text{for all} \  v \in V_{\mathcal{O}(C)}(F);$$
\item if $C$ is projective,
$$X(F) \neq \emptyset \iff X(F_v) \neq \emptyset \ \text{for all} \  v \in V(F).$$
\end{enumerate}
\end{cor}

\begin{proof}
If $k$ is trivially valued, then the trivial valuation $v_0$ of $F$ is in $V_{\mathcal{O}(C)}(F)$ (resp. $V(F)$), and since $F_{v_0}=F,$ the statement is clear in this case. 

Otherwise, it is a consequence Proposition \ref{15}, Theorem \ref{42}, and Theorem \ref{ohlala} in view of Proposition \ref{val} and Corollary ~\ref{e fundit}.
\end{proof}

\begin{cor}\label{quad}
Let $k$ be a complete non-archimedean valued field. 
Let $C$ be a compact irreducible normal $k$-analytic curve. Set ${F=\mathscr{M}(C)}.$ Suppose $char(F) \neq 2.$ Let $q$ be a quadratic form over $F$ of dimension different from $2$. The following local-global principles hold:
\begin{enumerate}
\item If $C$ is affinoid and $\sqrt{|k^\times|} \neq \mathbb{R}_{>0}$, $q$ is isotropic over $F$ if and only if it is isotropic over all completions $F_v, v\in V_{\mathcal{O}(C)}(F),$ of $F.$
\item If $C$ is projective, $q$ is isotropic over $F$ if and only if it is isotropic over all completions $F_v, v\in V(F),$ of $F.$
\end{enumerate}
\end{cor}

\begin{proof}
If $k$ is trivially valued, then the trivial valuation $v_0$ of $F$ is in $V_{\mathcal{O}(C)}(F)$ (resp. $V(F)$), and since $F_{v_0}=F,$ the statement is clear in this case. 

Otherwise, by Witt decomposition, ${q=q_t \bot q_r},$ where $q_r$ is regular and $q_t$ is totally isotropic. If $q_t \neq 0,$ then $q$ is isotropic. Otherwise, $q$ is regular, so smooth, and we conclude by Corollary \ref{uhlala}.
\end{proof}
\begin{rem}
Recall that for any finitely generated field extension $F/k$ of transcendence degree 1, there exists a unique normal projective $k$-algebraic curve $C^{\mathrm{alg}}$ with function field ~$F.$ Let $C$ be the analytification of $C^{\mathrm{alg}}.$ Then, $\mathscr{M}(C)=F$ (see \cite[Proposition 3.6.2]{Ber90}), so the local-global principles above are applicable to any such field $F.$
\end{rem}

By Corollary 3.8 of \cite{HHK}, if $G_1$ and $G_2$ are linear algebraic groups such that $G_1 \times G_2$ is a connected rational linear algebraic group, then all the results proven in this section remain true for $G_1$ and $G_2.$ 

\section{Comparison of Overfields}

The purpose of this section is to draw a comparison between one of the local-global principles we proved (Theorem \ref{ohlala}) and the one proven in HHK (\cite[Theorem 3.7]{HHK}). More precisely, we will interpret what the overfields appearing in \cite{HHK} represent in the Berkovich setting, and show that \cite[Theorem 3.7]{HHK} can be obtained as a consequence of Theorem \ref{ohlala}. When working over a ``fine" enough model, we show that the converse is also true.

Throughout this section, for a non-archimedean valued field $E,$ we will denote by $E^{\circ}$ the ring of integers of $E,$ $E^{\circ \circ}$ the maximal ideal of $E^{\circ},$ and by $\widetilde{E}$ the residue field of $E.$

Until the end of this section, we assume $k$ to be a complete discretely valued field.
\subsection{Analytic generic fiber and the reduction map} We will be using the notion of generic fibre in the sense of Berkovich. To see the construction in more detail and under less constrictive conditions, we refer the reader to \cite[Section 1]{ber} and \cite[Section 1]{berber}. 

 Let $\mathscr{X}=\text{Spec} \ A$ be a flat finite type scheme over $k^{\circ}.$ Then, the formal completion $\widehat{\mathscr{X}}$ of $\mathscr{X}$ along its special fiber is $\text{Spf}(\widehat{A}),$ where $\widehat{A}$ is a topologically finitely presented ring over $k^{\circ}$ (\textit{i.e.} isomorphic to some $k^{\circ}\{T_1, \dots, T_n\}/I,$ where $I$ is a finitely generated ideal). Remark that $\widehat{A} \otimes_{k^{\circ}} k$ is a strict $k$-affinoid algebra.

The \textit{analytic generic fiber} of $\widehat{\mathscr{X}},$ denoted by $\widehat{\mathscr{X}}_{\eta},$ is defined to be $\mathcal{M}(\widehat{A} \otimes_{k^{\circ}} k),$ where $\mathcal{M}(\cdot)$ denotes the \textit{Berkovich spectrum}. There exists a \textit{reduction map} $\pi: \widehat{\mathscr{X}}_{\eta} \rightarrow \widehat{\mathscr{X}}_s,$ where $\widehat{\mathscr{X}}_s$ is the special fiber of $\widehat{\mathscr{X}},$ which is anti-continuous, meaning the pre-image of a closed subset is open. We remark that $\widehat{\mathscr{X}}_s=\mathscr{X}_s,$ where $\mathscr{X}_s$ is the special fiber of $\mathscr{X}.$ Let us describe ~$\pi$ more explicitly. 

There are embeddings $A \hookrightarrow \widehat{A} \hookrightarrow (\widehat{A} \otimes_{k^{\circ}} k)^{\circ},$ where $(\widehat{A} \otimes_{k^{\circ}} k)^{\circ}$ is the set of all elements ~$f$ of $\widehat{A} \otimes_{k^{\circ}} k$ for which $|f(x)| \leqslant 1$ for all $x \in \mathcal{M}(\widehat{A} \otimes_{k^{\circ}} k).$ Let $x \in \mathcal{M}(\widehat{A} \otimes_{k^{\circ}} k).$ This point then determines a bounded morphism $A \rightarrow \mathcal{H}(x)^{\circ},$ which induces an application ${\varphi_x : A \otimes_{k^{\circ}} \widetilde{k} \rightarrow \widetilde{\mathcal{H}(x)}}.$ The reduction map $\pi$ sends $x$ to $\ker{\varphi_x}.$

The following commutative diagram, where $\phi: \text{Spec}(\widetilde{\widehat{A} \otimes_{k^{\circ}} k}) \rightarrow  \text{Spec}(A \otimes_{k^{\circ}} \widetilde{k})$ is the canonical map, gives the relation between this reduction map and the one from \cite[Section ~2.4]{Ber90}. The morphism $\phi$ is finite and dominant (see \cite[6.1.2 and 6.4.3]{bo} and \mbox{\cite[pg. 17]{thu})}.

$$
\begin{tikzcd}
\mathcal{M}(\widehat{A} \otimes_{k^{\circ}} k) \arrow{r}{r}  \arrow{rd}{\pi} 
  & \text{Spec}(\widetilde{\widehat{A} \otimes_{k^{\circ}} k}) \arrow{d}{\phi} \\
    & \text{Spec}(A \otimes_{k^{\circ}} \widetilde{k})
\end{tikzcd}
$$

The construction above has nice glueing properties. Let $\mathscr{X}$ be a finite type scheme over ~$k^{\circ},$ and $\widehat{\mathscr{X}}$ its formal completion along the special fiber. Then, the \textit{analytic generic fiber} $\widehat{\mathscr{X}}_\eta$ of $\widehat{\mathscr{X}}$ is the $k$-analytic space we obtain by glueing the analytic generic fibers of an open affine cover of the formal scheme $\widehat{\mathscr{X}}.$ In general, $\widehat{\mathscr{X}}_\eta$ is a compact analytic domain of the Berkovich analytification $\mathscr{X}^{\mathrm{an}}$ of $\mathscr{X}.$ If $\mathscr{X}$ is proper, then $\mathscr{X}^{\mathrm{an}}=\widehat{\mathscr{X}}_{\eta}$ (see \cite[2.2.2]{muni}). Similarly, there exists an anti-continuous \textit{reduction map} $\pi: \widehat{\mathscr{X}}_{\eta} \rightarrow \mathscr{X}_s,$ where $\mathscr{X}_s$ is the special fiber of $\mathscr{X}.$

A property we will need is the following:

\begin{prop}\label{philly}
With the same notation as above, suppose $A$ is a normal domain. Then, $\widehat{A}=(\widehat{A} \otimes_{k^{\circ}} k)^{\circ},$ and
the finite morphism $\phi$ from the diagram above is a bijection. 
\end{prop}

\begin{proof}
\begin{sloppypar}
Let us denote by $t$ a uniformizer of $k^{\circ},$ and by $I$ the ideal $t\widehat{A}.$ Recall that $\widehat{A}$ is the completion of $A$ with respect to the ideal $tA$ (and is isomorphic to some $k^{\circ}\{T_1, T_2, \dots, T_n\}/P$; remark that then $\widehat{A} \otimes_{k^{\circ}} k$ is isomorphic to the $k$-affinoid algebra $k\{T_1, T_2, \dots, T_n\}/P$). 
\end{sloppypar}
Set $B=(\widehat{A} \otimes_{k^{\circ}} k)^{\circ}$ and $J=(\widehat{A} \otimes_{k^{\circ}} k)^{\circ \circ}$ - the elements $f$ of $\widehat{A} \otimes_{k^{\circ}} k$ such that $|f|_x <1$ for all $x \in \mathcal{M}(\widehat{A} \otimes_{k^{\circ}} k)$ (\textit{i.e.} $\rho(f)<1$, where the $\rho$ is the spectral norm on $\widehat{A} \otimes_{k^{\circ}} k$).

Remark that for any maximal ideal $m$ of $A,$ $t \in m$ (\textit{i.e.} the closed points of $\text{Spec} \ A$ are in the special fiber). This means that $tA$ is contained in the Jacobson radical of $A.$ Considering this and the fact that $A$ is excellent and normal, by \cite[7.8.3.1]{groth}, $\widehat{A}$ is also normal. At the same time, by \cite[6.1.2, 6.3.4]{bo}, $B$ is the integral closure of $\widehat{A}$ in $\widehat{A} \otimes_{k^o} k.$ Since $\text{Frac} \ \widehat{A}=\text{Frac} \ B,$ we obtain $\widehat{A}=B.$

Let us look at the canonical map $A/t= \widehat{A}/I \rightarrow B/J$ inducing $\phi.$ Let $|\cdot|$ be the norm on the affinoid algebra $\widehat{A} \otimes_{k^{\circ}} k.$

Remark that $\sqrt{I}=J$: let $x \in J,$ so that $\rho(x)=\lim_{n \rightarrow \infty} |x^n|^{1/n}<1,$ implying ${|x^n| \rightarrow 0},$ ${n \rightarrow +\infty.}$ Thus, for large enough $n,$ $x^n \in I,$ so $J \subseteq \sqrt{I}.$ The other containment is clear seeing as $\rho(\cdot) \leqslant |\cdot|$. This means that any prime ideal of $\widehat{A}$ contains $I$ if and only if it contains $J,$ and thus that $\phi$ is a bijection.  
\end{proof}
\subsection{The setup of HHK's \cite{HHK}} Let us start by recalling HHK's framework (see \cite[Notation 3.3]{HHK}):

\begin{nota}Let $T=k^{\circ}$ be a complete discrete valuation ring with uniformizer $t,$ fraction field ~$k,$ and residue field $\widetilde{k}.$ Let $\mathscr{C}$ be a flat normal irreducible projective $T$-curve with function field ~$F.$ Let us denote by $\mathscr{C}_s$ the special fiber of $\mathscr{C}.$ 

For any point $P \in \mathscr{C}_s,$ set $R_P=\mathcal{O}_{\mathscr{C},P}.$ Since $T$ is complete discretely valued, $R_P$ is an excellent ring. Let us denote by ~$\widehat{R_P}$ the completion of $R_P$ with respect to its maximal ideal. Since $R_P$ is normal and excellent, $\widehat{R_P}$ is also a domain. Set $F_P=\mathrm{Frac} \ \widehat{R_P}.$ 

Let $U$ be a proper subset of one of the irreducible components of $\mathscr{C}_s.$ Set ${R_U=\bigcap_{P \in U} \mathcal{O}_{\mathscr{C}, P}}.$ Let us denote by $\widehat{R_U}$ the $t$-adic completion of $R_U.$ By \cite[Notation 3.3]{HHK}, for any $Q \in U,$ $\widehat{R_U} \subseteq \widehat{R_Q}.$ Thus, $\widehat{R_U}$ is an integral domain. Set $F_U=\mathrm{Frac} \ \widehat{R_U}.$
\end{nota}

Let $\mathscr{P}$ be a finite set of closed points of $\mathscr{C}_s$ containing all points at which distinct irreducible components of $\mathscr{C}_s$ meet. Let $\mathscr{U}$ be the set of all irreducible components of $\mathscr{C}_s \backslash \mathscr{P}$ (which here are also its connected componenets). 

The following is the local-global principle proven by HHK in \cite{HHK} and \cite{HHK1}:

\begin{thm}[{\cite[Theorem 3.7]{HHK}}, {\cite[Theorem 9.1]{HHK1}}]  \label{HHK}
Let $G$ be a connected rational linear algebraic group over $F$ that acts strongly transitively on an $F$-variety $X.$ The following statements are equivalent:
\begin{enumerate}
\item $X(F) \neq \emptyset;$
\item $X(F_P) \neq \emptyset$ for all $P \in \mathscr{P}$ and $X(F_U) \neq \emptyset$ for all $U \in \mathscr{U};$
\item $X(F_Q) \neq \emptyset$ for all $Q \in \mathscr{C}_s.$
\end{enumerate}
\end{thm}

The implication $(1) \Rightarrow (2)$ is immediate seeing as $F$ is embedded into $F_P$ and $F_U$ for all $P \in \mathscr{P}$ and $U \in \mathscr{U}.$ Considering for any $U \in \mathscr{U}$ and any $Q \in U,$ $F_U \subseteq F_Q,$ we obtain that $(2) \Rightarrow (3).$ 

We now proceed to show that the remaining implication $(3) \Rightarrow (1)$ is a consequence of Theorem ~\ref{ohlala}. To do this, a comparison will be drawn between the fields $F_Q, Q \in \mathscr{C}_s,$ and the ones appearing in Theorem \ref{ohlala}.

\subsection{The comparison} Let us denote by $C$ the Berkovich analytification of the generic fiber of $\mathscr{C}.$ It is a normal irreducible projective $k$-analytic curve. By \cite[Proposition 3.6.2]{Ber90}, $\mathscr{M}(C)=F,$ where $\mathscr{M}$ is the sheaf of meromorphic functions on $C.$ Since $\mathscr{C}$ is projective, $C=\widehat{\mathscr{C}}_\eta.$ Let $\pi : C \rightarrow \mathscr{C}_s$ be the reduction map.

Let $\mu$ be the generic point of one of the irreducible components of $\mathscr{C}_s.$ Then, $\mathcal{O}_{\mathscr{C}, \mu}$ is a discrete valuation ring with fraction field $F,$ whose valuation extends that of $k.$ As the residue field of $\mathcal{O}_{\mathscr{C}, \mu}$ is of transcendence degree $1$ over $\widetilde{k},$ $\mu$ determines a unique type 2 point $x_{\mu}$ on the Berkovich curve $C.$

\begin{lm} \label{a}
Let $\mu$ be the generic point of one of the irreducible components of $\mathscr{C}_s.$ Then, $\pi^{-1}(\mu)=\{x_{\mu}\}.$
\end{lm}

\begin{proof}
Let $U=\text{Spec} \ A$ be an open affine neighborhood of $\mu$ in $\mathscr{C}.$ Since $\mathscr{C}$ is irreducible, we obtain that $\mathrm{Frac} \ A=F.$ By \cite[pg. 541]{ber}, $\pi^{-1}(U_s)=\widehat{U_{\eta}}$, and the restriction of $\pi$ on $\widehat{U}_{\eta}$ is the reduction map $\widehat{U}_{\eta} \rightarrow U_s.$ Explicitly, we have $\pi: \mathcal{M}(\widehat{A} \otimes_{k^{\circ}} k) \rightarrow \text{Spec}(A \otimes_{k^{\circ}} \widetilde{k}),$ where $x \in \mathcal{M}(\widehat{A} \otimes_{k^{\circ}} k)$ is sent to the kernel of the map $A \otimes_{k^{\circ}} \widetilde{k}=A/k^{\circ \circ}A \rightarrow \widetilde{\mathcal{H}(x)}.$

By construction, for any  $x \in \pi^{-1}(\mu)$ and any $f \in A,$  $f(\mu)=0$ if and only if $|f|_x<1,$ and $f(\mu)\neq 0$ if and only if ~${|f|_x=1}.$  As a consequence, $|f|_{x_\mu}<1$ if and only if $|f|_x<1,$ and $|f|_{x_\mu}=1$ if and only if $|f|_x=1.$ This implies that $x$ and $x_{\mu}$ define the same norm on $A$ (and hence on $F$), so $x_{\mu}=x$ in $C,$ and $\pi^{-1}(\mu)=\{x_{\mu}\}.$
\end{proof}

\begin{prop} \label{b} Let $\mu$ be the generic point of one of the irreducible components of $\mathscr{C}_s.$ Set $\{x_{\mu}\}:=\pi^{-1}(\mu).$ Then, $F_{\mu}=\mathcal{H}(x_{\mu}).$ Let $X$ be an  $F$-variety. If ${X(F_{\mu}) \neq \emptyset},$ then $X(\mathscr{M}_{C,x_{\mu}}) \neq \emptyset.$ 
\end{prop}

\begin{proof}
Remark that $F_{\mu}={\text{Frac} \ \widehat{\mathcal{O}_{\mathscr{C}, \mu}}}$ is the completion of $F$ with respect to the valuation~$x_{\mu}.$ Seeing as $x_{\mu}$ is of type 2, $\mathcal{O}_{C,x_{\mu}}=\mathscr{M}_{C, x_{\mu}}$, and by Proposition \ref{val}, ${F_{\mu}=\widehat{\mathscr{M}_{C, x_{\mu}}}}$ ${=\mathcal{H}(x_{\mu})}$.

If $X$ is smooth or $F$ is perfect, we can conclude by Corollary \ref{e fundit}.

Otherwise, the restriction morphism of the sheaf of meromorphic functions gives us $\mathrm{Frac} \ \mathcal{O}_{\mathscr{C},\mu}=F=\mathscr{M}(C)   \hookrightarrow \mathcal{O}_{C, x_{\mu}},$ so there exist embeddings $\mathcal{O}_{\mathscr{C}, \mu} \subseteq \mathcal{O}_{C,x_{\mu}} \subseteq \mathcal{H}(x_\mu).$ Seeing as all elements of $\mathcal{O}_{\mathscr{C}, \mu}$ have norm at most $1,$ $R_{\mu}=\mathcal{O}_{\mathscr{C}, \mu} \subseteq \mathcal{O}_{C,x_{\mu}}^{\circ}$ - the valuation ring of $\mathcal{O}_{C, x_{\mu}}.$

By the proof of \cite[Proposition 5.8]{HHK1}, $X(F_{\mu}) \neq \emptyset$ implies $X(\widehat{R_{\mu}})\neq \emptyset.$
The ring ${R_{\mu}=\mathcal{O}_{\mathscr{C}, \mu}}$ is excellent, so by Artin's Approximation Theorem (\cite[Theorem 1.10]{art}), $X(R_{\mu}^h) \neq \emptyset,$ where $R_{\mu}^h$ denotes the henselization of the local ring $R_{\mu}.$ Seeing as $\mathcal{O}_{C, x_{\mu}}^{o}$ is Henselian (\mbox{\cite[Thm. 2.3.3, Prop. 2.4.3]{ber93}}), $R_{\mu} \subseteq R_{\mu}^h \subseteq \mathcal{O}_{C,x_{\mu}}^{\circ} \subseteq \mathscr{M}_{C,x_{\mu}}.$ Consequently, ${X(\mathscr{M}_{C,x_{\mu}}) \neq \emptyset}.$
\end{proof}

We recall that the reduction map is anti-continuous. 

\begin{prop}\label{c}
Let $P$ be a closed point of $\mathscr{C}_s.$ Then, ${\widehat{R_P} = \mathcal{O}^{\circ}_C(\pi^{-1}(P))},$ where $\mathcal{O}^{\circ}$ is the sheaf of analytic functions $f$ such that $|f|_{sup} \leqslant 1.$
Consequently, if $X(F_P) \neq \emptyset,$ then $X(\mathscr{M}(\pi^{-1}(P))) \neq \emptyset.$
\end{prop}

\begin{proof}
Let $V=\text{Spec} \ A$ be an open integral affine neighborhood of $P$ in $\mathscr{C}.$ As $\mathscr{C}$ is normal, so is $A.$ Note that $P \in V_s,$ where $V_s$ is the special fiber of $V.$ 

Let $\pi$ denote the specialization map correspoding to $\mathscr{C}$. By \textit{cf.} \cite[pg. 541]{ber}, ${\pi^{-1}(V) = \widehat{V_{\eta}}}$ - the analytic generic fiber of $V$, and the restriction of $\pi$ to $\widehat{V}_{\eta}$ is the specialization map $\widehat{V}_{\eta} \rightarrow V_s$ of $V$.  Thus, $\pi^{-1}(P) \subseteq \widehat{V}_\eta.$
Let us come back to the commutative diagram constructed above: 
$$
\begin{tikzcd}
\widehat{V}_\eta=\mathcal{M}(\widehat{A} \otimes_{k^{\circ}} k) \arrow{r}{r}  \arrow{rd}{\pi} 
  & \text{Spec}(\widetilde{\widehat{A} \otimes_{k^{\circ}} k}) \arrow{d}{\phi} \\
    & \text{Spec}(A \otimes_{k^{\circ}} \widetilde{k})=V_s
\end{tikzcd}
$$

Set $B=(\widehat{A} \otimes_{k^\circ} k)^{\circ}.$
By Proposition \ref{philly}, $\phi$ is a bijection, and $B=\widehat{A}.$ Let $m_P$ be the maximal ideal of $A$
corresponding to the point $P$ on the special fiber, and $\widehat{m_P}$ the corresponding ideal in $\widehat{A},$ \textit{i.e.} the  completion of $m_P$ along the special fiber. Then, $\phi^{-1}(P)$ is a closed point of $\text{Spec}(\widetilde{\widehat{A} \otimes_{k^{\circ}} k})$ corresponding to the maximal ideal $\widehat{m_P}$ of $B=\widehat{A}.$

Since $k^{\circ \circ}A \subseteq m_P,$ $\widehat{A}^{m_P}=\widehat{\widehat{A}}^{\widehat{m_P}}=\widehat{B}^{\widehat{m_P}},$ where the notation $\widehat{R}^S$ is used for the completion of a ring $R$ with respect to the topology induced by an ideal $S.$

As $V$ is reduced, so is its analytification $V^{\mathrm{an}}$ (\cite[Th\'eor\`eme 3.4]{dex}). Since $\widehat{V}_{\eta}$ is an analytic domain of $V^{\mathrm{an}}$, it is reduced (see \cite[Th\'eor\`eme 3.4]{dex}). By a theorem of Bosch (see \cite[Theorem 3.1]{flo},  \cite[Theorem 5.8]{bo1}), 
$$\widehat{B}^{\widehat{m_P}}=\mathcal{O}^{\circ}_{\widehat{V_{\eta}}}(r^{-1}(\phi^{-1}(P)))=\mathcal{O}^{\circ}_{\widehat{V_{\eta}}}(\pi^{-1}(P)).$$
As $P$ is a closed point of $\mathscr{C}_s$ (resp. $V_s$), $\pi^{-1}(P)$ is an open subset of $C$ (resp. $\widehat{V_{\eta}}$), implying $\mathcal{O}^{\circ}_{\widehat{V_{\eta}}}(\pi^{-1}(P))=\mathcal{O}^{\circ}_{C}(\pi^{-1}(P)).$

As a consequence,
$$\widehat{R_{P}}=\widehat{\mathcal{O}_{\mathscr{C},P}}=\widehat{A}^{m_P}= \widehat{B}^{\widehat{m_P}}=\mathcal{O}_C^{\circ} (\pi^{-1}(P)).$$
This implies that $F_P=\text{Frac} \ \mathcal{O}^{\circ}(\pi^{-1}(P)) \subseteq \mathscr{M}(\pi^{-1}(P)).$ The last part of the statement is now immediate. 
\end{proof}

We are now able to state and prove the following argument, thus concluding the proof that HHK's local-global principle (Theorem \ref{HHK}) can be obtained as a consequence of Theorem ~\ref{ohlala}.
\begin{prop} \label{i} 
Using the same notation as in Theorem \ref{HHK}, $(3) \Rightarrow (1).$
\end{prop}

\begin{proof}
Let $x$ be any point of $C.$ Recall $\pi$ denotes the reduction map $C \rightarrow \mathscr{C}_s$.
\begin{enumerate}
\item If $\pi(x)=\mu \in \mathscr{C}_s$ is the generic point of one of the irreducible components of $\mathscr{C}_s,$ then by Proposition \ref{b}, $X(F_{\mu}) \neq \emptyset$ implies $X(\mathscr{M}_{C,x}) \neq \emptyset.$
\item If $\pi(x)=P \in \mathscr{C}_s$ is a closed point, by Proposition \ref{c}, $ F_P \subseteq \mathscr{M}(\pi^{-1}(P)).$ Since $x \in \pi^{-1}(P)$ and $\pi^{-1}(P)$ is open, we obtain $\mathscr{M}(\pi^{-1}(P)) \subseteq \mathscr{M}_{\pi^{-1}(P),x}=\mathscr{M}_{C,x}.$ Hence, $X(F_P)\neq \emptyset$ implies $X(\mathscr{M}_{C,x})\neq \emptyset.$
\end{enumerate}
Finally, seeing as $X(\mathscr{M}_x) \neq \emptyset$ for all $x \in C,$ by Theorem \ref{ohlala}, $X(F) \neq \emptyset.$
\end{proof}

Lastly, using Ducros' work on semi-stable reduction in the analytic setting (see \cite{Duc}, in particular Chapter 6), we can say something in the other direction as well:

\begin{prop}
Let $F$ be a finitely generated field extension of $k$ of transcendence degree ~$1.$ Let $C$ be the normal irreducible projective Berkovich $k$-analytic curve for which $F=\mathscr{M}(C).$ Let $X/F$ be a variety. Then, there exists a flat normal irreducible projective model $\mathscr{C}'$ over $T=k^{\circ}$ of $F,$ such that 
$$  X(\mathscr{M}_x) \neq \emptyset \ \text{for all} \ x \in C  \ \Rightarrow \ X(F_P) \neq \emptyset \ \text{for all} \ P \in \mathscr{C}'_s,$$
where $F_P=\widehat{\mathcal{O}_{\mathscr{C}',P}},$ and $\mathscr{C}'_s$ is the special fiber of $\mathscr{C}'.$  
\end{prop}
Consequently, a local-global principle with respect to the overfields $F_P, P \in \mathscr{C}'_s,$ implies a local-global principle with respect to the $\mathscr{M}_x, x \in C.$

\begin{proof} Suppose $X(\mathscr{M}_x) \neq \emptyset$ for all $x \in C.$ As $C$ is projective, it is strict, so by \cite[Proposition 2.2.3(iii)]{Ber90}, the strict affinoid domains of $C$ form a basis of neighborhoods of the topology of $C.$ Taking also into account that $C$ is compact, there exists a finite cover $\mathcal{U}$ of $C$ such that: 
\begin{enumerate}
\item for any $U \in \mathcal{U},$ $U$ is a connected strict affinoid domain in $C;$
\item $\bigcup_{U \in \mathcal{U}} \text{Int}(U)=C;$
\item for any $U \in \mathcal{U}, X(\mathscr{M}(U)) \neq \emptyset.$
\end{enumerate} 
Let $S$ be the set of all boundary points of the elements of $\mathcal{U}.$ By construction, $S$ is a finite set of type 2 points.

Let us show that $S$ is a \textit{vertex set}  of $C$ using \cite[Th\'eor\`eme 6.3.15]{Duc} (see \cite[6.3.17]{Duc} for the definition of a vertex set, which is called \textit{ensemble sommital} there).
Since $C$ is projective (implying boundaryless) and irreducible, conditions $\alpha), \beta)$ and $\gamma)$ of  \mbox{\cite[Th\'eor\`eme 6.3.15 ii)]{Duc}} are satisfied. Finally, condition $\delta)$ is a consequence of the fact that $S$ contains only type~2 points (see \cite[Commentaire 6.3.16]{Duc}).   

 By \cite[6.3.23]{Duc}, this implies the existence of an irreducible projective model $\mathscr{C}'$ of $F$ over $T$ with special fiber $\mathscr{C}'_s,$ and specialization map ${\pi: C \rightarrow \mathscr{C}'_s},$ such that $\pi$ induces a bijection between $S$ and the generic points of the irreducible components of $\mathscr{C}'_s.$ Furthermore, by \cite[6.3.9.1]{Duc}, since $k$ is discretely valued and $C$ reduced, $\mathscr{C}'$ is locally topologically finitely presented. Finally, by \cite[6.3.10]{Duc}, since $C$ is normal, the model $\mathscr{C}'$ is flat and normal. 

By Proposition \ref{c}, for any closed point $P \in \mathscr{C}'_s,$ $\widehat{\mathcal{O}_{\mathscr{C}',P}}=\mathcal{O}^o(\pi^{-1}(P)),$ where $\mathcal{O}^o$ is the sheaf of holomorphic functions $f,$ such that $|f|_{sup} \leqslant 1.$ In particular, remark that if $V$ is an affinoid domain of $C$, since all holomorphic functions are bounded on $V,$ we have $\mathcal{O}^o(V) \subseteq \mathcal{O}(V).$ This implies $\mathrm{Frac} \ \mathcal{O}^o(V) \subseteq \mathscr{M}(V).$ Let $\frac{f}{g} \in \mathscr{M}(V),$ with $f, g \in \mathcal{O}(V).$ Let $\alpha \in k$ be such that $|\alpha f|_{sup}, |\alpha g|_{sup} \leqslant 1$ (it suffices to choose $\alpha$ so that $|f|_{sup}, |g|_{sup} \leqslant |\alpha^{-1}|,$ which is possible seeing as $k$ is non-trivially valued). Then, $\frac{f}{g}=\frac{\alpha f}{\alpha g} \in \mathrm{Frac} \ \mathcal{O}^o(V),$ implying $\mathscr{M}(V)=\mathrm{Frac} \ \mathcal{O}^o(V).$ By construction, there exists $U \in \mathcal{U},$ such that $\pi^{-1}(P) \subseteq U.$ In particular, $\mathscr{M}({U}) =\mathrm{Frac} \ \mathcal{O}^o(U) \subseteq \text{Frac}(\mathcal{O}^o(\pi^{-1}(P)))=F_P,$ so $X(F_P) \neq \emptyset.$ 

If $P$ is a generic point of $\mathscr{C}'_s,$ then $\pi^{-1}(P)$ is a single type 2 point $x_P$ and $\mathscr{M}_{x_P} \subseteq \mathcal{H}(x_P) =F_P$ (Proposition \ref{b}). Thus, $X(F_P) \neq \emptyset$. 

Since $\pi$ is surjective (\cite[Lemma 4.11]{flo}), this implies that $X(F_P) \neq \emptyset$ for all $P \in \mathscr{C}'_s.$

\end{proof}

\section{The Local Part for Quadratic Forms}

In view of the local-global principle we  proved for quadratic forms (Theorem \ref{17}), we now want to find sufficient conditions under which there is local isotropy. To do this, we will need to put further restrictions on the base field. Throughout this section, we will suppose the dimension of $\sqrt{|k^{\times}|}$ as a $\mathbb{Q}$-vector space (\textit{i.e.} the rational rank of $|k^\times|$) is $n \in \mathbb{Z}.$ In the special case that $|k^{\times}|$ is a free $\mathbb{Z}$-module (\textit{e.g.} if $k$ is a discretely valued field), the sufficient conditions for local isotropy can be refined. The class of such fields is quite broad, especially when it comes to arithmetic questions: if we work over a complete ultrametric base field $k$ satisfying this condition, then for any $k$-analytic space and any of its points $x,$ the field $\mathcal{H}(x)$ also satisfies it.

For any valued field $E,$ we denote by  $E^{\circ}$ its ring of integers, by $E^{\circ \circ}$ the corresponding maximal ideal, and by $\widetilde{E}$ its residue field. 

For the following two propositions, the case of characteristic $2$ can be treated uniformly with the general one. Afterwards, we will restrict to residual characteristic different from ~$2.$ 

\begin{prop}\label{19}
Let $l$ be a valued field. Suppose $|l^\times|$ is a free $\mathbb{Z}$-module of finite rank $n$. Let $L$ be a valued field extension of $l.$ Let $q$ be a non-zero diagonal quadratic form over $L.$ Suppose for any non-zero coefficient $a$ of $q,$ $|a| \in {|l^{\times}|}.$
 There exists a family $Q$ of at most $2^n$ quadratic forms with coefficients in $(L^{\circ})^{\times},$ such that $q$ is $L$-isometric to $\bot_{\sigma \in Q} C_{\sigma} \cdot \sigma,$ where $C_{\sigma} \in L^{\times}$ for any $\sigma \in Q.$
\end{prop}

\begin{proof}
Let us fix $\pi_1, \pi_2, \dots, \pi_n \in l^{\times},$ such that their norms form a basis of the $\mathbb{Z}$-module $|l^{\times}|.$ Set $\mathcal{A}= \lbrace{\prod_{i=1}^n \pi_i^{\delta_i} | \delta_i \in \{0,1\}\rbrace}.$ For any coefficient $a$ of $q,$ let $p_1, p_2, \dots, p_n \in \mathbb{Z}$ be such that $|a|=\prod_{i=1}^n |\pi_i|^{p_i}.$ Then, there exist 
$v_a \in (L^{o})^\times$ and $s_a \in \mathcal{A},$ 
  such that $a \equiv v_a s_a \ \textrm{mod} \ (L^{\times})^2.$ Consequently, for any $A \in \mathcal{A},$ there exists a diagonal quadratic form $\sigma_A$ with coefficients in $(L^{\circ})^{\times},$ such that $q$ is $L$-isometric to $\bot_{A \in \mathcal{A}} A \cdot \sigma_A.$
\end{proof}

The following is the analogue of Proposition \ref{19} in a more general case. 

\begin{prop}\label{18} Let $l$ be a valued field, such that $\dim_{\mathbb{Q}} \sqrt{|l^\times|}$ equals an integer $n.$ 
Let $L$ be a valued field extension of $l.$  Let $q$ be a non-zero diagonal quadratic form over $L.$ Suppose for any non-zero coefficient $a$ of $q,$ $|a| \in \sqrt{|l^\times|}.$ Then, there exists a family $Q$ of at most $2^{n+1}$ quadratic forms with coefficients in $(L^{\circ})^\times,$ such that~ $q$ is $L$-isometric to $\bot_{\sigma \in Q} C_{\sigma} \cdot \sigma,$ where $C_{\sigma} \in L^\times$ for any $\sigma \in Q.$ 
\end{prop}

\begin{proof} 

To ease the notation, let us start by introducing the following:

\begin{nota}
Let $M$ be a multiplicative $\mathbb{Z}$-module, such that the divisible closure $\sqrt{M}$ of $M$ as a group is a finite dimensional $\mathbb{Q}$-vector space. Set $n=\dim_{\mathbb{Q}}\sqrt{M}.$ Set $M^2=\{m^2: m \in M\}.$

There exist $t_1, t_2, \dots, t_n \in M,$ such that for any $t \in M,$ there exist unique ${p_1, p_2, \dots, p_n \in \mathbb{Q}},$ for which $t=\prod_{i=1}^n t_i^{p_i}.$ Let us fix such elements $t_1, t_2, \dots, t_n.$
\end{nota}

In the particular situation that is of interest to us, $M=|l^{\times}|,$ and there exist ${\pi_1, \pi_2, \dots, \pi_n \in l},$ with $|\pi_i|=t_i,$ such that for any $\epsilon \in \sqrt{|l^{\times}|}$, there exist unique $p_1, p_2, \dots, p_n \in \mathbb{Q},$ for which $\epsilon=\prod_{i=1}^n |\pi_i|^{p_i}.$ Let us fix such elements $\pi_1, \pi_2, \cdots, \pi_n.$

\begin{defn}
Let $\epsilon \in M.$ Suppose $\epsilon= \prod_{i=1}^n t_i^{\frac{s_i}{r_i}},$ for $\frac{s_i}{r_i} \in \mathbb{Q}$ with $s_i, r_i$ coprime, $i=1,2,\dots, n.$
\begin{enumerate}
\item Let $r$ be the least common multiple of $r_i, i=1,2,\dots, n.$ We will say $r$ is \textit{the order of} $\epsilon.$
\item Let $\frac{s_i}{r_i}=\frac{s_i'}{r},$ $i=1,2,\dots,n.$ If there exists $i_0,$ such that $s_{i_0}'=1,$  then $t_{i_0}$ will be said to be \textit{a base of} $\epsilon.$
\end{enumerate} 
\end{defn}

Let $\epsilon \in M,$ and suppose $\epsilon=\prod_{i=1}^n t_i^{p_i},$ for $p_i \in \mathbb{Q},$ ${i=1,2, \dots, n.}$ Let $\alpha$ be the order of $\epsilon.$ 

\begin{lm}\label{70}
If $\alpha$ is odd, then for any $i=1,2,\dots, n,$ there exist $ \delta_i \in \{0,1\},$ such that $\epsilon \equiv \prod_{i=1}^n t_i^{\delta_i} \ \textrm{mod} \  M^2.$
\end{lm}

\begin{proof}
Remark that since $\alpha$ is odd, $\epsilon \equiv \epsilon^{\alpha} \ \textrm{mod} \  M^2,$ and $\epsilon^{\alpha}=\prod_{i=1}^n t_i^{s_i},$ with $s_i \in \mathbb{Z}$ for all ~$i.$ Let $s_i=2s_i'+\delta_i,$ where $s_i' \in \mathbb{Z}$ and $\delta_i \in \{0,1\}.$ Then, $\epsilon \equiv \prod_{i=1}^n t_i^{\delta_i} \ \textrm{mod} \  M^2.$ 
\end{proof}

\begin{lm}\label{71}
If $\alpha$ is even, then there exist $m \in M,$ $x_i, y \in \mathbb{Z},$  $i=1,2,\dots,n,$ with $y>0,$ satisfying:
\begin{enumerate} 
\item{$\epsilon \equiv  m \ \textrm{mod} \  M^2$;} 
\item{
 $m=\prod_{i=1}^n t_i^{x_i/2^y};$}
\item{there exists $i_0 \in \{1,2, \dots, n\},$ such that $x_{i_0}=1$.}
\end{enumerate}
\end{lm}
Remark that $t_{i_0}$ is a base of $m$ and its order is $2^y.$
\begin{proof}
Let $\alpha=2^y \cdot z,$ with $z$ odd and $y>0$. Then, $\epsilon \equiv \epsilon^{z} \ \textrm{mod} \  M^2,$ and $(\epsilon^z)^{2^y}=\prod_{i=1}^n t_i^{e_i},$ with $e_i \in \mathbb{Z}, i=1,2, \dots, n$. Furthermore, there exists $i_0 \in \{1,2,\dots, n\},$ such that $e_{i_0}$ is odd. 

Seeing as $(2^y, e_{i_0})=1,$ there exist  $A,B \in \mathbb{Z},$ with $A$ odd, such that $Ae_{i_0}+2^y B=1.$ Then, $\epsilon^z \equiv \epsilon^z \cdot (\epsilon^z)^{A-1} \ \textrm{mod} \  M^2,$ and $\epsilon^{zA}=t_{i_0}^{1/2^y-B} \cdot$  $\prod_{i\neq i_0} t_i^{Ae_i/2^y}.$ Hence, there exists $m'_B \in M,$ such that $\epsilon^{zA} \equiv m'_B \ \textrm{mod} \  M^2$, and
\begin{itemize}
\item{$m'_B=t_{i_0}^{1/2^y}\prod_{i\neq i_0} t_i^{Ae_i/2^y}$ if $B$ is even;}
\item{$m_B'=t_{i_0}^{1/2^y+1}\prod_{i\neq i_0} t_i^{Ae_i/2^y}$ if $B$ is odd.}
\end{itemize}
If $B$ is odd, $m_B'':=m_B' \cdot 
{m'_B}^{2^y} t_{i_0}^{-2-2^y} \equiv m_B' \ \textrm{mod} \  M^2,$ and $m_B''=t_{i_0}^{1/2^y} \prod_{i\neq i_0} t_i^{\frac{Ae_i}{2^y}(2^y+1)}.$  
 
Consequently, in either case, there exist $ m\in M$ and $x_i \in \mathbb{Z},$ for $i=1,2,\dots n,$ with $x_{i_0}=1,$  such that  $\epsilon \equiv m \ \textrm{mod} \  M^2,$ and $m=\prod_{i=1}^n t_i^{x_i/2^y}.$
\end{proof}

For $\varepsilon \in L,$ such that $|\varepsilon| \in \sqrt{|l^\times|},$ we will say that the order of $|\varepsilon|$ is \textit{the order of} $\varepsilon.$ If $|\pi_{i_0}|$ is a base of $|\varepsilon|,$ we will say $\pi_{i_0}$ is \textit{a base of} $\varepsilon.$ By applying the last two lemmas to the valued field $L,$ we obtain:

\begin{cor}\label{grande} Let $\varepsilon \in L^{\times}.$ Suppose $|\varepsilon|=\prod_{i=1}^n |\pi_i|^{p_i}$ for $p_i \in \mathbb{Q}, i=1,2,\dots, n.$ 
\begin{enumerate}
\item If the order of $|\varepsilon|$ is odd, then for any $i=1,2,\dots, n,$ there exists $\delta_i \in \{0,1\},$ such that $\varepsilon \equiv \prod_{i=1}^n \pi_i^{\delta_i} \ \textrm{mod} \  (L^{\times})^2 (L^o)^{\times}.$
\item If the order of $|\varepsilon|$ is even, then there exist $\varepsilon' \in L^{\times},$ $x_i, y \in \mathbb{Z},$  $i=1,2,\dots,n,$ with $y>0,$ satisfying:
\begin{enumerate} 
\item{$\varepsilon \equiv \varepsilon' \ \textrm{mod} \  (L^{\times})^2 (L^o)^{\times}$;} 
\item{
 $|\varepsilon'|=\prod_{i=1}^n|\pi_i|^{x_i/2^y};$}
\item{there exists $i_0 \in \{1,2, \dots, n\},$ such that $x_{i_0}=1$.}
\end{enumerate}
\end{enumerate}
\end{cor}

We immediately obtain as a by-product of the proof:

\begin{cor}\label{75}
Let $\varepsilon \in L^{\times},$ such that $|\varepsilon| \in \sqrt{|l^\times|}.$ Suppose the order of $|\varepsilon|$ is $2^\nu,$ so that there exist $\nu_i \in \mathbb{Z}, i=1,2,\dots, n,$ such that $|\varepsilon|=\prod_{i=1}^n |\pi_i|^{\nu_i/2^\nu}.$ If $\nu_{i'}$ is odd for some $i'$, then there exists $\varepsilon' \in L^\times,$ such that $\varepsilon \equiv \varepsilon' \ \textrm{mod} \  (L^{\times})^2 (L^{\circ})^{\times},$ and $|\pi_{i'}|$ is a base of $|\varepsilon'|.$
\end{cor}

Let $q_1$ (resp. $q_2$) be the part of $q$ whose coefficients have odd (resp. even) order. We remark that $q_1, q_2$ are diagonal quadratic forms over $L,$ and that $q=q_1 \bot q_2.$

\textit{Decomposition of $q_1$:} Set $\mathcal{A}=\left\lbrace \prod_{i=1}^n 
\pi_i^{\delta_i} | \delta_i \in \{0,1\}  
\right\rbrace.$
Let $e$ be any coefficient of $q_1.$ By Corollary \ref{grande} (1), there exist $u_e \in (L^{\circ})^{\times}$ and $A_e \in \mathcal{A},$ such that $e \equiv u_e \cdot A_e \ \textrm{mod} \ (L^{\times})^2.$  Consequently, for any $A \in \mathcal{A},$ there exists a diagonal quadratic form $\sigma_A$ with coefficients in $(L^{\circ})^\times,$ such that $q_1$ is $L$-isometric to $\bot_{A \in \mathcal{A}}A \cdot \sigma_A.$

\textit{Decomposition of $q_2$:} We first need an auxiliary result, which requires the following:
\begin{defn}
Let $\varepsilon \in L^{\times}$ be such that there exist $p_i \in \mathbb{Q}$, $i=1,2,\dots, n,$ for which $|\varepsilon|=\prod_{i=1}^n |\pi_i|^{p_i}.$ Let $I \subseteq \{0,1 \dots, n\},$  such that $\{i : p_i \neq 0\} \subseteq I.$ We will say that $\varepsilon$ is \textit{given in $|I|$ parameters}, where $|I|$ is the cardinality of $I,$ or that $\varepsilon$ is \textit{given in parameters over $I.$} 
\end{defn}
Notice that $a \in L$ is given in $0$ parameters if and only if $a \in (L^{\circ})^{\times}.$

\begin{lm}\label{72}
Let $\tau$ be a diagonal quadratic form over $L$ with coefficients of order either $1$ or an even number.  Let $I \subseteq \{1,2,\dots, n\},$ with $1 \leqslant |I|=m \leqslant n,$ such that the coefficients of $\tau$ are given in parameters over $I$. Then, there exist:
\begin{itemize}
\item $J \subseteq I,$ with $|J|=m-1,$
\item $x_1, x_2 \in L^{\times},$
\item diagonal quadratic forms $\tau_1, \tau_2$ over $L$ with coefficients of order either $1$ or an even number and in parameters over $J,$
\end{itemize}
such that $\tau$ is $L$-isometric to $x_1 \tau_1 \bot x_2  \tau_2.$  
\end{lm}

\begin{proof}
Roughly, the idea is to find some $i_0$ and a partition $A_j, j=1,2,$ of the set of coefficients, for which there exist $x_j \in L^\times,$ satisfying: if $a \in A_j,$ there exists $B_a \in L^\times,$ such that, modulo squares, $a=x_j \cdot B_a,$ and $|B_a|=\prod_{i \neq i_0}|\pi_i|^{p_{i,a}}, p_{i,a} \in \mathbb{Q}.$ In what follows, we find suitable representatives of the coefficients modulo squares, from which we can read the factorization $x_j \cdot B_a.$

Without loss of generality, let us assume that $I=\{1,2,\dots, m\}.$ Suppose the coefficients of $\tau$ are all of order $1.$ If they are given in zero parameters, the statement is clear. Otherwise, suppose that there is a coefficient given over a set of parameters containing $t_1.$ 

Let $d$ be any coefficient of the quadratic form. There exist $s_{i} \in \mathbb{Z}, i=1,2,\dots, n,$ such that $|d|=\prod_{i=1}^n |\pi_i|^{s_{i}}.$ 
 As a consequence, there exist $d' \in L^{\times}$ and $s_i' \in \mathbb{Z}, i=2,\dots, n,$ for which $d \equiv d' \ \textrm{mod} \ (L^\times)^2(L^{\circ})^{\times},$ and either $|d'|=\prod_{i=2}^n |\pi_i|^{s_i'}$ or $|d'|=|\pi_1| \cdot \prod_{i=2}^n |\pi_i|^{s_i'}.$ Hence, there exist diagonal quadratic forms $\tau_1, \tau_2,$ whose coefficients are all of order $1,$ in parameters over $\{2,3,\dots, m\},$ such that $\tau$ is $L$-isometric to $\pi_1 \tau_1 \bot \tau_2.$

Suppose there exists at least one coefficient of $\tau$ of even order. Let $\tau'$ be the quadratic form obtained from $\tau$ by:
\begin{enumerate}
\item{leaving the coefficients of order $1$ intact;}
\item{applying Corollary \ref{grande} (2) to the coefficients of even order to substitute them by elements of $L^{\times}$ that satisfy properties 2 and 3 of the lemma.}
\end{enumerate}
We remark that due to the proof of Corollary \ref{grande} (2) (\textit{i.e.} Lemma \ref{71}), the set of parameters over which the coefficients of $\tau'$ are given doesn't change. The quadratic form $\tau'$ is $L$-isometric to $\tau.$ Let us fix $a',$ one of the coefficients of $\tau'$ with largest order. Suppose the order of $a'$ is $2^{\alpha'}.$ Without loss of generality, we may assume that $\pi_1$ is a base of $a'.$ For $i=2,\dots, m,$ let $\alpha_i \in \mathbb{Z}$ be such that $|a'|=|\pi_1|^{1/2^{\alpha'}} \cdot \prod_{i=2}^m |\pi_i|^{\alpha_i/2^{\alpha'}}.$

Let $c$ be any other coefficient of $\tau'.$  Let $\pi_{i_0}$ be a base of $c,$ and $2^\gamma, \gamma \geqslant 0,$ its order. For  $i=1,2,\dots,m,$ let $\gamma_i \in \mathbb{Z}$ be such that $|c|=\prod_{i=1}^m |\pi_i|^{\gamma_i/2^{\gamma}}.$
\begin{sloppypar}
\begin{itemize}
\item{Suppose  $\alpha'>\gamma.$ Set $c'=c\cdot {a'}^{(2^\gamma-\gamma_1)\cdot 2^{\alpha'-\gamma}}$ Then,  $c' \equiv c \ \textrm{mod} \  (L^{\times})^2 (L^{\circ})^{\times},$ and $|c'|=|\pi_1| \cdot \prod_{i=2}^m |\pi_i|^{\frac{\gamma_i+\alpha_i(2^{\gamma}-\gamma_1)}{2^\gamma}}.$}
\item{Suppose $\alpha'=\gamma$ and $\gamma_1$ is odd. By Corollary \ref{75}, there exist $\alpha_i' \in \mathbb{Z},i=2,3,\dots, n,$ and $c'' \in L^\times$ of order $2^{\alpha'},$ having $\pi_1$ as a base, such that $c'' \equiv c \ \textrm{mod} \ (L^{\times})^2 (L^{\circ})^{\times}$ and $|c''|=|\pi_1|^{1/2^{\alpha'}}\cdot \prod_{i=2}^m |\pi_i|^{\alpha_i'/2^{\alpha'}}.$}
\item{Suppose $\alpha'=\gamma$ and $\gamma_1$ is even. Let $\gamma_1'/2^{\delta}$ be the reduced form of $\gamma_1/2^{\gamma},$ meaning $\gamma_1'$ is odd. Set ${c'''=c\cdot {a'}^{(2^\delta-\gamma_1')\cdot 2^{\alpha'-\delta}}}.$ Then, ${c''' \equiv c \ \textrm{mod} \  (L^{\times})^2(L^{\circ})^\times},$ and ${|c'''|=|\pi_1| \cdot \prod_{i=2}^m |\pi_i|^{\frac{\gamma_i+\alpha_i(2^{\gamma}-\gamma_1)}{2^\gamma}}}.$} 
\end{itemize}
To summarize, there exist $\bar{c} \in L^{\times}$ and $\epsilon_2, \cdots, \epsilon_m \in \mathbb{Z},$ such that $c \equiv \bar{c}$ $\ \textrm{mod} \ (L^{\times})^2 (L^{\circ})^{\times},$ and either $|\bar{c}|=$ $|\pi_1|^{1/2^{\alpha}} \cdot \prod_{i=2}^m |\pi_i|^{\epsilon_i/2^{\alpha'}}=$ $|a'| \cdot \prod_{i=2}^m |\pi_i|^{\frac{\epsilon_i-\alpha_i}{2^{\alpha'}}}$ or $|\bar{c}|=|\pi_1| \cdot \prod_{i=2}^m |\pi_i|^{\epsilon_i/2^{\alpha'}}.$
\end{sloppypar}
Therefore, there exist diagonal quadratic forms $\tau_1, \tau_2$ over $L,$ such that $\tau \cong \pi_1 \tau_1 \bot a' \tau_2,$ and for any coefficient $h$ of $\tau_1$ or $\tau_2,$ the order of $h$ is either $1$ or an even integer. Furthermore, $h$ is with parameters over $\{2,3, \dots, m\}.$  
\end{proof}

Using induction, an immediate consequence of Lemma \ref{72} is that there exists a family $T$ of $2^n$ quadratic forms with coefficients in $(L^{\circ})^\times,$ such that $\tau$ is $L$-isometric to $\bot_{\sigma \in T} B_{\sigma} \cdot \sigma,$ where $B_{\sigma} \in L^\times$ for any $\sigma \in T.$

Finally, by combining the decomposition results of $q_1$ and $q_2$, we obtain the statement of Proposition \ref{18}.
\end{proof}

The following framework corresponds to Berkovich curves:

\begin{set} Let $k$ be a complete ultrametric field.
Let $k \subseteq R$ be a Henselian valuation ring with maximal ideal $m_R,$ and fraction field $F_R=\mathrm{Frac} \ R.$ Set $L'=R/m_R$, and suppose it is endowed with a valuation making it a Henselian (called \textit{quasicomplete} in ~\cite{ber93}) valued field extension of $k.$ Let $L/L'$ be an immediate Henselian extension. Set $t=\mathrm{rank}_{\mathbb{Q}}(|L^\times| / |k^{\times}| \otimes_{\mathbb{Z}} \mathbb{Q})=\mathrm{rank}_{\mathbb{Q}}(|L'^\times| / |k^{\times}| \otimes_{\mathbb{Z}} \mathbb{Q})$ and ${s=\deg  \mathrm{tr}_{\widetilde{k}}{\widetilde{L}}=\deg  \mathrm{tr}_{\widetilde{k}}{\widetilde{L'}}}.$ 
Suppose $s+t \leqslant 1.$
\end{set}

The motivation behind this setup is:

\begin{ex}
Let $C$ be any $k$-analytic curve, and $x \in C$ any point. The hypotheses of the setting above are satisfied for $R=\mathcal{O}_x,$ $F_R=\mathscr{M}_x,$ $L'=\kappa(x),$ and $L=\mathcal{H}(x).$
\end{ex}

For any quadratic form $\sigma$ with coefficients in $R,$ let us denote by $\sigma_L$ (resp. $\sigma_{L'}$) its image over $L$ (resp. $L'$). 

We recall: 
\begin{defn}
Let $K$ be a field. 
\begin{enumerate}
\begin{sloppypar}
\item {[Kaplansky] The $u$-\textit{invariant} of $K,$ denoted by $u(K),$ is the maximal dimension of anisotropic quadratic forms over $K.$ We say that $u(K)=\infty$ if there exist anisotropic quadratic forms over $K$ of arbitrarily large dimension.}
\item {[HHK] The strong $u$-invariant of $K,$ denoted by $u_s(K),$ is the smallest real number~$m$, such that:
\begin{itemize}
\item $u(E) \leqslant m$ for all finite field extensions $E/K;$ 
\item $\frac{1}{2}u(E) \leqslant m$ for all finitely generated field extensions $E/K$ of transcendence degree 1.  
\end{itemize}
We say that $u_s(K)=\infty$ if there exist such field extensions $E$ of arbitrarily large~\mbox{$u$-invariant.}}
\end{sloppypar}
\end{enumerate}
\end{defn}

\begin{nota}
From now on, let $k$ be a complete ultrametric field, such that $\dim_{\mathbb{Q}} \sqrt{|k^\times|}$ equals an integer $n.$ Also, suppose $\text{char} \ \widetilde{k} \neq 2.$
\end{nota}

\begin{prop}\label{50}
Let $L/k$ be a valued field extension, such that $\mathrm{rank}_{\mathbb{Q}}(|L^\times| / |k^{\times}|$ $ {\otimes_{\mathbb{Z}} \mathbb{Q})=0}$ and $\deg \mathrm{tr}_{\widetilde{k}}{\widetilde{L}}=0.$ Let $\tau$ be a quadratic form over $L,$ with $\dim{\tau} > 2^{n+1}u_s(\widetilde{k}).$
\begin{enumerate}
\item Suppose $L$ is Henselian. Then, $\tau$ is isotropic.
\item Under the same hypotheses as in Setting 2, let $q$ be a diagonal quadratic form over~ $R,$ such that $q_L=\tau.$ Then, $q$ is isotropic over $F_R.$
\end{enumerate}
\end{prop}

\begin{proof} Since $\text{char}(L) \neq 2,$ we may assume that $\tau$ is a diagonal quadratic form.
Seeing as $\dim_{\mathbb{Q}} \sqrt{|L^{\times}|}=n,$ 
by Proposition \ref{18} there exists a set $Q$ of at most $2^{n+1}$ quadratic forms with coefficients in $(L^{\circ})^\times,$ such that $\tau$ is $L$-isometric to $\bot_{\sigma \in Q} C_\sigma \cdot \sigma,$ with $C_\sigma \in L^\times$ for every $\sigma \in Q.$

Since $\dim{\tau}>2^{n+1}u_s(\widetilde{k}),$ there exists $\tau' \in Q,$ such that $\dim{\tau'} > u_s(\widetilde{k}).$ Let $\widetilde{\tau'}$ be the image of $\tau'$ over $\widetilde{L}.$ Seeing as the coefficients of $\tau'$ are all in $(L^{\circ})^{\times},$ $\dim{\widetilde{\tau'}=\dim{\tau'}>u_s(\widetilde{k}}).$ Since $\deg  \mathrm{tr}_{\widetilde{k}} \widetilde{L}=0,$ the extension $\widetilde{L}/\widetilde{k}$ is algebraic. Let $E$ be the finite field extension of $\widetilde{k}$ generated by the coefficients of $\widetilde{\tau'}.$ Then, $u(E)\leqslant u_s(\widetilde{k})< \dim{\widetilde{\tau'}},$ implying $\widetilde{\tau'}$ is isotropic over $E,$ and hence over $\widetilde{L}.$ Since $L$ is Henselian, $\tau'$ is isotropic over $L,$ and thus so is $\tau.$

For the second part, if $\tau=q_L$ for some diagonal $R$-quadratic form $q,$ seeing as $\widetilde{\tau'}$ is isotropic over $\widetilde{L}=\widetilde{L'},$ the image of $q$ in $\widetilde{L'}$ is so as well. From Henselianity of $L',$ we obtain that the image of $q$ in $L'$ is isotropic there. Finally, from Henselianity of $R,$ the quadratic form $q$ is isotropic over $F_R$. 
\end{proof}

The bound $2^{n+1} u_s(\widetilde{k})$ in Proposition \ref{50} will remain the same regardless of whether we demand $|k^{\times}|$ to be a free $\mathbb{Z}$-module or not. The reason behind this is that in any case, the hypotheses of said proposition tell us only that $\dim_{\mathbb{Q}}\sqrt{|L^\times|}=n,$ but not necessarily that   $|L^\times|$ is a free $\mathbb{Z}$-module.

\begin{prop}
Let $L/k$ be a valued field extension, such that $\mathrm{rank}_{\mathbb{Q}}(|L^\times| / |k^{\times}|$ $ {\otimes_{\mathbb{Z}} \mathbb{Q})=0}$ and 
$\deg \mathrm{tr}_{\widetilde{k}}{\widetilde{L}}=1.$ Let $\tau$ be a quadratic form over $L,$ with $\dim{\tau} > 2^{n+2}u_s(\widetilde{k}).$
\begin{enumerate}
\item Suppose $L$ is Henselian. Then, $\tau$ is isotropic.
\item Under the same hypotheses as in Setting 2, let $q$ be a diagonal quadratic form over ~$R,$ such that $q_L=\tau.$ Then, $q$ is isotropic over $F_R.$
\end{enumerate}
If $|L^{\times}|$ is a free $\mathbb{Z}$-modules of dimension $n$, the statement is true for $\dim{\tau}>2^{n+1}u_s(\widetilde{k}).$

\end{prop}

\begin{proof} Since $\text{char}(L) \neq 2,$ we may assume that $\tau$ is a diagonal quadratic form.
Again, let $\bot_{\sigma \in Q} C_\sigma \sigma$ be the  $L$-quadratic form isometric to $\tau$ obtained from Proposition \ref{18} (resp. Proposition \ref{19}), where $Q$ has cardinality at most $2^{n+1}$ (resp. $2^n$). Then, there exists $\tau' \in Q,$ such that $\dim{\tau'}> 2u_s(\widetilde{k}).$ 
Let $\widetilde{\tau'}$ be the image of $\tau'$ over $\widetilde{L}.$ Since the coefficients of $\tau'$ are all in $(L^{o})^{\times},$ $\dim{\widetilde{\tau'}}=\dim{\tau'}>2u_s(\widetilde{k}).$

As the extension $\widetilde{L}/\widetilde{k}$ is finitely generated of transcendence degree $1,$ one obtains ${u(\widetilde{L}) \leqslant 2u_s(\widetilde{k})<\dim{\tau'}}.$ This implies that $\tau'$ is isotropic over $\widetilde{L}.$ Since $L$ is Henselian, the quadratic form $\tau'$ is isotropic over $L,$ and thus so is $\tau.$

For the second part, if $\tau=q_L$ for some diagonal quadratic form $q$ over $R,$ we conclude by using the same argument as in Proposition \ref{50}, seeing as $\widetilde{\tau'}$ is isotropic over $\widetilde{L'}.$
\end{proof}

\begin{prop}
Let $L/k$ be a valued field extension, such that $\mathrm{rank}_{\mathbb{Q}}(|L^\times| / |k^{\times}|$ $ {\otimes_{\mathbb{Z}} \mathbb{Q})=1}$ and $\deg \mathrm{tr}_{\widetilde{k}}{\widetilde{L}}=0.$ Let $\tau$ be a  quadratic form over $L,$ with $\dim{\tau} > 2^{n+2}u_s(\widetilde{k}).$
\begin{enumerate}
\item Suppose $L$ is Henselian. Then, $\tau$ is isotropic.
\item Under the same hypotheses as in Setting 2, let $q$ be a diagonal quadratic form over~ $R,$ such that $q_L=\tau.$ Then, $q$ is isotropic over $F_R.$
\end{enumerate}
If $|k^{\times}|$ is a free $\mathbb{Z}$-module, the statement is true for $\dim{\tau}>2^{n+1}u_s(\widetilde{k}).$
\end{prop}

\begin{proof} Since $\text{char}(L) \neq 2,$ we may assume that $\tau$ is a diagonal quadratic form.
Since $\mathrm{rank}_{\mathbb{Q}}(|L^\times| / |k^{\times}| \otimes_{\mathbb{Z}} \mathbb{Q})=1,$ there exists $\rho \in \mathbb{R}_{>0} \backslash \sqrt{|k^{\times}|},$ such that the group $|L^{\times}|$ is generated by $|k^{\times}|$ and $\rho.$ Let $T$ be an element of $L$ with $|T|=\rho.$ Then, for any $a \in L^{\times},$ there exist $m \in \mathbb{Z},$ $p_i \in \mathbb{Q}$ (resp. $p_i \in \mathbb{Z}$), $i=1,2,\dots, n,$  such that $|a|=|T|^m \cdot \prod_{i=1}^n|\pi_i|^{p_i}.$ Consequently, there exist diagonal quadratic forms $q_1, q_2$ over $L,$ for which $\tau$ is isometric to $q_1 \bot Tq_2,$ where the coefficients of $q_1, q_2$ have norms in $|k^{\times}|.$ 

By applying Proposition \ref{18} (resp. Proposition \ref{19}) to $q_1$ and $q_2$, we obtain a family $S$ of at most $2^{n+2}$ (resp. $2^{n+1}$) diagonal quadratic forms with coefficients in $(L^{\circ})^\times,$ such that $\tau$ is isometric to $\bot_{\sigma \in S} C_{\sigma} \cdot \sigma,$ where $C_{\sigma} \in L^\times$ for every $\sigma \in S.$ Thus, there exists $\tau' \in S,$ such that $\dim{\tau'} > u_s(\widetilde{k}).$ Let $\widetilde{\tau'}$ be the image of $\tau'$ in $\widetilde{L}.$ Seeing as the coefficients of $\tau'$ are all 
in $(L^{\circ})^{\times},$ $\dim{\widetilde{\tau'}}=
\dim{\tau'}>u_s(\widetilde{k}).$ 

The extension $\widetilde{L}/
\widetilde{k}$ is finite algebraic, so 
$u(\widetilde{L})\leqslant u_s(\widetilde{k})< 
\dim{\widetilde{\tau'}},$ implying $
\widetilde{\tau'}$ is isotropic over $
\widetilde{L}.$ Since $L$ is Henselian, $\tau'
$ is isotropic over $L,$ and thus so is $\tau.$ 

For the second part, if $\tau=q_L$ for some $q,$ as $\widetilde{\tau'}$ is isotropic over $\widetilde{L'},$ we conclude as in Proposition \ref{50}.
\end{proof}

Keeping the same notation, the three propositions above can be summarized into:

\begin{thm}\label{52}
\begin{sloppypar}
Let $L/k$ be a valued field extension. Suppose that the inequality ${\mathrm{rank}_{\mathbb{Q}}(|L^\times| / |k^{\times}| \otimes_{\mathbb{Z}} \mathbb{Q})+\deg \mathrm{tr}_{\widetilde{k}}{\widetilde{L}} \leqslant 1}$ holds. Let $\tau$ be a quadratic form over~ $L,$ with $\dim{\tau} > 2^{n+2}u_s(\widetilde{k}).$
\end{sloppypar}
\begin{enumerate}
\item Suppose $L$ is Henselian. Then, $\tau$ is isotropic.
\item Under the same hypotheses as in Setting 2, let $q$ be a diagonal quadratic form over~ $R,$ such that $q_L=\tau.$ Then, $q$ is isotropic over $F_R.$
\end{enumerate}
If $|k^{\times}|$ is a free $\mathbb{Z}$-module, and $|L^\times|$ is a free $\mathbb{Z}$-module with $\mathrm{rank}_{\mathbb{Z}} |L^\times|=n$ if ${\deg \mathrm{tr}_{\widetilde{k}}{\widetilde{L}}=1}$ and
$\mathrm{rank}_{\mathbb{Q}}(|L^\times| / |k^{\times}| \otimes_{\mathbb{Z}} \mathbb{Q})=0,$  then the statement is true for ${\dim{\tau}>2^{n+1}u_s(\widetilde{k})}.$
\end{thm}

A result we will be using often in what follows:

\begin{lm}\label{24} Suppose $|k^\times|$ is a free $\mathbb{Z}$-module of dimension $n$. Let $k'/k$ be a valued field extension, such that $|{k'}^\times|$ is finitely generated over $|k^\times|,$ and $|{k'}^\times|/|k^\times|$ is a torsion group. Then, $|{k'}^\times|$ is also a free $\mathbb{Z}$-module of dimension $n.$ 

Suppose $k'/k$ is a finite field extension.
Let $\tau$ be a diagonal quadratic form over ~$k'$ with $\dim{\tau}>2^nu_s(\widetilde{k}).$ Then, $q$ is $k'$-isotropic.
\end{lm}

\begin{proof}
Seeing as $|{k'}^{\times}|/|k^{\times}|$ is a torsion group, its rank as a $\mathbb{Z}$ module is $0.$ Considering $\mathrm{rank}_{\mathbb{Z}} |{k'}^{\times}|= \mathrm{rank}_{\mathbb{Z}}|{k'}^{\times}|/|k^{\times}| + \mathrm{rank}_{\mathbb{Z}}|k^{\times}|,$ we obtain $\mathrm{rank}_{\mathbb{Z}}|{k'}^{\times}|=n.$ Furthermore, being a finitely generated torsion-free module over $\mathbb{Z},$ it is free. 

Let $\bot_{\sigma \in Q} C_{\sigma} \cdot \sigma$ be the quadratic form $k'$-isometric to $\tau$ obtained by applying Proposition ~\ref{19}. There exists $\sigma_0 \in Q$ with coefficients in $({k'}^{\circ})^{\times},$ such that $\dim{\widetilde{\sigma_0}}=\dim{\sigma_0}>u_s(\widetilde{k})$, where $\widetilde{\sigma_0}$ is the image of $\sigma_0$ over $\widetilde{k'}.$ Suppose $k'/k$ is a finite field extension. Seeing as then $\widetilde{k'}/\widetilde{k}$ is also finite, $\widetilde{\sigma_0}$ is $\widetilde{k'}$-isotropic. From Henselianity of $k',$ we obtain that $\sigma_0$ is $k'$-isotropic, thus so is $\tau.$
\end{proof}

The following shows that if $|k^\times|$ is a free finitely generated $\mathbb{Z}$-module of dimension $n,$ the last conditions of Theorem \ref{52} are satisfied in the Berkovich setting. 

\begin{cor}\label{56} 
\begin{sloppypar}
Suppose $|k^\times|$ is a free $\mathbb{Z}$-module with $\mathrm{rank}_{\mathbb{Z}}|k^\times|=n.$ Let $C$ be a $k$-analytic curve. 
If $x \in C$ is a type 2 point, then $|\mathcal{H}(x)^\times|$ is a free $\mathbb{Z}$-module and ${\mathrm{rank}_{\mathbb{Z}}(|
\mathcal{H}(x)^\times|)=n}.$
\end{sloppypar}
\end{cor}

\begin{proof}
Since $x$ is an Abhyankar point, $|\mathcal{H}(x)^\times|$ is finitely generated over $|k^\times|,$ and since it is of type 2, $|\mathcal{H}(x)^\times|/|k^\times|$ is a torsion group, so this follows from Lemma \ref{24}.
\end{proof}

Another result we will be needing in what is to come:  

\begin{lm}\label{55} Under the same hypotheses as in Setting 2, suppose $R$ is a discrete valuation ring. Let $q$ be a diagonal quadratic form over $F_R.$
Then, there exist diagonal $F_R$-quadratic forms $q_1, q_2$  with coefficients in $R,$ and $a \in F_R^\times,$ such that:
\begin{itemize}
\item  $q$ is isometric to $q_1 \bot a q_2;$
\item  $q_{i,L}$ has coefficients in $(L^{\circ})^\times, i=1,2;$
\item there exists $i_0 \in \{1,2\},$ such that $\dim{q_{i_0,L}} \geqslant \frac{1}{2}\dim{q}.$
\end{itemize}
In particular, if either of $q_1, q_2$ is isotropic over $F_R,$ then so is $q.$ 
\end{lm}

\begin{proof}
Let $\pi$ be a uniformizer of $R.$ For any coefficient $b$ of $q,$ either $b \equiv 1 \ \textrm{mod} \ (F_R^\times)^2 (F_R^{\circ})^{\times}$ or $b \equiv \pi \ \textrm{mod} \ (F_R^\times)^2 (F_R^{\circ})^{\times}.$ Hence, there exist diagonal $F_R$-quadratic forms $q_1, q_2$ with coefficients in $(F_R^{\circ})^\times=R^\times,$ such that $q$ is $F_R$-isometric to $q'=q_1 \bot \pi q_2.$ Then, ${\dim{q}=\dim{q'}},$ and there exists $i_0,$ such that $\dim{q_{i_0}} \geqslant \frac{1}{2}\dim{q}.$ Since the coefficients of $q_1, q_2$ are in $R^{\times},$ their images over $L$ are of same dimension, so $\dim{q_{i_0, L}} \geqslant \frac{1}{2}\dim{q}.$ Finally, the last sentence of the statement is obvious. 
\end{proof}

The following theorem gives the motivation behind the hypotheses we put upon $R, L'$ and $L.$
\begin{thm}\label{21}
Suppose $char(\widetilde{k})\neq 2.$ Let $C$ be a normal irreducible $k$-analytic curve. Set ${F=\mathscr{M}(C)}.$ Let $q$ be a quadratic form over $F$ of dimension $d,$ with $d>2^{n+2}u_s(\widetilde{k}).$ Then, for any $x \in C,$ the quadratic form $q$ is isotropic over $\mathscr{M}_x$ for all $x \in C.$ 

If $|k^{\times}|$ is a free $\mathbb{Z}$-module, the statement is true for $d>2^{n+1}u_s(\widetilde{k}).$
\end{thm}

\begin{proof}
Seeing as $\mathrm{char}(\widetilde{k})\neq 2,$ neither of the overfields of $k$ has characteristic $2.$ In particular, $\mathrm{char}(F) \neq 2,$ so there exists a diagonal quadratic form $q'$ over $F$ isometric to $q.$ By replacing $q$ with $q'$ if necessary, we may directly assume that $q$ is a diagonal quadratic form.

Recall that $\mathcal{O}_x$ and $\kappa(x)$ are Henselian \cite[Sections 2.1 and 2.3]{ber93}. Furthermore, $\mathcal{H}(x)$ is the completion of $\kappa(x),$ so it is a Henselian immediate extension. We know that for any $x \in C,$ the field $\mathcal{H}(x)$ is either a finite extension of $k$ or a completion of $F$ with respect to some valuation extending that of $k.$ Abhyankar's inequality tells us that ${\mathrm{rank}_{\mathbb{Q}}(|\mathcal{H}(x)^\times|/|k^\times|\otimes_{\mathbb{Z}} \mathbb{Q})+\deg \mathrm{tr}_{\widetilde{k}} \widetilde{\mathcal{H}(x)} \leqslant 1}.$ We will apply part 2 of Theorem \ref{52} by taking $R=\mathcal{O}_x,$ $F_R=\mathscr{M}_x,$ $L'=\kappa(x),$ and $L=\mathcal{H}(x).$

If $\mathcal{H}(x)/k$ is finite, \textit{i.e.} if $x$ is a rigid point, then $\mathcal{H}(x)=\kappa(x)=\mathcal{O}_x/m_x.$ Being a normal Noetherian local ring with Krull dimension one, $\mathcal{O}_x$ is a discrete valuation ring. By Lemma \ref{55}, there exists a diagonal $\mathscr{M}_x$-quadratic form $\tau$ with coefficients in $\mathcal{O}_x,$ such that $\dim{\tau_L} \geqslant \frac{1}{2}\dim{q}>2^{n+1}u_s(\widetilde{k})$ (resp. $\dim{\tau_L} \geqslant \frac{1}{2}\dim{q}>2^{n}u_s(\widetilde{k})$) and the isotropy of $\tau$ implies that of ~$q.$ Seeing as $\mathrm{rank}_{\mathbb{Q}}(|\mathcal{H}(x)^\times|/|k^{\times}| \otimes_{\mathbb{Z}} \mathbb{Q})=\deg \mathrm{tr}_{\widetilde{k}} \widetilde{\mathcal{H}(x)}=0,$ we can apply Proposition \ref{50} (resp. Lemma \ref{24}) to $\tau$.  

Otherwise, $\mathcal{O}_x=\kappa(x)$ is a field, and $\mathcal{H}(x)$ is its completion. In the general case, we conclude by a direct application of Theorem \ref{52}. In particular, if $|k^\times|$ is a free $\mathbb{Z}$-module, then this is an application of Theorem \ref{52} in view of Corollary \ref{56}.
\end{proof}

We also obtain:

\begin{cor}\label{59}
Suppose $char(\widetilde{k}) \neq 2.$ Let $C$ be a normal irreducible $k$-analytic curve. Let $x$ be any point of $C.$ Let $q$ be a quadratic form over $\mathcal{H}(x),$ such that $\dim{q}>2^{n+2}u_s(\widetilde{k}).$ Then, $q$ is isotropic. 

If $|k^\times|$ is a free $\mathbb{Z}$-module, then the statement is true for $\dim{q}>2^{n+1}u_s(\widetilde{k}).$
\end{cor}

\begin{proof}
This is a direct consequence of part (1) of Theorem \ref{52} (in view of Corollary ~\ref{56} for the special case).
\end{proof}

\section{Applications}

We will now apply the results obtained in the previous section to the (strong) $u$-invariant. Let $k$ be a complete ultrametric field.
 
\begin{thm}\label{25}
Suppose $char(\widetilde{k}) \neq 2.$ Let $F$ be a finitely generated field extension of $k$ of transcendence degree $1.$ Let $q$ be a quadratic form over $F$ of dimension $d.$
\begin{enumerate}
\item If $\dim_{\mathbb{Q}}\sqrt{|k^{\times}|}=:n, n \in \mathbb{N},$ and $d> 2^{n+2} u_s(\widetilde{k}),$ then $q$ is isotropic.
\item If $|k^{\times}|$ is a free $\mathbb{Z}$-module with $\mathrm{rank}_{\mathbb{Z}}|k^{\times}|=:n, n \in \mathbb{N},$ and $d>2^{n+1}u_s(\widetilde{k}),$ then $q$ is isotropic. 
\end{enumerate}
\end{thm}

\begin{proof}
There exists a connected normal projective $k$-analytic curve $C$, such that ${F=\mathscr{M}(C)}.$ By Theorem \ref{17}, the quadratic form $q$ is isotropic over $F$ if and only if it is isotropic over ~$\mathscr{M}_x$ for all $x \in C.$ The statement now follows in view of Theorem \ref{21}.
\end{proof}

\begin{cor}\label{31} Suppose $char(\widetilde{k}) \neq 2.$ 
\begin{enumerate}
\item If $\dim_{\mathbb{Q}}\sqrt{|k^{\times}|}=:n, n \in \mathbb{N},$ then $u_s(k) \leqslant 2^{n+1}u_s(\widetilde{k}).$
\item If $|k^{\times}|$ is a free $\mathbb{Z}$-module with $\mathrm{rank}_{\mathbb{Z}}|k^{\times}|=:n, n \in \mathbb{N}$, then $u_s(k) \leqslant 2^{n}u_s(\widetilde{k}).$
\end{enumerate}
\end{cor}

\begin{proof}
Let $l/k$ be a finite field extension.  Let $q$ be an $l$-quadratic form of dimension $d>2^{n+1}u_s(\widetilde{k})$ (resp. $d>2^nu_s(\widetilde{k})$). Since $\mathrm{char}(\widetilde{k}) \neq 2,$ we may assume $q$ to be diagonal. In view of part 1 of Proposition \ref{50} (resp. Lemma \ref{24}), $q$ is $l$-isotropic, so  ${u(l) \leqslant 2^{n+1}u_s(\widetilde{k})}$ (resp. $u(l) \leqslant 2^n u_s(\widetilde{k})$).
In combination with Theorem \ref{25}, this completes the proof of the statement.
\end{proof}

\begin{cor}
Suppose $char(\widetilde{k}) \neq 2.$ Let $C$ be a normal irreducible $k$-analytic curve. Let $x$ be any point of $C.$
\begin{enumerate}
\item If $\dim_{\mathbb{Q}}\sqrt{|k^{\times}|}=:n, n \in \mathbb{N},$ then $u(\mathcal{H}(x))\leqslant 2^{n+2} u_s(\widetilde{k}).$
\item If $|k^{\times}|$ is a free $\mathbb{Z}$-module with $\mathrm{rank}_{\mathbb{Z}}|k^{\times}|=:n, n \in \mathbb{N}$, then $u(\mathcal{H}(x))\leqslant 2^{n+1} u_s(\widetilde{k}).$
\end{enumerate}
\end{cor}

\begin{proof}
See Corollary \ref{59}.
\end{proof}

In particular, when $k$ is discretely valued we obtain the upcoming corollary. It is the most important result on quadratic forms of HHK in \cite{HHK}, and from it we obtain that $u(\mathbb{Q}_p(T))=8$ when $p \neq 2,$ originally shown in \cite{parsur}.

\begin{cor}\label{32}
Let $k$ be a complete discretely valued field, such that $char(\widetilde{k})\neq 2.$ Then, $u_s(k)=2u_s(\widetilde{k}).$
\end{cor} 

\begin{proof}
The inequality $u_s(k) \leqslant 2 u_s(\widetilde{k})$ is a special case of Corollary 6.2. For the other direction, a proof that is independent of the patching method and relies on the theory of quadratic forms is given in \cite[Lemma 4.9]{HHK}.
\end{proof}

\bigskip

{\footnotesize%
 \textsc{Vler\"e Mehmeti}, Laboratoire de math\'ematiques Nicolas Oresme, Universit\'e de Caen Normandie, 14032 ~Caen Cedex, France \par
  \textit{E-mail address}: \texttt{vlere.mehmeti@unicaen.fr} 
}

\end{document}